\documentclass[a4paper]{article}

\usepackage{amsfonts}
\usepackage{graphicx}
\usepackage{amsthm}
\usepackage{amsmath}
\usepackage{amssymb}
\usepackage{enumerate}
\usepackage{dsfont}
\usepackage{abstract}
\usepackage{verbatim}
\usepackage{xcolor}

\usepackage{sagetex}

\allowdisplaybreaks

\newtheorem{proposition}{Proposition}[section]
\newtheorem{lemma}[proposition]{Lemma}
\newtheorem{corollary}[proposition]{Corollary}
\newtheorem{theorem}[proposition]{Theorem}

\newtheorem*{theorem*}{Theorem}
\theoremstyle{definition}

\newtheorem{remark}[proposition]{Remark}
\numberwithin{equation}{section}

\newenvironment{mysage}{\sagesilent}{\endsagesilent}

\makeatletter
\newcommand{\customlabel}[2]{%
	\protected@write \@auxout {}{\string \newlabel {#1}{{#2}{}}}}
\makeatother
\newenvironment{proofbold}[1]{\paragraph{Proof of {#1}.}}{\hfill$\square$}

\title{\begin{center}
\LARGE
\textbf{On the Atkinson formula for the $\zeta$ function}
\bigskip
\end{center}}
\author{Daniele Dona, Sebastian Zuniga Alterman}
\date{\today}

\begin{document}

\maketitle

\renewcommand{\thefootnote}{}
\renewcommand{\thefootnote}{\arabic{footnote}}
\setcounter{footnote}{0}

\begin{abstract}
Thanks to Littlewood (1922) and Ingham (1928), we know the first two terms of the asymptotic formula for the square mean integral value of the Riemann zeta function $\zeta$ on the critical line. Later, Atkinson (1939) presented this formula with an error term of order $O(\sqrt{T}\log^{2}(T))$, which we call the \textit{Atkinson formula}. Following the latter approach and the work of Titchmarsh (1986), we present an explicit version of the Atkinson formula, improving on a recent bound by Simoni\v{c} (2020). Moreover, we extend the Atkinson formula to the range $\Re(s)\in\left[\frac{1}{4},\frac{3}{4}\right]$, giving an explicit bound for the square mean integral value of $\zeta$ and improving on a bound by Helfgott and the authors (2019).
We use mostly classical tools, such as the approximate functional equation and the explicit convexity bounds of the zeta function given by Backlund (1918).

\end{abstract}

\normalsize

\newcommand{\Addresses}{{
  {\footnotesize
\ \\
  D.~Dona, \textsc{Einstein Institute of Mathematics, The Hebrew University of Jerusalem, Edmond J. Safra Campus Givat Ram, Jerusalem 9190401, Israel.}\\
  \texttt{daniele.dona@mail.huji.ac.il}

\ \\
  S.~Zuniga Alterman, \textsc{Wydzia\l\ Matematyki i Informatyki, Nicolaus Copernicus University, 12-18 Chopina, 87-100 Toru\'n, Poland.}\\
  \texttt{szalterman@mat.umk.pl}

}}}

\footnotesize
\begin{mysage}
### Constants ###
gamma_int = RIF( euler_gamma ) # Euler-Mascheroni constant
pi_int = RIF( pi ) # pi
e_int = RIF( e ) # e
### Rounding ###
def roundup(x,d):
    return round(ceil(x*10^d)/10^d,ndigits=d)
def rounddown(x,d):
    return round(floor(x*10^d)/10^d,ndigits=d)
### Redefining usual functions ###
# This will ensure that functions like exp interact well
# with the use of RIF(). Thanks are due to Marc Mezzarobba.
def exp_int(x):
    return (RIF(x)).exp()
def log_int(x):
    return (RIF(x)).log()
def sqrt_int(x):
    return exp_int(1/2*log_int(x))
def abs_int(x):
    return (RIF(x)).abs()
def gam_int(x):
    return (RIF(x)).gamma()
def zet_int(x):
    return (RIF(x)).zeta()
### More precise version of exp above ###
def findbits(x):
    y=53
    while y<x:
        y+=1
    return y
def exp_int_precise(x,d):
    RIFPr=RealIntervalField(findbits(d))
    return (RIFPr(x)).exp()
### Definition of T0 ###
T0 = 100
### Number of digits we will use ###
digits = 3
### Check_c ###
# Setting this value to True makes the Sage code
# check whether the assertions about the best value of c
# in the last section are true.
# Setting this value to False skips the check.
check_c = False
### Check_i ###
# Setting this value to True makes the Sage code
# check whether the assertion about the integral
# in {co:main} is true.
# Setting this value to False skips the check.
# Warning: if True is set, the computation takes about 5 minutes!
check_i = False
# Same applies for the check about the integral
# in {sec:numerical}.
# Warning: if True is set, the computation takes about 10 minutes!
check_ip = False
###
###
###
# Here follows an explanation of some of the notation used throughout
# the Sage code, hopefully improving the reader's understanding of the code.
# Due to the length of the code, the conventions below may not always have
# been followed: the authors may have changed their mind at some intermediate
# stage, or they may have thought that another notation was locally more clear,
# or they may have just been distracted. For this, we apologize.
###
# If "X" is the name of some parameter depending on variables chosen by us,
# such as T0, sigma or c, "Fu_X" is the name of the function that computes X.
###
# In intermediate results, for the names of the parameters we use a common
# letter for all those that appear in the same expression, followed by numbers
# arranging them by order of magnitude: such an expression could for instance
# be of the form A_1*T+A_2*sqrt(T)+A_3*log(T).
###
# When computing coefficients inside more general results, usually made by
# composing several intermediate results together, we define them as follows.
# The general convention is of the form X1y_X2_X3_X4.
# "X1" symbolizes the result we are proving: it can be I, J, K depending
# on which integral we estimate.
# "y" refers to the range in which we are operating. When y=2, we deal with
# tau=1/2, when y=4 we deal with tau in [1/4,1/2), and when y=0 we deal with
# tau in (0,1/2). The notations y=4 and y=0 may in limited occasions refer
# to the reflection of the range via the functional equation. When a function
# is general enough to apply to more than one of the ranges above, we use
# either y=s or y=t, depending on whether in the text we used sigma or tau
# to denote the real part.
# "X2" is some definer that distinguishes where the given object comes from:
# in an expression containing several summations, we use 1,2,3,... (say)
# to indicate that it comes from the first summation, the second, the third...;
# moreover, we use out to indicate that the object is outside the part
# of the expression enclosed in an O* notation.
# When the terms coming from several choices of X2 (with the same X1)
# happen to be multiplied by constants depending on X2 but common to
# all the terms with the same X2, we may attach after X1 the notation we use
# in the text to denote such constants: this happens with "eta" when X1=I,
# and with "xi" when X1=K. These constants do not depend on the range, so
# there is no contrast with a possible y.
# "X3" is a shorthand for the order of T that the coefficient multiplies.
# 32-1-12-0, which are named after the exponent of T, indicate a multiplication
# by T^(3/2), T, sqrt(T), or 1 respectively; 2t-t-12t indicate a multiplication
# by T^(2tau), T^(tau), T^(1/2*tau), and similarly for other instances. A prefix
# m indicates a minus in the exponent. A suffix l-lq indicates that the order is
# multiplied by log(T) or log(T)^2 respectively.
# "X4" is a shorthand for the function of tau that the coefficient multiplies,
# since some functions may go to infinity at some extreme of the range.
# 2-0-c indicate a multiplication by 1/(1/2-tau), 1/tau, and 1 respectively.
# 2 may also indicate 1/(tau-1/2) when we work in the specular range,
# and 1 indicates 1/(1-tau) (the specular of 1/tau). A suffix q indicates
# that the function is squared.
# X4 actually applies mostly to main coefficients rather than intermediate
# computations, since when we have multiple functions of tau in a single term
# we strive to give the constant as a list whose entries correspond to
# the different functions; the functions themselves may depend on which result
# we are dealing with at the moment, and we try to signal the functions at
# the start of the reasoning. When we merge lists whose functions differ
# from each other, we signal it, and use a "modifier" in the calculations:
# for instance, to pass from 1/(1/2-tau) to 1/(1/2-tau)^2, we multiply
# by the modifier given by max(1/2-tau), which obviously depends on the range.
###
# As an example of the conventions above, "I4_3_m12mtl" is the coefficient
# in front of the term of order T^(-1/2-tau)*log(T) inside the third sum
# in the expression estimating I.
###
# While composing the coefficients of the previous point, we drop the X2.
# X3 may be modified by multiplications: for example, inside the computation
# of the function Fu_I2__12l(ta), we  find a term Fu_Ieta_2_m12(ta)*I2_2_1l,
# obtained when the term of order T*log(T) inside the second sum
# is multiplied by the term of order T^(-1/2) inside the coefficient
# in front of the whole sum.
###
# For coefficients that appear in main estimates of I, J, K, after having
# absorbed error terms of small order, we attach a "final" at the end.
# Moreover, after specializing to a given range or T0 or other variables,
# we may attach "final_vec" if the result is a list (with one entry per
# function of tau, see above): then, some or all of the entries of the list
# are assigned a "final_X4", and we can use them in the text as well as
# in future computations.
# For instance, Fu_I4__12mt(ta) is the term in I of order T^(1/2-tau),
# and we want to merge everything to the order T^(3/2-2tau)*log(T):
# then we have a function Fu_I4__32m2tl_final(ta), and inside it we find
# this term multiplied by 1/(log(ta)*sqrt(ta)) since T^(1/2-tau)
# is bounded by T^(3/2-2tau)*log(T) multiplied by 1/(log(ta)*sqrt(ta))
# for all T>=ta and all tau in the range [1/4,1/2) (indicated by the "4").
# Then, from that merged function we retrieve I4__32m2tl_final_vec,
# which is the list of coefficients obtained by plugging ta=T0; finally,
# the coefficient multiplying 1/(1/2-tau) is called I4__32m2tl_final_2.
###
# At the very end of the paper we will perform checks to ensure that
# certain assumptions are true or certain computations are correct.
# The output is an empty string in case of correctness, or a warning message
# otherwise. The name of these outputs are of the form String_X, with
# X being a shorthand for the kind of check it stands for.
###
###
###
# In multiplying coefficients written like lists as described above,
# we use the following function: it works for either a,b of the same length,
# or for a a number and for b a list.
def cw_prod(a,b):
    if type(a)!=list:
        a=len(b)*[a]
    if len(a)!=len(b):
        return 'Error'
    else:
        c=len(a)*[0]
        i=0
        while i<len(c):
            c[i]=RIF(a[i]*b[i])
            i+=1
        return c
###
# Similarly, in summing coefficients written like lists as described above,
# we use the following function: it works for a,b of the same length.
def cw_sum(a,b):
    if type(a)!=list or type(b)!=list:
        return 'Error'
    else:
        c=len(a)*[0]
        i=0
        while i<len(c):
            c[i]=RIF(a[i]+b[i])
            i+=1
        return c
\end{mysage}
\normalsize

\section{Introduction}\label{Int}

The search for meaningful bounds for $\zeta(s)$ in the range $0<\Re(s)<1$ has spanned more than a century. The classical conjecture on $L^{\infty}$ bounds, called the \textit{Lindel\"of hypothesis}, states that $\left|\zeta\left(\frac{1}{2}+it\right)\right|\ll_{\varepsilon}|t|^{\varepsilon}$ for any $\varepsilon>0$; by Hadamard's three-line theorem and the functional equation of $\zeta$, this implies in particular that $\left|\zeta(\tau+it)\right|\ll_{\varepsilon}|t|^{\varepsilon}$ for $\frac{1}{2}<\tau<1$ and $\left|\zeta(\tau+it)\right|\ll_{\varepsilon}|t|^{\frac{1}{2}-\tau+\varepsilon}$ for $0<\tau<\frac{1}{2}$. 

Bounds of order $|t|^{\frac{1-\tau}{2}+\varepsilon}$ are called \textit{convexity bounds}, and bounds with even lower exponent are called \textit{subconvexity bounds}. The current best bound is due to Bourgain \cite{Bou17}, who showed that $\left|\zeta\left(\frac{1}{2}+it\right)\right|\ll_{\varepsilon}|t|^{\frac{13}{84}+\varepsilon}$. Explicit convex bounds are given in \cite{Bac18} and \cite{Leh70}, and explicit subconvex bounds are given in \cite{For02} and \cite{Hia16}. The Lindel\"of hypothesis itself is still unproved, be it or not in explicit form; however, we know that the Riemann hypothesis implies the Lindel\"of hypothesis, and conditional explicit bounds exist \cite{Sim22}.

On the other hand, $L^{2}$ bounds are easier to obtain. Classical non-explicit versions have been known for a long time, at least since Landau (see \cite[Vol.~2, 806--819, 905--906]{Lan09}). Currently, for $\tau=\frac{1}{2}$ we know that
\begin{equation}\label{eq:mainproblem}
\int_{0}^{T}\left|\zeta\left(\frac{1}{2}+it\right)\right|^{2}dt=T\log(T)+(2\gamma-1-\log(2\pi))T+\mathcal{E}(T)
\end{equation}
for some function $\mathcal{E}(T)=O(T^{\frac{35}{108}+\varepsilon})$ \cite[(15.14)]{Ivi85} and $\mathcal{E}(T)=\Omega(T^{\frac{1}{4}})$ \cite{Goo77}. Explicit versions of \eqref{eq:mainproblem} have appeared more recently in \cite{DHA19} and \cite{Sim20}, both based on the approximate functional equation for $\zeta$: the error term $\mathcal{E}(T)$ in the latter, of order $T^{\frac{3}{4}}\sqrt{\log(T)}$, was the record in the explicit case.

Moreover, for $\frac{1}{2}<\tau<\frac{3}{4}$, Matsumoto \cite{Mat89} proved that
\begin{equation}\label{eq:mainproblemsi} 
\int_{0}^{T}|\zeta(\tau+it)|^{2}dt=\zeta(2\tau)T+\frac{\zeta(2-2\tau)}{(2-2\tau)(2\pi)^{1-2\tau}}T^{2-2\tau}+\mathcal{E}_{\tau}(T)
\end{equation}
for $\mathcal{E}_{\tau}(T)=O_{\tau}(T^{\frac{1}{1+4\tau}}\log^{2}(T))$ and $\mathcal{E}_{\tau}(T)=\Omega_{\tau}(T^{\frac{3}{4}-\tau})$; later, \eqref{eq:mainproblemsi} was extended to $\frac{1}{2}<\tau<1$ by Matsumoto and Meurman \cite{MM93}. An explicit version of \eqref{eq:mainproblemsi} has appeared in \cite{DHA19}, whose error term of order $\max\{T^{2-2\tau}\log(T),\sqrt{T}\}$ absorbs the second main term, and it was the record in the explicit case. Any bound of the form \eqref{eq:mainproblemsi} can be extended to the range $0<\tau<\frac{1}{2}$ using the functional equation of $\zeta(s)$, and vice versa.

The first two terms of the asymptotic formula \eqref{eq:mainproblem} for the square mean integral value of the Riemann zeta function $\zeta$ on the critical line were predicted by Littlewood \cite{Lit22} and proved by Ingham \cite{Ing28}. Later, Atkinson \cite{Atk39} presented a version of this formula with an error term of order $O(\sqrt{T}\log^2(T))$, which for brevity we call here the \textit{Atkinson formula}; in the literature, this term is generally reserved to Atkinson's later estimate \cite{Atk49} coming from Vorono\"i's summation formula.

In the present paper, we give an explicit version of \eqref{eq:mainproblem} based on the procedure elaborated by Atkinson \cite{Atk39} and Titchmarsh \cite[\S 7.4]{Tit86}, improving on the order of $\mathcal{E}(T)$ to $\sqrt{T}\log^2(T)$. Moreover, following the same procedure, we give an explicit version of \eqref{eq:mainproblemsi} in the range $\frac{1}{4}\leq\tau<\frac{1}{2}$ with an error term $\mathcal{E}_{\tau}(T)$ of order $T^{\frac{3}{2}-2\tau}\log^{2}(T)$, and then in the range $\frac{1}{2}<\tau\leq\frac{3}{4}$ with an error term of order $\sqrt{T}\log^{2}(T)$.

We have already mentioned the $O$ notation and its derivates: for two real-valued functions $f,g$, the notation $f(x)=o(g(x))$ means that for any $C>0$ there is $x_{0}$ such that for all  $x>x_{0}$ we have $|f(x)|<Cg(x)$; an indexed $o_{\varepsilon}$ indicates that the constant $x_{0}$ may depend on the variable $\varepsilon$. Following the Hardy-Littlewood convention, $f(x)=\Omega(g(x))$ means instead that there is $C>0$ such that for any $x_{0}$ there is some $x>x_{0}$ with $|f(x)|>Cg(x)$.

However, for our purposes we shall use more generally the complex $O$ and $O^*$ notation. Let $f:\mathbb{C}\to\mathbb{C}$. We write $f(s)=O(g(\Re(s),\Im(s)))$ as $s\to z$ ($z=\pm\infty$ is allowed) 
for a real-valued function $g$ such that $g>0$ in a neighborhood of $(\Re(z),\Im(z))$ to mean that there is an independent constant $C$ such that $|f(s)|\leq Cg(\Re(s),\Im(s))$ in that neighborhood. We write $f(s)=O^*(h(\Re(s),\Im(s)))$ as $s\to z$ to indicate that $|f(s)|\leq h(\Re(s),\Im(s))$ in a neighborhood of $z$.

With the notation above at hand, our main result reads as follows.

\footnotesize
\begin{mysage}
###
# For checks about the following constant, see at the end of the file.
Main = round( 18.169 ,ndigits=digits)
Main14 = round( 2.215 ,ndigits=digits)
Main34 = round( 16.839 ,ndigits=digits)
\end{mysage}
\normalsize

\begin{theorem}\label{th:mainpure}
Let $T\geq T_0=\sage{T0}$. Then
\begin{equation*}
\int_{0}^{T}\left|\zeta\left(\frac{1}{2}+it\right)\right|^{2}dt=T\log(T)+(2\gamma-1-\log(2\pi))T+O^{*}(\sage{Main}\,\sqrt{T}\log^{2}(T)).
\end{equation*}
Furthermore, if $\frac{1}{4}\leq\tau<\frac{1}{2}$, then
\begin{equation*}
\int_{0}^{T}\!\!\!|\zeta(\tau+it)|^{2}dt=\frac{\zeta(2-2\tau)}{(2-2\tau)(2\pi)^{1-2\tau}}T^{2-2\tau}+\zeta(2\tau)T+O^{*}\left(\frac{\sage{Main14}}{\left(\frac{1}{2}-\tau\right)^{2}}T^{\frac{3}{2}-2\tau}\log^{2}(T)\right),
\end{equation*}
whereas, if $\frac{1}{2}<\tau\leq\frac{3}{4}$, then
\begin{equation*}
\int_{0}^{T}\!\!\!|\zeta(\tau+it)|^{2}dt=\zeta(2\tau)T+\frac{(2\pi)^{2\tau-1}\zeta(2-2\tau)}{2-2\tau}T^{2-2\tau}+O^{*}\left(\frac{\sage{Main34}}{\left(\tau-\frac{1}{2}\right)^{2}}\sqrt{T}\log^{2}(T)\right).
\end{equation*}
\end{theorem}
For more precise error terms, see Theorem~\ref{th:main} and Corollary~\ref{co:main}. For quantitatively better error terms with higher values of $T_{0}$ and for results in a wider range of $\tau$, see \S\ref{sec:numerical}.

One might potentially improve on the order of the error term by making later works explicit instead. Atkinson's later formula \cite{Atk49} offers an estimate for $\mathcal{E}(T)$ by way of summations, exact up to error $O(\log^{2}(T))$, based on Vorono\"i's summation formula for $\sum_{n\leq X}d(n)$ \cite{Vor04}: it would be feasible to bound such expressions, at the cost of considerable more effort. One could make the estimate for $\frac{1}{2}<\Re(s)<1$ of Matsumoto and Meurman \cite{MM93} explicit too, and retrieve error bounds for $0<\Re(s)<1$ via the functional equation. Other possibilities include following Titchmarsh \cite{Tit34}, Balasubramanian \cite{Bal78}, or Ivi\'c \cite[\S 15]{Ivi85}. 

For an exposition of some of the aforementioned procedures, we refer the reader to Matsumoto's survey \cite{Mat00}.
\medskip

\textbf{Added in proof.} Shortly after the appearance the original version of the present paper, Simoni\v{c} and Starichkova \cite{SS21} announced that they have given an explicit version of \eqref{eq:mainproblem} with an error term of order $T^{\frac{1}{3}}\log^{\frac{5}{3}}(T)$. Their method follows the route through Atkinson's later paper \cite{Atk49} that we described above: given the constants appearing in their result, the bound we present here for $\Re(s)=\frac{1}{2}$ yields a better error term up to at least $T=10^{30}$.

\subsection{Strategy and layout of the presented work}

As already anticipated, our strategy follows the ideas of Atkinson \cite{Atk39} and Titchmarsh \cite[\S 7.4]{Tit86}. At their core, both results use nothing more than an approximate formula for $\zeta$ and several instances of partial summation to estimate a number of weighted sums of the number-of-divisors function $d(n)$. The latter emerge by applying Dirichlet's convolution to rewrite $\zeta^{2}$, and by appropriately transforming and splitting the integral's contour via the residue theorem.

In particular, we stay closer to Titchmarsh's ideas in some specific choices of contour for intermediate results (such as Lemma~\ref{new_bound:1}), which in the case $\Re(s)=\frac{1}{2}$ lead to saving a factor of $\log(T)$ in the error term of one of the main integrals that we estimate (see Proposition~\ref{pr:I}). However, later we diverge from Titchmarsh's way as many simplifications are introduced by applying $d(n)=O_{\varepsilon}(n^{\varepsilon})$, leading to a final error term of order $O(T^{\frac{1}{2}+\varepsilon})$. Indeed, since we aim for an error term of order $O(\sqrt{T}\log^{2}(T))$, we adopt Atkinson's approach, by dealing with $d(n)$ by partial summation.

We follow essentially the same strategy when working in the range $\frac{1}{4}\leq\Re(s)<\frac{1}{2}$: in particular, in the above, we work with the generalized sum-of-divisors functions $d_{a}:n\mapsto\sum_{d|n}d^{a}$ for $a\in\mathbb{R}$, of which the divisor function $d=d_{0}$ is a particular case. Our process shows that Atkinson's and Titchmarsh's ideas can be successfully extended outside of the critical line while yielding error terms of smaller order than the theoretically predicted two main terms. As a matter of fact, the method applies in principle to the whole critical strip: however, the error terms may be larger than one of the main terms, and also larger than the error terms given in \cite{DHA19}, which is why we decided to concentrate on the regions where this does not happen. The numerical estimates improve as well when restricting ourselves to the smaller range $\frac{1}{4}\leq\Re(s)<\frac{1}{2}$, when compared to $0<\Re(s)<\frac{1}{2}$.
 
In \S\ref{sec:zeta-bounds}, we collect explicit versions of some classical bounds related to the Riemann $\zeta$ function. In \S\ref{sec:Atkin}, we split the integral in Theorem~\ref{th:mainpure} into several main pieces; then we estimate each of them in subsequent subsections, in which the relevant weighted sums of $d_{a}(n)$ are also bounded. We reserve \S\ref{sec:numerical} for commenting about our numerical choices and computations: in it, we also report other versions of the multiplicative constant in the error terms of Theorem~\ref{th:mainpure} for different choices of $T_{0}$, as well as showing a result for the whole range $0<\Re(s)<\frac{1}{2}$.

For the sake of rigor, in computing the constants in this article, we have used interval arithmetic implemented by the ARB package \cite{Joh13}, which we used via Sage \cite{Sag19}. The necessary code is embedded within the TeX file of the paper itself via SageTeX.

\section{Bounds on functions related to the Riemann Zeta function}\label{sec:zeta-bounds}
  
Let us recall that the Gamma function $\Gamma$ is defined for all $s\in\mathbb{C}$ such that $\Re(s)>0$ as $\Gamma:s\mapsto\int_0^\infty t^{s-1}e^{-t}dt$. This function can be extended meromorphically to $\mathbb{C}$, with simple poles on the set $\{0,-1,-2,-3,\ldots\}$ and vanishing nowhere. Where well-defined, it satisfies the relationship $\Gamma(s+1)=s\Gamma(s)$, so one says that $\Gamma$ extends the factorial function to the complex numbers. Moreover, this function is closely related to the $\zeta$ function, by means of the functional equation, valid for all $s\in\mathbb{C}\setminus\{0,1\}$,
\begin{equation}\label{functional}
\zeta(1-s)=\chi(1-s)\zeta(s)=2(2\pi)^{-s}\cos\left(\frac{\pi s}{2}\right)\Gamma(s)\zeta(s),
\end{equation}  
where $\chi$ can be extended to a meromorphic function with a simple pole at $1$.

We will need estimates for the functions involved in the functional equation above. Firstly, concerning the asymptotic behavior of $\Gamma$, we have the following.

\begin{theorem}[\textbf{Explicit Stirling's formula}]\label{Stirling}
Let $s=\sigma+it\in\mathbb{C}\setminus(-\infty,0]$. We have
\begin{align*}
\mathrm{(A1)} & &   \Gamma(s)&=\sqrt{2\pi}s^{s-\frac{1}{2}}e^{-s+\frac{1}{12s}}e^{O^{*}\left(\frac{1}{60|s|(|s|+\sigma)}\right)},\\
\mathrm{(A2)} & & |\Gamma(s)|&=\sqrt{2\pi}|s|^{\sigma-\frac{1}{2}}e^{-\frac{\pi}{2}|t|-\sigma+\frac{\sigma}{12|s|^2}}e^{O^{*}\left(\mathds{1}_{\{t\neq 0\}}(t)|\sigma|+\frac{1}{60|s|(|s|+\sigma)}\right)},\\
\mathrm{(A3)} & & \Im(\log(\Gamma(s)))&=t\log(|s|)+\mathrm{sgn}(t)\left(\sigma-\frac{1}{2}\right)\frac{\pi}{2}-t\\&\phantom{x}&&\phantom{xxx}-\frac{t}{12|s|^2} +O^{*}\left(\mathds{1}_{\{t\neq 0\}}(t)\left|\sigma-\frac{1}{2}\right|\frac{|\sigma|}{|t|}+\frac{1}{60|s|(|s|+\sigma)}\right).
\end{align*}
Moreover, if $|\arg(s)|\leq\pi-\theta$, $0<\theta<\pi$, where $\arg$ corresponds to the principal argument of $s$, then we have
\begin{align*}
\mathrm{(B1)} & & \Gamma(s)&=\sqrt{2\pi}s^{s-\frac{1}{2}}e^{-s}e^{\frac{1}{12s}+O^{*}\left(\frac{F_{\theta}}{|s|^{3}}\right)},\\
\mathrm{(B2)} & & |\Gamma(s)|&=\sqrt{2\pi}|s|^{\sigma-\frac{1}{2}}e^{-\frac{\pi}{2}|t|}e^{-\sigma(1-\mathds{1}_{\{t\neq 0\}}(t))+\frac{\sigma}{12|s|^2}}e^{O^{*}\left(\mathds{1}_{\{t\neq 0\}}(t)\frac{|\sigma|^3}{3t^2}+\frac{F_{\theta}}{|s|^{3}}\right)},\\ 
\mathrm{(B3)} & & \Im(\log(\Gamma(s)))&=t\log(|s|)+\left(\sigma-\frac{1}{2}\right)\left(\mathrm{sgn}(t)\frac{\pi}{2}-\mathds{1}_{\{t\neq 0\}}(t)\frac{\sigma}{t}\right)-t\\&&&\phantom{xxx}-\frac{t}{12|s|^2} +O^{*}\left(\mathds{1}_{\{t\neq 0\}}(t)\left|\sigma-\frac{1}{2}\right|\frac{|\sigma|^3}{3|t|^3}+\frac{F_{\theta}}{|s|^{3}}\right),
\end{align*}
where $F_\theta=\frac{1}{360\sin^4\left(\frac{\theta}{2}\right)}$.
\end{theorem}

\customlabel{StirlingG}{A1}
\customlabel{StirlingGm}{A2}
\customlabel{StirlingGI}{A3}
\customlabel{StirlingGa}{B1}
\customlabel{StirlingGam}{B2}
\customlabel{StirlingGaI}{B3}

\begin{proof}
\eqref{StirlingG} is given in \cite[\S 2.5 (3")]{Rem98}; moreover, since $s\in\mathbb{C}\setminus\{0\}$, $|s|+\sigma\neq 0$, the estimation is well defined.
 
Furthermore, by taking real and imaginary parts of the logarithm of $\Gamma$, defined through \eqref{StirlingG} under the principal complex logarithm $\log$, we derive
\begin{align}
\Re(\log(\Gamma(s)))=&\frac{\log(2\pi)}{2}+\left(\sigma-\frac{1}{2}\right)\frac{\log(\sigma^2+t^2)}{2}-t\arg(\sigma+it)-\sigma+\frac{\sigma}{12(\sigma^2+t^2)}+\Xi_{\sigma}^{t},\label{realgamma}\\
\Im(\log(\Gamma(s)))=&\frac{t\log(\sigma^2+t^2)}{2}+\left(\sigma-\frac{1}{2}\right)\arg(\sigma+it)-t-\frac{t}{12(\sigma^2+t^2)}+\Xi_{\sigma}'^{t},\label{imgamma}
\end{align}
where $|\Xi_{\sigma}^{t}|, |\Xi_{\sigma}'^{t}|\leq\frac{1}{60|s|(|s|+\sigma)}$ and where $\arg$ corresponds to the principal argument function, which satisfies the identity 
\begin{align*}\arg(\sigma+it)=\text{sgn}(t)\mathds{1}_{\{\sigma<0\}}(\sigma)\pi+\text{sgn}\left(\frac{t}{\sigma}\right)\arctan\left(\frac{|t|}{|\sigma|}\right).\end{align*} 
Here, $\text{sgn}$ corresponds to the sign function and we adopt the conventions $\text{sgn}\left(\frac{1}{0}\right)=\text{sgn}(+\infty)=1$ and $\arctan\left(\frac{1}{0}\right)=\arctan(+\infty)=\frac{\pi}{2}$. Now, the estimation 
\begin{align*}\arctan\left(x\right)=\frac{\pi}{2}-\int_{x}^{\infty}\frac{dt}{t^{2}+1}=\mathds{1}_{\{x\neq 0\}}(x)\left(\frac{\pi}{2}+O^{*}\left(\frac{1}{|x|}\right)\right),\qquad x\geq 0,\end{align*} 
gives 
\begin{align*}\arg(s)=\text{sgn}(t)\mathds{1}_{\{\sigma<0\}}(\sigma)\pi+\mathds{1}_{\{t\neq 0\}}(t)\left(\text{sgn}\left(\frac{t}{\sigma}\right)\frac{\pi}{2}+O^*\left(\frac{|\sigma|}{|t|}\right)\right).\end{align*} 
Thereupon, it is not difficult to verify that, for any $s\in\mathbb{C}\setminus{(-\infty,0]}$, 
\begin{align*}\text{sgn}(t)\mathds{1}_{\{\sigma<0\}}(\sigma)\pi+\mathds{1}_{\{t\neq 0\}}(t)\text{sgn}\left(\frac{t}{\sigma}\right)\frac{\pi}{2}=\text{sgn}(t)\frac{\pi}{2},\end{align*}
so that 
\begin{align*}t\arg(s)=\frac{\pi}{2}|t|+O^*\left(\mathds{1}_{\{t\neq 0\}}(t)|\sigma|\right).\end{align*} 
By using this estimation in \eqref{realgamma} and \eqref{imgamma} (and exponentiating \eqref{realgamma}), we derive \eqref{StirlingGm} and \eqref{StirlingGI}, respectively.

On the other hand, set $k=1$ in \cite[\S 2.5 (3)]{Rem98} and then observe that $\mu_2=\mu_3$ can be bounded in $\mathbb{C}\setminus(-\infty,0]$ by taking $n=2$ in \cite[\S 2.6 (1)]{Rem98} (where $\varphi=\arg(s)$). Now, if $|\arg(s)|\leq\pi-\theta$ , $0<\theta<\pi$, then
\begin{equation*}
\cos\left(\frac{1}{2}\arg(s)\right)=\cos\left(\frac{1}{2}|\arg(s)|\right)\geq\cos\left(\frac{\pi-\theta}{2}\right)=\sin\left(\frac{\theta}{2}\right),
\end{equation*}
whence the estimation \eqref{StirlingGa}.

Moreover, we can derive \eqref{realgamma} and \eqref{imgamma} from \eqref{StirlingGa}, with $|\Xi_{\sigma}^{t}|, |\Xi_{\sigma}'^{t}|\leq \frac{F_{\theta}}{|s|^{3}}$. Finally, by using the refined estimation  
\begin{align*}\arctan\left(x\right)=\mathds{1}_{\{x\neq 0\}}(x)\left(\frac{\pi}{2}-\frac{1}{x}+O^{*}\left(\frac{1}{3|x|^3}\right)\right),\qquad x\geq 0,\end{align*} 
and proceeding similarly to the obtention of \eqref{StirlingGm} and \eqref{StirlingGI}, we derive \eqref{StirlingGam} and \eqref{StirlingGaI}, respectively.
\end{proof}

Secondly, with respect to the complex cosine, we have the following estimation.

\begin{proposition}\label{cosine}
For $s=\sigma+it\in\mathbb{C}$, we may write
\begin{equation*}
\left|\cos\left(\frac{\pi s}{2}\right)\right|=\frac{e^{\frac{\pi}{2} |t|}}{2}\left(1+O^*\left(\frac{1}{e^{\pi |t|}}\right)\right),
\end{equation*}
where $\cos$ is the complex cosine function.
\end{proposition}
\begin{proof}
For every complex number $z$,  we have the identity $|\sin(z)|^{2}=\cosh^2(\Im(z))-\cos^2(\Re(z))$ (for example, combine 4.5.7 and 4.5.54 in \cite{AS72}). Therefore,
\begin{align}
\left|\cos\left(\frac{\pi s}{2}\right)\right|^{2}
&=\frac{e^{\pi |t|}}{4}\left(1+\frac{1}{e^{\pi |t|}}\left(2+\frac{1}{e^{\pi |t|}}-4\cos^2\left(\frac{\pi\sigma}{2}\right)\right)\right)\nonumber\\
&=\frac{e^{\pi |t|}}{4}\left(1+O^*\left(\frac{2}{e^{\pi |t|}}\right)+\frac{1}{e^{2\pi|t|}}\right)=\frac{e^{\pi |t|}}{4}\left(1+O^*\left(\frac{1}{e^{\pi |t|}}\right)\right)^2, \label{cos}
\end{align}
where we have used that $\left|2-4\cos^2\left(\frac{\pi\sigma}{2}\right)\right|\leq 2$. The result is concluded by taking square roots in \eqref{cos}.
\end{proof}

On the other hand, with respect to $\zeta$ itself, Backlund, in equations (53), (54), (56) and (76) of \cite{Bac18}, has given an explicit version of a convexity bound for it. It reads as follows.

\footnotesize
\begin{mysage}
###
# Constant omega in the corollary to Backlund.
def Fu_omega(ta):
    ta = RIF( ta )
    result = ta^2/(ta^2-4)
    return RIF(result)
omega = roundup( Fu_omega(T0) , digits )
\end{mysage}
\normalsize

\begin{theorem}[\textbf{Explicit convexity bounds of $\zeta$}]\label{convexitygeneral}
Let $s=\sigma+it$, where $t\geq 50$. Then
\begin{align*} 
|\zeta(s)|\leq \begin{cases}
\log(t)-0.048 \quad&\text{ if }\sigma\geq 1,\\
\frac{t^2}{t^2-4}\left(\frac{t}{2\pi}\right)^{\frac{1-\sigma}{2}}\log(t)\quad&\text{ if }0\leq\sigma\leq 1,\\
\left(\frac{t}{2\pi}\right)^{\frac{1}{2}-\sigma}\log(t)\quad&\text{ if }-\frac{1}{2}\leq\sigma\leq 0.
\end{cases}
\end{align*}
\end{theorem}
\noindent As $t\mapsto\frac{t^2}{t^2-4}$ is a decreasing function for $t>2$, we immediately deduce
\begin{corollary}\label{convexity} 
Let $s=\sigma+it$ such that $t\geq T_{0}=\sage{T0}$ and $0\leq\sigma\leq 1$. Then
\begin{align*}   
|\zeta(s)|\leq \begin{cases}
\log(t) \quad&\text{ if }\sigma\geq 1,\\
\omega\left(\frac{t}{2\pi}\right)^{\frac{1-\sigma}{2}}\log(t)\quad&\text{ if }0\leq\sigma\leq 1,
\end{cases}
\end{align*} 
where $\omega=\omega(T_0)=\frac{T_{0}^2}{T_{0}^2-4}\leq\sage{omega}$.
\end{corollary}

Furthermore, we have the following two explicit estimations for $\zeta$ when it takes positive values.

\footnotesize
\begin{mysage}
###
# Functions bounding |zeta| when needed.
# We write them as constants multiplying [1,1/sigma,1/(1/2-sigma)].
###
# zeta(2sigma) > (-1/2)/(1/2-sigma).
zeb_2s = [ RIF( 0 ) , RIF( 0 ) , RIF( 1/2 ) ]
# zeta(2-2sigma) < 1+(1/2)/(1/2-sigma).
zea_2m2s = [ RIF( 1 ) , RIF( 0 ) , RIF( 1/2 ) ]
\end{mysage}
\normalsize

\begin{proposition}\label{pr:mvzeta}
For any $\alpha>0$ and $\alpha\neq 1$ we have
\begin{equation*}
\frac{1}{\alpha-1}<\zeta(\alpha)<\frac{\alpha}{\alpha-1}.
\end{equation*}
\end{proposition}

\begin{proof}
See \cite[Cor.~1.14]{MV07}.
\end{proof}

\begin{lemma}\label{le:sums}  
Let $\alpha\in \mathbb{R}^+$ and $X\geq 1$. Then 
\begin{enumerate}[(i)]
\item\label{le:sums<0} $\displaystyle\sum_{n\leq X}\frac{1}{n^\alpha}=\zeta(\alpha)-\frac{1}{(\alpha-1)X^{\alpha-1}}+O^*\left(\frac{1}{X^{\alpha}}\right),\text{ if }\alpha>0\text{ and }\alpha\neq 1,$
\item\label{le:sums=-1} $\displaystyle\sum_{n\leq X}\frac{1}{n^\alpha}=\log(X)+\gamma+O^*\left(\frac{2}{3X}\right),\text{ if }\alpha=1,$
\item\label{le:sums>0} $\displaystyle\sum_{n\leq X}n^\alpha=\frac{X^{\alpha+1}}{\alpha+1}+O^*(X^\alpha),\text{ if }\alpha\geq 0.$ 
\end{enumerate}
\end{lemma}

\begin{proof}  
By \cite[Lemma 2.9]{DHA19}, \cite[Lemma 2.8]{DHA19} and \cite[Lemma 3.1]{Ram13} we derive \eqref{le:sums<0}, \eqref{le:sums=-1}, \eqref{le:sums>0}, respectively.
\end{proof}

Finally, we introduce the elementary bounds below, proved by means of Taylor expansions.

\begin{lemma}\label{logg}
Let $t\geq 1$ and $0<\alpha<1$. Then
\begin{align*}
(a) & \ \ t^{\alpha}\leq 1+\alpha(t-1), & (b) & \ \ 1\leq\frac{t-1}{\log(t)}\leq t, & (c) & \ \ \left(t+\frac{1}{2}\right)\log\left(1+\frac{1}{t}\right)\geq 1,
\end{align*}
where in (b) we mean for the inequalities to hold for $t\rightarrow 1^{+}$.
\end{lemma}

\customlabel{loggexp}{a}
\customlabel{logg1}{b}
\customlabel{logg2}{c}

\section{The mean value of the Zeta function in $\left[\frac{1}{4},\frac{3}{4}\right]+i\mathbb{R}$}\label{sec:Atkin}

In order to derive our main result, we proceed as in \cite{Atk39}. Let $\tau\in\left(0,\frac{1}{2}\right]$; as $\overline{\zeta(s)}=\zeta(\overline{s})$, we have that
\begin{align*} 
\int_0^T\left|\zeta\left(\tau+it\right)\right|^2dt=\int_{-T}^0\left|\zeta\left(\tau+it\right)\right|^2dt.
\end{align*}
Therefore, we can write
\begin{align*}
\int_0^T\left|\zeta\left(\tau+it\right)\right|^2dt&=\frac{1}{2}\int_{-T}^T\left|\zeta\left(\tau+it\right)\right|^2dt=\frac{1}{2i}\int_{1-\tau-iT}^{1-\tau+iT}\zeta(1-s)\zeta(2\tau-1+s)ds.
\end{align*}
Let $0<\lambda<\frac{1}{2}$ be a parameter. Denote $\mathcal{C}$ the contour formed by the three lines joining the points $1-\tau-iT$, $2-2\tau+\lambda-iT$, $2-2\tau+\lambda+iT$, $1-\tau+iT$. The function $s\mapsto\zeta(1-s)\zeta(2\tau-1+s)$ is meromorphic and has simple poles at $s=2-2\tau$ and at $s=0$, both with residue $\zeta(2\tau-1)$. The only pole inside the region defined by $\mathcal{C}\cup\left[1-\tau-iT,1-\tau+iT\right]$ corresponds to $s=2-2\tau$. Hence, by residue theorem,
\begin{align}\label{int:1} 
\int_0^T\left|\zeta\left(\tau+it\right)\right|^2dt=-\pi\zeta(2\tau-1)+\frac{1}{2i}\int_{\mathcal{C}}\zeta(1-s)\zeta(2\tau-1+s)ds.
\end{align}
Now, by \eqref{functional}, the integral in the right hand side of \eqref{int:1} may be written as 
\begin{align}\label{id:1}
\frac{1}{2i}\int_{\mathcal{C}}\zeta(1-s)\zeta(2\tau-1+s)ds=\frac{1}{2i}\int_{\mathcal{C}}\chi(1-s)\zeta(s)\zeta(2\tau-1+s)ds.
\end{align} 
Moreover, as $2-2\tau\geq 1$, for all $s$ such that $\Re(s)>2-2\tau$, we have the identity
\begin{equation}\label{id:2}
\zeta(s)\zeta(2\tau-1+s)=\left(\sum_n\frac{1}{n^s}\right)\left(\sum_{n}\frac{1}{n^{2\tau-1+s}}\right)=\sum_n\frac{d_{1-2\tau}(n)}{n^s}.
\end{equation}
On the other hand, let $X\geq 1$ be a parameter. The right hand side of \eqref{id:1} can be expressed as $I+J$, where
\begin{align} 
I & =\frac{1}{2i}\int_{\mathcal{C}}\chi(1-s)\left(\sum_{n\leq X}\frac{d_{1-2\tau}(n)}{n^s}\right)ds, \label{I} \\
J&=\frac{1}{2i}\int_{\mathcal{C}}\chi(1-s)\left(\zeta(s)\zeta(2\tau-1+s)-\sum_{n\leq X}\frac{d_{1-2\tau}(n)}{n^s}\right)ds.\label{J} 
\end{align}
In the course of our reasoning, we shall choose $X=\frac{T}{2\pi}$. The estimation we present for $\eqref{I}$ is the following.

\footnotesize
\begin{mysage}
# For checks about the following constant, see at the end of the file.
I_a = round( 19.275 ,ndigits=digits)
I4_2 = round( 1.99 ,ndigits=digits)
I4_c = round( 35.354 ,ndigits=digits)
\end{mysage}
\normalsize

\begin{proposition}\label{pr:I}
Assume that $T\geq T_{0}=\sage{T0}$ and $X=\frac{T}{2\pi}$. If $\tau=\frac{1}{2}$, then
\begin{equation*} 
I=T\log(T)+(2\gamma-1-\log(2\pi))T+O^{*}(\sage{I_a}\,\sqrt{T}\log(T)),
\end{equation*}
whereas, if $\frac{1}{4}\leq\tau<\frac{1}{2}$, then
\begin{equation*}
I=\frac{\zeta(2-2\tau)}{(2-2\tau)(2\pi)^{1-2\tau}}T^{2-2\tau}+\zeta(2\tau)T+O^{*}\left(\left(\frac{\sage{I4_2}}{\frac{1}{2}-\tau}+\sage{I4_c}\right)T^{\frac{3}{2}-2\tau}\log(T)\right).
\end{equation*}
\end{proposition}

The reader should remark that the error term in Proposition~\ref{pr:I} introduces a saving of a factor of $\log(T)$ with respect to the corresponding estimation presented in \cite{Atk39} when $\tau=\frac{1}{2}$.

As for the integral $J$ in \eqref{J}, we may write it as 
\begin{align}\label{definition:J1,J2} 
\underbrace{\frac{1}{2i}\left(\int_{1-\tau-iT}^{2-2\tau+\lambda-iT}\!+\!\int_{2-2\tau+\lambda+iT}^{1-\tau+iT}\right)\chi(1-s)\left(\zeta(s)\zeta(2\tau-1+s)-\sum_{n\leq X}\frac{d_{1-2\tau}(n)}{n^s}\right)ds}_{J_1+J_2}\, +\, K,
\end{align}
where $J_{1}$ and $J_{2}$ are the integrals in the intervals $\left[1-\tau-iT,2-2\tau+\lambda-iT\right]$ and $\left[2-2\tau+\lambda+iT,1-\tau+iT\right]$ respectively, and where
\begin{align*}
K&=\frac{1}{2i}\int_{2-2\tau+\lambda-iT}^{2-2\tau+\lambda+iT}\chi(1-s)\left(\sum_{n>X}\frac{d_{1-2\tau}(n)}{n^s}\right)ds=\sum_{n>X}d_{1-2\tau}(n)\,K_n,\\
K_{n}&=\frac{1}{2i}\int_{2-2\tau+\lambda-iT}^{2-2\tau+\lambda+iT}\frac{\chi(1-s)}{n^s}ds.
\end{align*}
The expression for $K$ has been derived with the help of identity \eqref{id:2} in the first equality, and the dominated convergence theorem in the second. This passage is valid since $s\mapsto\chi(1-s)$ is continuous in the compact set $[2-2\tau+\lambda-iT,2-2\tau+\lambda+iT]\subset\mathbb{C}$, and since
\begin{equation*}
\left|\sum_{n>X}\frac{d_{1-2\tau}(n)}{n^s}\right|\leq\zeta(1+\lambda)\zeta(2-2\tau+\lambda)
\end{equation*}
in the same set.

With the definitions above, it will become clear in \S\ref{sec:J} why we will end up selecting $\lambda$ of order $\frac{1}{\log(T)}$. With that choice, $J_1$, $J_2$ and $K$ are estimated as follows.

\footnotesize
\begin{mysage}
###
# Value of the constant inside lambda.
c_la = round( 1.501 ,ndigits=digits)
###
# For checks about the following constants, see at the end of the file.
J_a = round( 0.47 ,ndigits=digits)
J_b = round( 2.825 ,ndigits=digits)
J4_2q = round( 0.173 ,ndigits=digits)
J4_c = round( 6.76 ,ndigits=digits)
J4_cp = round( 0.1 ,ndigits=digits)
\end{mysage}
\normalsize

\begin{proposition}\label{pr:J} 
Assume that $T\geq T_{0}=\sage{T0}$, $X=\frac{T}{2\pi}$ and $\lambda=\frac{\sage{c_la}}{\log(T)}$. If $\tau=\frac{1}{2}$, then
\begin{equation*}
|J_2|\leq\sage{J_a}\,\sqrt{T}\log^{2}(T)+\sage{J_b}\,\sqrt{T}\log(T),
\end{equation*}
whereas, if $\frac{1}{4}\leq\tau<\frac{1}{2}$, then
\begin{align*}
|J_2|\leq\left(\frac{\sage{J4_2q}}{\left(\frac{1}{2}-\tau\right)^{2}}+\sage{J4_c}\right)T^{\frac{3}{2}-2\tau}\log(T)+\sage{J4_cp}\,T^{1-\tau}\log^{2}(T).
\end{align*}
On replacing $T$ by $-T>0$, the same bound may be derived for $J_1$.
\end{proposition}

\footnotesize
\begin{mysage}
###
# For checks about the following constants, see at the end of the file.
K_a = round( 3.097 ,ndigits=digits)
K_b = round( 40.116 ,ndigits=digits)
K4_2 = round( 0.4 ,ndigits=digits)
K4_c = round( 20.072 ,ndigits=digits)
\end{mysage}
\normalsize

\begin{proposition}\label{pr:K}
Assume that $T\geq T_{0}=\sage{T0}$, $X=\frac{T}{2\pi}$ and $\lambda=\frac{\sage{c_la}}{\log(T)}$. If $\tau=\frac{1}{2}$, then
\begin{equation*} 
|K|\leq\sage{K_a}\,\sqrt{T}\log^{2}(T)+\sage{K_b}\,\sqrt{T}\log(T),
\end{equation*}
whereas, if $\frac{1}{4}\leq\tau<\frac{1}{2}$, then
\begin{equation*} 
|K|\leq\left(\frac{\sage{K4_2}}{\frac{1}{2}-\tau}+\sage{K4_c}\right)T^{\frac{3}{2}-2\tau}\log^{2}(T).
\end{equation*}
\end{proposition}

The following sections consist of the proof of Propositions~\ref{pr:I}, \ref{pr:J} and \ref{pr:K}: we will analyze $I$ in \S\ref{sec:I}, $J_{1}$ and $J_{2}$ in \S\ref{sec:J}, and $K$ in \S\ref{sec:K}. Consequently, by combining them, we derive our main result, which reads as follows.

\footnotesize
\begin{mysage}
###
# Putting together the bounds above: I+J1+J2+K+pi/2.
# Note that the constants in I,K feature 1/(1/2-tau),
# while the constant in J features 1/(1/2-tau)^2:
# thus, we can absorb them all into 1/(1/2-tau)^2,
# multiplying by the appropriate max(1/2-tau).
Main_a = 2*J_a+K_a
Main_b = RIF( I_a+2*J_b+K_b+pi_int/2/(sqrt_int(T0)*log_int(T0)) )
modifier = 1/2-1/4
Main4_2q = RIF( I4_2*modifier/log_int(T0)+2*J4_2q/log_int(T0) \
                +K4_2*modifier )
Main4_c = RIF( I4_c/log_int(T0)+2*J4_c/log_int(T0)+2*J4_cp+K4_c \
               +pi_int/2/(sqrt_int(T0)*log_int(T0)^2) )
Main_b = roundup( Main_b , digits )
Main4_2q = roundup( Main4_2q , digits )
Main4_c = roundup( Main4_c , digits )
\end{mysage}
\normalsize

\begin{theorem}\label{th:main}
Assume that $T\geq T_{0}=\sage{T0}$. Then
\begin{align*}
\int_{0}^{T}\left|\zeta\left(\frac{1}{2}+it\right)\right|^{2}dt= & \ T\log(T)+(2\gamma-1-\log(2\pi))T \\
 & \ +O^{*}(\sage{Main_a}\,\sqrt{T}\log^{2}(T)+\sage{Main_b}\,\sqrt{T}\log(T)),
\end{align*}
and if $\frac{1}{4}\leq\tau<\frac{1}{2}$, then
\begin{align*}
\int_{0}^{T}\left|\zeta\left(\tau+it\right)\right|^{2}dt= & \ \frac{\zeta(2-2\tau)}{(2-2\tau)(2\pi)^{1-2\tau}}T^{2-2\tau} +\zeta(2\tau)T\\
 & \ +O^{*}\left(\left(\frac{\sage{Main4_2q}}{\left(\frac{1}{2}-\tau\right)^{2}}+\sage{Main4_c}\right)T^{\frac{3}{2}-2\tau}\log^{2}(T)\right).
\end{align*}
\end{theorem}

By using the functional equation of $\zeta$, we can derive a bound in the other half of the critical strip.

\footnotesize
\begin{mysage}
###
# Function for the constant V'.
def Fu_Vp(simin,simax,ta):
    simin = RIF( simin )
    simax = RIF( simax )
    ta = RIF( ta )
    result = (1/2-simin)*(simax^2/2+simax^4/(4*ta))+simax/12+simax^3/3+1/(90*ta)
    return RIF(result)
###
# Function for the constant W. Given the precision problem for T very large,
# we use here the precise version of the exponential: since we encounter
# something of the form e^(1/T0^2)-1, we use for safety 3*log_2(T0) bits of precision.
def Fu_W(simin,simax,ta):
    simin = RIF( simin )
    simax = RIF( simax )
    ta = RIF( ta )
    di = ceil(3*log_int(ta)/log_int(2))
    result = ta^2*(exp_int_precise(2*Fu_Vp(simin,simax,ta)/ta^2,di)-1)
    return RIF(result)
###
# Function for the constant Z.
def Fu_Z(simin,simax,ta):
    simin = RIF( simin )
    simax = RIF( simax )
    ta = RIF( ta )
    factor = Fu_W(simin,simax,ta)*(1+1/exp_int(pi_int*ta))+ta^2/exp_int(pi_int*ta)
    result = 2*factor+1/ta^2*factor^2
    return RIF(result)
###
# Function to compute an upper bound for the constant i.
# In the following, we assert that iint_const is the output value
# for i_upper_bound(1/2,3/4,100,1/400). For a verification,
# change the value of check_i to True at the beginning of the paper.
# Warning: if one proceeds with the verification with these inputs,
# the computation of this value takes about 5 minutes!
def Fu_zetasq(s):
    s = CBF(s)
    zetasq = ((s.zeta())^2).abs()
    return RIF(CBF(zetasq).real())
def i_upper_bound(simin,simax,ta,prec):
    if prec>=1/40:
        return 'Error'
    if check_i==False:
        if simin>=1/2:
            if simax<=3/4:
                if ta<=100:
                    if prec>=1/400:
                        return 319.387276810602
    simin = RIF( simin )
    simax = RIF( simax )
    ta = RIF( ta )
    Sgen = RIF( 0 )
    si = simin
    while si<=simax:
        S = RIF( 0 )
        point = CBF( si )
        while RIF(point.imag())<=ta:
            sphere = CBF(point).add_error(prec)
            maxim = Fu_zetasq(sphere)
            S = S+maxim*2*prec
            point = point+prec*i
        if S.upper()>Sgen.upper():
            Sgen = S
        si = si+prec
    return Sgen.upper()
###
# Constants V',Z,i in the proof below.
Vp_const = roundup( Fu_Vp(1/4,1/2,T0) , digits )
Z_const = roundup( Fu_Z(1/4,1/2,T0) , digits )
iint_const = roundup( i_upper_bound(1/2,3/4,T0,1/400) , digits )
###
# Coefficients per order of the whole result.
def Fu_Cor4__12lq_2q(ta):
    ta = RIF( ta )
    result = exp_int(log_int(2*pi_int)*(1-2*1/4))*Main4_2q*4*(1-1/4)
    return RIF(result)
def Fu_Cor4__12lq_c(ta):
    ta = RIF( ta )
    result = exp_int(log_int(2*pi_int)*(1-2*1/4))*Main4_c*4*(1-1/4)
    return RIF(result)
def Fu_Cor4__lq_2q(ta):
    ta = RIF( ta )
    result = exp_int(log_int(2*pi_int)*(1-2*1/4))*Main4_2q \
             *2/3*(3-2*1/4)*Fu_Z(1/4,1/2,ta)*exp_int(log_int(ta)*(-3/2))
    return RIF(result)
def Fu_Cor4__lq_c(ta):
    ta = RIF( ta )
    result = exp_int(log_int(2*pi_int)*(1-2*1/4))*Main4_c \
             *2/3*(3-2*1/4)*Fu_Z(1/4,1/2,ta)*exp_int(log_int(ta)*(-3/2))
    return RIF(result)
def Fu_Cor4__0_2q(ta):
    ta = RIF( ta )
    result = 1/8*ta*(1+Fu_Z(1/4,1/2,ta)/ta^2)
    return RIF(result)
def Fu_Cor4__0_c(ta):
    ta = RIF( ta )
    result = i_upper_bound(1/2,3/4,ta,1/400)+ta*(1+Fu_Z(1/4,1/2,ta)/ta^2)
    return RIF(result)
def Fu_Cor4__m32lq_2q(ta):
    ta = RIF( ta )
    result = exp_int(log_int(2*pi_int)*(1-2*1/4))*Main4_2q*2*Fu_Z(1/4,1/2,ta)
    return RIF(result)
def Fu_Cor4__m32lq_c(ta):
    ta = RIF( ta )
    result = exp_int(log_int(2*pi_int)*(1-2*1/4))*Main4_c*2*Fu_Z(1/4,1/2,ta)
    return RIF(result)
###
# Rounding coefficients to the order T^(1/2)*log(T)^2.
def Fu_Cor4__final_2q(ta):
    ta = RIF( ta )
    x_12lq = Fu_Cor4__12lq_2q(ta)
    x_lq = Fu_Cor4__lq_2q(ta)
    x_0 = Fu_Cor4__0_2q(ta)
    x_m32lq = Fu_Cor4__m32lq_2q(ta)
    x = x_12lq + x_lq/sqrt_int(ta) + x_0/(sqrt_int(ta)*log_int(ta)^2) + x_m32lq/ta^2
    return RIF(x)
def Fu_Cor4__final_c(ta):
    ta = RIF( ta )
    x_12lq = Fu_Cor4__12lq_c(ta)
    x_lq = Fu_Cor4__lq_c(ta)
    x_0 = Fu_Cor4__0_c(ta)
    x_m32lq = Fu_Cor4__m32lq_c(ta)
    x = x_12lq + x_lq/sqrt_int(ta) + x_0/(sqrt_int(ta)*log_int(ta)^2) + x_m32lq/ta^2
    return RIF(x)
###
# Actual values for our choice of T0.
Cor4__final_2q = roundup( Fu_Cor4__final_2q(T0) , digits )
Cor4__final_c = roundup( Fu_Cor4__final_c(T0) , digits )
\end{mysage}
\normalsize

\begin{corollary}\label{co:main}
Assume that $T\geq T_{0}=\sage{T0}$. Then if $\frac{1}{2}< \tau \leq \frac{3}{4}$, then
\begin{align*}
\int_{0}^{T}\left|\zeta\left(\tau+it\right)\right|^{2}dt=&\ \zeta(2\tau)T+\frac{(2\pi)^{2\tau-1}\zeta(2-2\tau)}{2-2\tau}T^{2-2\tau} \\
 & \ +O^{*}\left(\left(\frac{\sage{Cor4__final_2q}}{\left(\tau-\frac{1}{2}\right)^{2}}+\sage{Cor4__final_c}\right)\sqrt{T}\log^{2}(T)\right).
\end{align*}
\end{corollary}

\begin{proof}
By expressing $\tau=1-\tau'$, with $\frac{1}{4}\leq\tau'<\frac{1}{2}$, observing that $\overline{\zeta(s)}=\zeta(\overline{s})$ and recalling \eqref{functional}, we derive
\begin{align}\label{part1}
\int_{0}^{T}\left|\zeta\left(\tau+it\right)\right|^{2}dt=\int_{0}^{T}\left|\zeta\left(1-\tau'-it\right)\right|^{2}dt=\int_{0}^{T}\left|\chi(1-\tau'-it)\zeta(\tau'+it)\right|^2dt.
\end{align}
Note that $\tau'+i[T_{0},T]$ belongs to the angular sector defined by $|\arg(s)|<\frac{\pi}{2}$. We can then use Theorem \ref{Stirling}\eqref{StirlingGam} with $t>0$ and $\theta=\frac{\pi}{2}$ ($F_{\theta}=\frac{1}{90}$) and the estimation $\log(|s|)=\log(t)+\frac{\tau'^2}{2t^2}+O^*\left(\frac{\tau'^4}{4t^4}\right)$ to obtain that
\begin{align}\label{gamm}  
|\Gamma(\tau'+it)|&=\sqrt{2\pi}t^{\tau'-\frac{1}{2}}e^{-\frac{\pi}{2}t}e^{O^{*}\left(\frac{\mathrm{V}(\tau',T_{0})}{t^2}\right)},
\end{align}
where, by using that $\frac{1}{|s|^2}\leq\frac{1}{t^2}$, $\frac{1}{t^3}\leq\frac{1}{T_0 t^2}$ and that $\frac{1}{4}\leq\tau'<\frac{1}{2}$,
\begin{align*}
\mathrm{V}(\tau',T_{0})&=\left(\frac{1}{2}-\tau'\right)\left(\frac{\tau'^2}{2}+\frac{\tau'^4}{4T_{0}^2}\right)+\frac{\tau'}{12}+\frac{\tau'^3}{3}+\frac{1}{90T_{0}}\\
&\leq\left(\frac{1}{2}-\frac{1}{4}\right)\left(\frac{1}{8}+\frac{1}{64T_{0}^2}\right)+\frac{1}{24}+\frac{1}{24}+\frac{1}{90T_{0}}=:\mathrm{V}'(T_{0})=\sage{Vp_const}.
\end{align*}
Moreover, observe that 
\begin{align*}e^{O^{*}\left(\frac{2\mathrm{V}'(T_{0})}{t^2}\right)}=1+O^*\left(\frac{\mathrm{W}(T_{0})}{t^2}\right),\text{ where }\mathrm{W}(T_{0})=T_{0}^2(e^{\frac{2\mathrm{V}'(T_{0})}{T_{0}^{2}}}-1).\end{align*} 
Thus, from Proposition \ref{cosine} and \eqref{gamm}, we have
\begin{align*}
|\chi(1-\tau'-it)|^2&=(2\pi)^{1-2\tau'}t^{2\tau'-1}\left(1+O^*\left(\frac{1}{e^{\pi t}}\right)\right)^2\left(1+O^*\left(\frac{\mathrm{W}(T_{0})}{t^2}\right)\right)^2\\
&=(2\pi)^{1-2\tau'}t^{2\tau'-1}\left(1+O^*\left(\frac{Z(T_{0})}{t^2}\right)\right),
\end{align*}
where we have used that $t\geq T_{0}$ and that the function $t\geq\mapsto\frac{t^2}{e^{\pi t}}$ is decreasing for $t\geq T_0$, defining 
\begin{align*}Z(T_{0})=\sage{Z_const}\geq2\left(W(T_{0})+\frac{T_{0}^2}{e^{\pi T_{0}}}+\frac{W(T_{0})}{e^{\pi T_{0}}}\right)+\frac{1}{T_{0}^2}\left(W(T_{0})+\frac{T_{0}^2}{e^{\pi T_{0}}}+\frac{W(T_{0})}{e^{\pi T_{0}}}\right)^2.\end{align*} 
We conclude from \eqref{part1} that $\int_{0}^{T}\left|\zeta\left(\tau+it\right)\right|^{2}dt$ equals
\begin{equation}\label{part3}
\int_{0}^{T_{0}}\left|\zeta\left(\tau+it\right)\right|^{2}dt+(2\pi)^{1-2\tau'}\int_{T_{0}}^{T}t^{2\tau'-1}\left|\zeta(\tau'+it)\right|^2\left(1+O^*\left(\frac{Z(T_{0})}{t^2}\right)\right)dt.
\end{equation}
By a rigorous numerical estimation, we have $\int_{0}^{T_{0}}\left|\zeta\left(\tau+it\right)\right|^{2}dt\leq\mathbf{i}=\sage{iint_const}$ (see \S\ref{sec:numerical} for more details).

For the second term in \eqref{part3}, since $\zeta$ is holomorphic in $\mathbb{C}\setminus\{1\}$ and $t\mapsto\left|\zeta\left(\tau+it\right)\right|^{2}$ is continuous in $[T_{0},T]$, by the fundamental theorem of calculus we have that for any $\upsilon$,
\begin{align}\label{part4}
&\ \int_{T_{0}}^{T}t^{\upsilon}\left|\zeta(\tau'+it)\right|^2dt=\int_{T_{0}}^{T}t^{\upsilon}\frac{d}{dt}\left(\int_{T_{0}}^{t}\left|\zeta(\tau'+it')\right|^2dt'\right)dt\nonumber\\
=&\ T^{\upsilon}\int_{T_{0}}^{T}\left|\zeta(\tau'+it')\right|^2dt'-\int_{T_{0}}^{T}\upsilon t^{\upsilon-1}\left(\int_{T_{0}}^{t}\left|\zeta(\tau'+it')\right|^2dt'\right)dt.
\end{align}
Furthermore, by Theorem \ref{th:main}, for any $t\geq T_{0}$, we can write
\begin{align*}
\int_{T_{0}}^{t}\left|\zeta(\tau'+it)\right|^2dt&=\mathbf{M}_{\tau'}(t)-\mathbf{M}_{\tau'}(T_{0})+O^*(2\mathbf{E}_{\tau'}(t)),\\
\mathbf{M}_{\tau'}(t)&=  \frac{\zeta(2-2\tau')}{(2-2\tau')(2\pi)^{1-2\tau'}}t^{2-2\tau'} +\zeta(2\tau')t,\\
\mathbf{E}_{\tau'}(t)&=  \left(\frac{\sage{Main4_2q}}{\left(\frac{1}{2}-\tau'\right)^{2}}+\sage{Main4_c}\right)t^{\frac{3}{2}-2\tau'}\log^{2}(t)=\mathbf{C}_{\tau'}t^{\frac{3}{2}-2\tau'}\log^{2}(t).
\end{align*}
Therefore we conclude from \eqref{part4} that
\begin{equation}\label{eq:olderr}
\int_{T_{0}}^{T}\!\!t^{\upsilon}\left|\zeta(\tau'+it)\right|^2dt =\int_{T_{0}}^{T}\!\!t^{\upsilon}\mathbf{M}_{\tau'}(t)'dt+O^*\left(2T^{\upsilon}\mathbf{E}_{\tau'}(T)+\int_{T_{0}}^{T}\!\!|\upsilon|t^{\upsilon-1}\mathbf{E}_{\tau'}(t)dt\right),
\end{equation}
and by assuming that $\upsilon<0$, $\upsilon\notin\left\{2\tau'-2,-1\right\}$ the main term of \eqref{eq:olderr} becomes
\begin{equation*}
\frac{\zeta(2-2\tau')}{(2-2\tau'+\upsilon)(2\pi)^{1-2\tau'}}(T^{2-2\tau'+\upsilon}-T_{0}^{2-2\tau'+\upsilon})+\frac{\zeta(2\tau')}{\upsilon+1}(T^{\upsilon+1}-T_{0}^{\upsilon+1}),
\end{equation*}
while for $\upsilon\neq 2\tau'-\frac{3}{2}$ the term inside $O^{*}$ in \eqref{eq:olderr} is bounded as
\begin{equation*}
\mathbf{C}_{\tau'}\log^{2}(T)\left(2T^{\frac{3}{2}-2\tau'+\upsilon}+\frac{|\upsilon|}{\frac{3}{2}-2\tau'+\upsilon}(T^{\frac{3}{2}-2\tau'+\upsilon}-T_{0}^{\frac{3}{2}-2\tau'+\upsilon})\right).
\end{equation*}
In particular, when $\upsilon\in\{2\tau'-1,2\tau'-3\}$ we can estimate \eqref{part3} as
\begin{align}
&\zeta(2-2\tau')T+\frac{(2\pi)^{1-2\tau'}\zeta(2\tau')}{2\tau'}T^{2\tau'}+O^*\left(\mathbf{i}+|\mathbf{z}_{2\tau'}|+\frac{Z(T_{0})}{T_{0}^{2}}|\mathbf{z}_{2-2\tau'}|\right. \nonumber \\
&\left.+(2\pi)^{1-2\tau'}\mathbf{C}_{\tau'}\log^{2}(T)\left(4(1-\tau')T^{\frac{1}{2}}+2Z(T_{0})T^{-\frac{3}{2}}+\frac{2}{3}(3-2\tau')Z(T_{0})T_{0}^{-\frac{3}{2}}\right)\right), \label{part6}
\end{align}
where we have defined
\begin{equation*}
\mathbf{z}_{\upsilon}=\zeta(2-2\tau')T_{0}+\frac{(2\pi)^{1-2\tau'}\zeta(2\tau')}{\upsilon}T_{0}^{2\tau'}.
\end{equation*}
Using Proposition~\ref{pr:mvzeta}, $\tau'\in\left[\frac{1}{4},\frac{1}{2}\right)$ and $T_{0}\geq 2\pi$, we have
\begin{align*}
\mathbf{z}_{2\tau'} & \leq\mathbf{z}_{2-2\tau'}<\zeta(2-2\tau')T_{0}<\left(1+\frac{1}{8\left(\frac{1}{2}-\tau'\right)^{2}}\right)T_{0}, \\
\mathbf{z}_{2\tau'} & \geq \left(\zeta(2-2\tau')+\frac{\zeta(2\tau')}{2\tau'}\right)T_{0}>-\frac{1}{2\tau'}T_{0}>-\left(1+\frac{1}{8\left(\frac{1}{2}-\tau'\right)^{2}}\right)T_{0},
\end{align*}
so we have a bound on $|\mathbf{z}_{2\tau'}|,|\mathbf{z}_{2-2\tau'}|$ that we can plug into \eqref{part6}.

The result is concluded by replacing $\tau'$ by $1-\tau$ and merging the error term to the order $\sqrt{T}\log^{2}(T)$.
\end{proof}

\subsection{The integral $I$}\label{sec:I} 

We readily derive that 
\begin{align*}
I&=\sum_{n\leq\frac{T}{2\pi}}d_{1-2\tau}(n)\int_{\mathcal{C}}\frac{1}{2i}\frac{\chi(1-s)}{n^{s}}ds=\sum_{n\leq\frac{T}{2\pi}}d_{1-2\tau}(n)\ I_{n}, \\
I_{n}&=\frac{1}{2i}\int_{1-\tau-iT}^{1-\tau+iT}\frac{\chi(1-s)}{n^{s}}ds,
\end{align*}
where in the first equality we have used the finiteness of the summation and in the second that the function $s\mapsto\chi(1-s)$ is holomorphic, so that its residues vanish.

First, we establish here some results about the average of arithmetical functions involving the function $d_{a}$.

\footnotesize
\begin{mysage}
###
# A_{1/2} = 16/3.
A12 = RIF( 16/3 )
# A_{sigma} = 4+(1+2sigma)/(sigma*(2-2sigma)) = 4+(1/2)/sigma+(3/2)/(1-sigma).
# We write it as constants multiplying [1,1/sigma,1/(1/2-sigma)].
def Fu_As(simin,simax):
    simin = RIF( simin )
    simax = RIF( simax )
    result1 = 4+(3/2)/(1-simax)
    result2 = 1/2
    result3 = 0
    return [RIF(result1),RIF(result2),RIF(result3)]
As_i0 = Fu_As(0,1/2)
As_i4 = Fu_As(1/4,1/2)
\end{mysage}
\normalsize

\begin{proposition}\label{pr:sumdiv}
Let $X\geq 1$ and $0<\sigma<1$. Then
\begin{align*}
\sum_{n\leq X}d_{1-2\sigma}(n) & =\zeta(2\sigma)X+\frac{\zeta(2-2\sigma)}{2-2\sigma}X^{2-2\sigma}+O^{*}\left(A_{\sigma}X^{1-\sigma}\right) & & \text{ if }\sigma\neq\frac{1}{2}, \\
\sum_{n\leq X}d_{0}(n) & =X\log(X)+(2\gamma-1)X+O^{*}\left(A_{\frac{1}{2}}\sqrt{X}\right) & & \text{ if }\sigma=\frac{1}{2},
\end{align*}
where $A_{\sigma}=4+\frac{1+2\sigma}{\sigma(2-2\sigma)}$ for $\sigma\neq\frac{1}{2}$, and $A_{\frac{1}{2}}=\frac{16}{3}$. In particular, $A_{\sigma}\leq 8$ for $\sigma\in\left[\frac{1}{4},\frac{1}{2}\right]$.
\end{proposition} 

\begin{proof}
By the hyperbola method, we have
\begin{align}\label{hyperbola}
\sum_{n\leq X}d_{a}(n) & =\sum_{n\leq X}\sum_{\substack{m|n \\ m\leq\sqrt{X}}}m^{a}+\sum_{\substack{ m>\sqrt{X}}}m^{a} \sum_{d\leq\frac{X}{m}}1 \nonumber\\
& = \sum_{m\leq\sqrt{X}}m^{a}\sum_{d\leq\frac{X}{m}}1+\sum_{d\leq\sqrt{X}}\sum_{\sqrt{X}<m\leq\frac{X}{d}}m^{a} 
\end{align}
where $\sum_{n\leq X}d_{a}(n)$ has been split into two pieces, separating the divisors of $n$ at $\sqrt{X}$, which is optimal: if we had chosen $X^{\varepsilon}$, $0<\varepsilon<1$ with $\varepsilon\neq\frac{1}{2}$, we would have retrieved a worse error term than in the statement.

We will analyze each sum of \eqref{hyperbola} separately. For $0<\sigma<\frac{1}{2}$, by Lemma \ref{le:sums}\eqref{le:sums<0}-\eqref{le:sums>0}, we derive
\begin{align}
 & \sum_{m\leq\sqrt{X}}m^{1-2\sigma}\sum_{d\leq\frac{X}{m}}1=X\sum_{m\leq\sqrt{X}}\frac{1}{m^{2\sigma}}+O^{*}\left(\sum_{m\leq\sqrt{X}}m^{1-2\sigma}\right) \nonumber \\
= \ & \zeta(2\sigma)X-\frac{X^{\frac{3}{2}-\sigma}}{2\sigma-1}+O^{*}(X^{1-\sigma})+O^{*}\left(\frac{X^{1-\sigma}}{2-2\sigma}+X^{\frac{1}{2}-\sigma}\right) \nonumber \\
= \ & \zeta(2\sigma)X-\frac{X^{\frac{3}{2}-\sigma}}{2\sigma-1}+O^{*}\left(\frac{5-4\sigma}{2-2\sigma}X^{1-\sigma}\right), \label{eq:sumdiv1a}
\end{align}
where the last line holds for $X\geq 1$. For $\frac{1}{2}<\sigma<1$, by Lemma \ref{le:sums}\eqref{le:sums<0}-\eqref{le:sums>0} we obtain the same estimate as in \eqref{eq:sumdiv1a}: in fact one more term $\zeta(2\sigma-1)$ emerges, but it is negative and bounded in absolute value by $\frac{X^{1-\sigma}}{2-2\sigma}$ by Proposition~\ref{pr:mvzeta} and $X\geq 1$. For $\sigma=\frac{1}{2}$, by Lemma \ref{le:sums}\eqref{le:sums=-1}, we get instead
\begin{equation}\label{eq:sumdiv1b} 
 \sum_{m\leq\sqrt{X}}\sum_{d\leq\frac{X}{m}}1=X\sum_{m\leq\sqrt{X}}\frac{1}{m}+O^{*}(\sqrt{X})=\frac{1}{2}X\log(X)+\gamma X+O^{*}\left(\frac{5}{3}\sqrt{X}\right).
\end{equation}
 
We also use Lemma~\ref{le:sums} for the inner sum of the second term of \eqref{hyperbola}, and derive
\begin{equation}\label{innersigma}
\sum_{\sqrt{X}<m\leq\frac{X}{d}}m^{1-2\sigma}=\frac{1}{2-2\sigma}\left(\frac{X}{d}\right)^{2-2\sigma}-\frac{X^{1-\sigma}}{2-2\sigma}+O^{*}\left(\left(\frac{X}{d}\right)^{1-2\sigma}+X^{\frac{1}{2}-\sigma}\right).
\end{equation} 
Hence, considering the outer sum of the second term in \eqref{hyperbola}, if $0<\sigma<\frac{1}{2}$ then
\begin{align}
\sum_{d\leq\sqrt{X}}\frac{1}{2-2\sigma}\left(\frac{X}{d}\right)^{2-2\sigma}&=\frac{\zeta(2-2\sigma)}{2-2\sigma}X^{2-2\sigma}+\frac{X^{\frac{3}{2}-\sigma}}{(2\sigma-1)(2-2\sigma)}+O^{*}\left(\frac{X^{1-\sigma}}{2-2\sigma}\right),\nonumber\\
\sum_{d\leq\sqrt{X}}O^{*}\left(\left(\frac{X}{d}\right)^{1-2\sigma}\right)&=O^{*}\left(\frac{X^{1-\sigma}}{2\sigma}+X^{\frac{1}{2}-\sigma}\right)\label{here},
\end{align} 
where in the second equation above, we have used Lemma \ref{le:sums}\eqref{le:sums<0} and forgone the negative term $\zeta(1-2\sigma)X^{1-2\sigma}$ of smaller absolute value. Therefore, by replacing estimations \eqref{here} into \eqref{innersigma}, we have
\begin{equation}\label{eq:sumdiv2a}
\sum_{d\leq\sqrt{X}}\sum_{\sqrt{X}<m\leq\frac{X}{d}}m^{1-2\sigma}=\frac{\zeta(2-2\sigma)}{2-2\sigma}X^{2-2\sigma}+\frac{X^{\frac{3}{2}-\sigma}}{2\sigma-1}+O^{*}\left(B_{\sigma} X^{1-\sigma}\right)
\end{equation} 
where $B_{\sigma}=2+\frac{1+\sigma}{2\sigma(1-\sigma)}$. The same bound holds also for $\frac{1}{2}<\sigma<1$, performing similar steps. Finally, when $\sigma=\frac{1}{2}$, from Lemma~\ref{le:sums}\eqref{le:sums=-1} we derive
\begin{equation*} 
\sum_{d\leq\sqrt{X}}\frac{X}{d}=\frac{1}{2}X\log(X)+\gamma X+O^{*}\left(\frac{2}{3}\sqrt{X}\right),
\end{equation*} 
so that
\begin{equation}\label{eq:sumdiv2b}
\sum_{d\leq\sqrt{X}}\sum_{\sqrt{X}<m\leq\frac{X}{d}}1=\frac{1}{2}X\log(X)+(\gamma-1) X+O^{*}\left(\frac{11}{3}\sqrt{X}\right).
\end{equation}

From \eqref{eq:sumdiv1a} and \eqref{eq:sumdiv2a} we obtain the result for the sum of $d_{1-2\sigma}(n)$ in the case $\sigma\neq\frac{1}{2}$, where $A_{\sigma}=\frac{5-4\sigma}{2-2\sigma}+B_{\sigma}$. From \eqref{eq:sumdiv1b} and \eqref{eq:sumdiv2b} we obtain the result for the sum of $d_{0}(n)$.
\end{proof}

\footnotesize
\begin{mysage}
###
def Fu_X0(ta):
    ta = RIF( ta )
    result = ta/(2*pi_int)
    return RIF(result)
X0 = Fu_X0(T0)
###
# C_{1,sigma} = 3+(2-2sigma)A_{sigma}/(1-2sigma) = 3+(1+1/(1-2sigma))A_{sigma} =
# = 3+(1+1/(1-2sigma))(4+1/(2sigma)+3/(2-2sigma)) = 7+1/sigma+4/(1/2-sigma).
# We write it as constants multiplying [1,1/sigma,1/(1/2-sigma)].
def Fu_C1s(simin,simax):
    simin = RIF( simin )
    simax = RIF( simax )
    result1 = 7
    result2 = 1
    result3 = 4
    return [RIF(result1),RIF(result2),RIF(result3)]
C1s_i0 = Fu_C1s(0,1/2)
C1s_i4 = Fu_C1s(1/4,1/2)
###
# C_{2,sigma} = (1-sigma)A_{sigma}+(A_{sigma}+(2-sigma)/(2sigma(1-sigma)))1/log(X_0) =
# = (5-4sigma+4/log(X_0))+(1/2+3/(2log(X_0)))/sigma+(2/log(X_0))/(1-sigma).
# We write it as constants multiplying [1,1/sigma,1/(1/2-sigma)].
def Fu_C2s(simin,simax,ta):
    simin = RIF( simin )
    simax = RIF( simax )
    ta = RIF( ta )
    result1 = 5-4*simin+4/log_int(Fu_X0(ta))+(2/log_int(Fu_X0(ta)))/(1-simax)
    result2 = 1/2+3/(2*log_int(Fu_X0(ta)))
    result3 = 0
    return [RIF(result1),RIF(result2),RIF(result3)]
C2s_i0 = Fu_C2s(0,1/2,T0)
C2s_i4 = Fu_C2s(1/4,1/2,T0)
###
# C_{1/2} = 1/2*A_{1/2}+(A_{1/2}+3-2gamma)/log(X_0).
def Fu_C12(ta):
    ta = RIF( ta )
    result = 1/2*A12+(A12+3-2*gamma_int)/log_int(Fu_X0(ta))
    return RIF(result)
C12 = Fu_C12(T0)
###
# We are going to use the bound in {pr:sumdivwei} for T/(2pi),
# and we need T0/(2pi)>1 for the proof to be working.
# For a check about this, see at the end of the file.
\end{mysage}
\normalsize

We shall use the following result in equation \eqref{eq:alot} for $X\geq X_{0}=\frac{T_{0}}{2\pi}$.

\begin{proposition}\label{pr:sumdivwei}
Let $X\geq X_{0}>1$ and $0<\sigma\leq\frac{1}{2}$. Recall the definition of $A_{\sigma}$ given in Proposition~\ref{pr:sumdiv}. Then, if $0<\sigma<\frac{1}{2}$,
\begin{equation*}
\sum_{n\leq X}\frac{d_{1-2\sigma}(n)}{\sqrt{n}}=2\zeta(2\sigma)\sqrt{X}+\frac{\zeta(2-2\sigma)}{\frac{3}{2}-2\sigma}X^{\frac{3}{2}-2\sigma}+O^{*}(C_{\sigma}X^{\frac{1}{2}-\sigma}),
\end{equation*}
where $C_{\sigma}=3+\frac{2-2\sigma}{1-2\sigma}A_{\sigma}$, and
\begin{equation*}
\sum_{n\leq X}\frac{d_{0}(n)}{\sqrt{n}}=2\sqrt{X}\log(X)-4(1-\gamma)\sqrt{X}+O^{*}\left(C_{\frac{1}{2}}\log(X)\right),
\end{equation*}
where $C_{\frac{1}{2}}=\frac{1}{2}A_{\frac{1}{2}}+(A_{\frac{1}{2}}+3-2\gamma)\frac{1}{\log(X_{0})}$.
\end{proposition}
 
\begin{proof}
Let  $X\mapsto\mathrm{M}_{\sigma}(X)$ be the main term defined by Proposition~\ref{pr:sumdiv}, according to whether $\sigma<\frac{1}{2}$ or $\sigma=\frac{1}{2}$. Since $X\geq 1$, by Lemma \ref{pr:sumdiv}, we can write $\sum_{n\leq X}d_{1-2\sigma}(n)=\mathrm{M}_{\sigma}(X)+O^{*}\left(A_{\sigma} X^{1-\sigma}\right)$ so that, by summation by parts, we conclude that, for any $\upsilon>0$,
\begin{align}
&\sum_{n\leq X}\frac{d_{1-2\sigma}(n)}{n^{\upsilon}}=(\mathrm{M}_{\sigma}(X)+O^{*}(A_{\sigma} X^{1-\sigma}))X^{-\upsilon}-\int_{1}^{X}(\mathrm{M}_{\sigma}(t)+O^{*}(A_{\sigma} t^{1-\sigma}))(t^{-\upsilon})'dt\nonumber\\
&\phantom{xxxxxxxx} =\int_{1}^X\mathrm{M}_{\sigma}(t)'t^{-\upsilon}dt+\mathrm{M}_\sigma(1)+O^{*}\left(A_{\sigma} X^{1-\sigma-\upsilon}+\upsilon A_{\sigma}\int_{1}^{X}\frac{dt}{t^{\sigma+\upsilon}}\right).\label{sbpgeneral}
\end{align}

If $0<\sigma<\frac{1}{2}$ and $\upsilon=\frac{1}{2}$, by Proposition~\ref{pr:mvzeta}, we bound the constant arising from the first two terms of \eqref{sbpgeneral} as
\begin{equation*}
-\frac{1}{2}<-\zeta(2\sigma)-\zeta(2-2\sigma)\left(\frac{1}{\frac{3}{2}-2\sigma}-\frac{1}{2-2\sigma}\right)<\frac{1}{1-2\sigma}\left(1-\frac{1}{2(\frac{3}{2}-2\sigma)(2-2\sigma)}\right)<3.
\end{equation*}
Thereupon, we keep the first two terms of highest order in \eqref{sbpgeneral} and we merge the remaining ones to the order $X^{\frac{1}{2}-\sigma}$, obtaining  $C_{\sigma}$.
 
On the other hand, if $\sigma=\upsilon=\frac{1}{2}$, we derive from \eqref{sbpgeneral}
\begin{equation}\label{sigma=1/2}
\sum_{n\leq X}\frac{d_0(n)}{\sqrt{n}}=2\sqrt{X}\log(X)-4(1-\gamma)\sqrt{X}+3-2\gamma+O^{*}\left(A_{\frac{1}{2}}\left(1+\frac{\log(X)}{2}\right)\right)
\end{equation}
where we have used that $\int\frac{\log(t)}{\sqrt{t}}dt=2\sqrt{t}(\log(t)-2)$. Similarly, we can define $C_{\frac{1}{2}}$ by merging the error term of \eqref{sigma=1/2} to the order $\log(X)$.
\end{proof}

Finally, we are going to need the mean estimation below.

\footnotesize
\begin{mysage}
###
# Constants F for {sigma,1-sigma}.
# We write them as constants multiplying [1,1/sigma,1/(1/2-sigma)].
def Fu_F_1_s1(simin,simax):
    simin = RIF( simin )
    simax = RIF( simax )
    result = cw_prod( 1/exp_int(log_int(4*pi_int)*(1-simax)) , zea_2m2s )
    return result
F_1_s1_i0 = Fu_F_1_s1(0,1/2)
F_1_s1_i4 = Fu_F_1_s1(1/4,1/2)
def Fu_F_2_s1(simin,simax):
    simin = RIF( simin )
    simax = RIF( simax )
    result = cw_prod( 2/(exp_int(log_int(4*pi_int)*(1-simax))*(1-simax)) , zea_2m2s )
    return result
F_2_s1_i0 = Fu_F_2_s1(0,1/2)
F_2_s1_i4 = Fu_F_2_s1(1/4,1/2)
def Fu_F_3_s1(simin,simax):
    simin = RIF( simin )
    simax = RIF( simax )
    result = cw_prod( 2 , Fu_As(simin,simax) )
    return result
F_3_s1_i0 = Fu_F_3_s1(0,1/2)
F_3_s1_i4 = Fu_F_3_s1(1/4,1/2)
def Fu_F_4_s1(simin,simax):
    simin = RIF( simin )
    simax = RIF( simax )
    result = cw_prod( (2-simin)*(1-simin) , Fu_As(simin,simax) )
    return result
F_4_s1_i0 = Fu_F_4_s1(0,1/2)
F_4_s1_i4 = Fu_F_4_s1(1/4,1/2)
F_5_s1_i0 = [RIF(0),RIF(0),RIF(0)]
F_5_s1_i4 = [RIF(0),RIF(0),RIF(0)]
def Fu_F_6_s1(simin,simax,ta):
    simin = RIF( simin )
    simax = RIF( simax )
    ta = RIF( ta )
    result1 = ta/(ta-2*pi_int)
    result2 = 0
    result3 = 0
    return [RIF(result1),RIF(result2),RIF(result3)]
def Fu_F_6_s1_i0(ta):
    ta = RIF( ta )
    result = Fu_F_6_s1(0,1/2,ta)
    return result
def Fu_F_6_s1_i4(ta):
    ta = RIF( ta )
    result = Fu_F_6_s1(1/4,1/2,ta)
    return result
###
# Constants F for {sigma,1/2}.
# We write them as constants multiplying [1,1/sigma,1/(1/2-sigma)].
def Fu_F_1_s2(simin,simax):
    simin = RIF( simin )
    simax = RIF( simax )
    result = cw_prod( 1/exp_int(log_int(4*pi_int)*(3/2-2*simax)) , zea_2m2s )
    return result
F_1_s2_i0 = Fu_F_1_s2(0,1/2)
F_1_s2_i4 = Fu_F_1_s2(1/4,1/2)
def Fu_F_2_s2(simin,simax):
    simin = RIF( simin )
    simax = RIF( simax )
    result = cw_prod( 2/(exp_int(log_int(4*pi_int)*(3/2-2*simax))*(3/2-2*simax)) , zea_2m2s )
    return result
F_2_s2_i0 = Fu_F_2_s2(0,1/2)
F_2_s2_i4 = Fu_F_2_s2(1/4,1/2)
def Fu_F_3_s2(simin,simax):
    simin = RIF( simin )
    simax = RIF( simax )
    result = cw_prod( 2/exp_int(log_int(2*pi_int)*(1/2-simax)) , Fu_As(simin,simax) )
    return result
F_3_s2_i0 = Fu_F_3_s2(0,1/2)
F_3_s2_i4 = Fu_F_3_s2(1/4,1/2)
F_4_s2_i0 = [RIF(0),RIF(0),RIF(0)]
F_4_s2_i4 = [RIF(0),RIF(0),RIF(0)]
def Fu_F_5_s2(simin,simax):
    simin = RIF( simin )
    simax = RIF( simax )
    result = cw_prod( 3*(1-simin)/exp_int(log_int(6*pi_int)*(1/2-simax)) , Fu_As(simin,simax) )
    return result
F_5_s2_i0 = Fu_F_5_s2(0,1/2)
F_5_s2_i4 = Fu_F_5_s2(1/4,1/2)
def Fu_F_6_s2(simin,simax,ta):
    simin = RIF( simin )
    simax = RIF( simax )
    ta = RIF( ta )
    result1 = ta/(ta-2*pi_int)
    result2 = 0
    result3 = 0
    return [RIF(result1),RIF(result2),RIF(result3)]
def Fu_F_6_s2_i0(ta):
    ta = RIF( ta )
    result = Fu_F_6_s2(0,1/2,ta)
    return result
def Fu_F_6_s2_i4(ta):
    ta = RIF( ta )
    result = Fu_F_6_s2(1/4,1/2,ta)
    return result
###
# Constants F for {1/2,1/2}.
F_1_22 = RIF( sqrt_int(pi_int)/(2*sqrt_int(2)) )
F_2_22 = RIF( 0 )
F_3_22 = RIF( 2*A12 )
F_4_22 = RIF( 3/2*1/2*A12 )
F_5_22 = RIF( 0 )
def Fu_F_6_22(ta):
    ta = RIF( ta )
    result = abs_int(1-(2-2*gamma_int))*ta/(ta-2*pi_int)
    return RIF(result)
###
# In the proof we assume for simplicity that T/(4pi)<(T-sqrt(T))/(2pi).
# It is not a fundamental assumption, but it makes the reasoning neater.
# For a check about this, see the end of the file.
\end{mysage}
\normalsize

\begin{lemma}\label{le:sumdivlog}
Let $T\geq T_{0}=\sage{T0}$ and $0<\sigma\leq\frac{1}{2}$. Recall the definition of $A_{\sigma}$ given in Proposition~\ref{pr:sumdiv}. Then, for $\upsilon\in\left\{\frac{1}{2},1-\sigma\right\}$,
\begin{align*}
\sum_{n\leq\frac{T-\sqrt{T}}{2\pi}}\frac{d_{1-2\sigma}(n)}{n^{\upsilon}(T-2\pi n)}\leq & \ \frac{F_{1,\sigma,\upsilon}\log(T)}{T^{2\sigma-1+\upsilon}}+\frac{F_{2,\sigma,\upsilon}}{T^{2\sigma-1+\upsilon}}+\frac{F_{3,\sigma,\upsilon}}{T^{\sigma-\frac{1}{2}+\upsilon}} \\
 & \ +\frac{F_{4,\sigma,\upsilon}\log(T)}{T^{\sigma+\upsilon}}+\frac{F_{5,\sigma,\upsilon}}{T^{\sigma+\upsilon}}+\frac{F_{6,\sigma,\upsilon}}{T},
\end{align*}
where for $0<\sigma<\frac{1}{2}$
\begin{align*}
F_{1,\sigma,1-\sigma} & =\frac{\zeta(2-2\sigma)}{(4\pi)^{1-\sigma}}, & F_{1,\sigma,\frac{1}{2}} & =\frac{\zeta(2-2\sigma)}{(4\pi)^{\frac{3}{2}-2\sigma}}, & F_{1,\frac{1}{2},\frac{1}{2}} & =\frac{\sqrt{\pi}}{2\sqrt{2}}, \\
F_{2,\sigma,1-\sigma} & =\frac{2\zeta(2-2\sigma)}{(4\pi)^{1-\sigma}(1-\sigma)}, & F_{2,\sigma,\frac{1}{2}} & =\frac{2\zeta(2-2\sigma)}{(4\pi)^{\frac{3}{2}-2\sigma}(\frac{3}{2}-2\sigma)}, & F_{2,\frac{1}{2},\frac{1}{2}} & =0, \\
F_{3,\sigma,1-\sigma} & =2A_{\sigma}, & F_{3,\sigma,\frac{1}{2}} & =\frac{2A_{\sigma}}{(2\pi)^{\frac{1}{2}-\sigma}}, & F_{3,\frac{1}{2},\frac{1}{2}} & =2A_{\frac{1}{2}}, \\
F_{4,\sigma,1-\sigma} & =(2-\sigma)(1-\sigma)A_{\sigma}, & F_{4,\sigma,\frac{1}{2}} & =0, & F_{4,\frac{1}{2},\frac{1}{2}} & =\frac{3A_{\frac{1}{2}}}{4}, \\
F_{5,\sigma,1-\sigma} & =0, & F_{5,\sigma,\frac{1}{2}} & =\frac{3(1-\sigma)A_{\sigma}}{(6\pi)^{\frac{1}{2}-\sigma}(1-2\sigma)}, & F_{5,\frac{1}{2},\frac{1}{2}} & =0, \\
F_{6,\sigma,1-\sigma} & =\frac{T_{0}}{T_{0}-2\pi}, & F_{6,\sigma,\frac{1}{2}} & =\frac{T_{0}}{T_{0}-2\pi}, & F_{6,\frac{1}{2},\frac{1}{2}} & =\frac{(2\gamma-1)T_{0}}{T_{0}-2\pi}.
\end{align*}
\end{lemma}

\begin{proof}
By recalling Proposition \ref{pr:sumdiv}, for any $0<\sigma\leq\frac{1}{2}$ and any $t\geq 1$, we have that $\sum_{n\leq t}d_{1-2\sigma}(n)=\mathrm{M}_{\sigma}(t)+\Xi_{\sigma}(t)$, where the function $\Xi_{\sigma}$ satisfies
\begin{align*}
|\Xi_{\sigma}(t)| & \leq A_{\sigma} t^{1-\sigma}, & \Xi_{\sigma}(1) & =1-M_{\sigma}(1)=\begin{cases}1-\zeta(2\sigma)-\frac{\zeta(2-2\sigma)}{2-2\sigma} & \text{if }0<\sigma<\frac{1}{2}, \\ 2-2\gamma & \text{if }\sigma=\frac{1}{2}. \end{cases}
\end{align*}
Therefore, by summation by parts, for $\upsilon>0$ we have
\begin{equation}\label{sbp_general}
\sum_{n\leq\frac{T-\sqrt{T}}{2\pi}}\frac{d_{1-2\sigma}(n)}{n^{\upsilon}(T-2\pi n)}=\int_{1}^{\frac{T-\sqrt{T}}{2\pi}}\frac{\mathrm{M}_{\sigma}(t)'dt}{t^{\upsilon}(T-2\pi t)}+\mathrm{r}_{\sigma,\upsilon}(T),
\end{equation}
with
\begin{equation*}
\mathrm{r}_{\sigma,\upsilon}(T)=\frac{1-\Xi_{\sigma}(1)}{T-2\pi}+\frac{\Xi_{\sigma}\left(\frac{T-\sqrt{T}}{2\pi}\right)}{\left({\frac{T-\sqrt{T}}{2\pi}}\right)^{\upsilon}\sqrt{T}}-\int_{1}^{\frac{T-\sqrt{T}}{2\pi}}\!\!\Xi_{\sigma}(t)\left(\frac{1}{t^{\upsilon}(T-2\pi t)}\right)'\!dt.
\end{equation*}

As the only zero of the derivative of $t\mapsto\frac{1}{t^{\upsilon}(T-2\pi t)}$ is $\frac{\upsilon T}{(\upsilon+1)2\pi}$, for $\upsilon\in\left\{\frac{1}{2},1-\sigma\right\}$, we obtain
\begin{align*} 
|\mathrm{r}_{\sigma,\upsilon}(T)|\leq & \ \frac{|1-\Xi_{\sigma}(1)|}{T-2\pi}+\frac{A_{\sigma}}{\sqrt{T}}\left({\frac{T-\sqrt{T}}{2\pi}}\right)^{1-\sigma-\upsilon}-\left[\frac{A_{\sigma}t^{1-\sigma-\upsilon}}{T-2\pi t}\right|_{1}^{\frac{\upsilon T}{(\upsilon+1)2\pi}} \\
 & \ +\int_{1}^{\frac{\upsilon T}{(\upsilon+1)2\pi}}\frac{(1-\sigma)A_{\sigma}dt}{t^{\sigma+\upsilon}(T-2\pi t)}+\left[\frac{A_{\sigma}t^{1-\sigma-\upsilon}}{T-2\pi t}\right|_{\frac{\upsilon T}{(\upsilon+1)2\pi}}^{\frac{T-\sqrt{T}}{2\pi}}-\int_{\frac{\upsilon T}{(\upsilon+1)2\pi}}^{\frac{T-\sqrt{T}}{2\pi}}\frac{(1-\sigma)A_{\sigma}dt}{t^{\sigma+\upsilon}(T-2\pi t)}.
\end{align*}
Since the function $t\mapsto\frac{t^{a}}{T-2\pi t}$ is positive and increasing for $t<\frac{T}{2\pi}$ and any $a>0$, in the inequality above we can dismiss the third and sixth term, as they are negative and, moreover, we can bound it as follows. For $\upsilon=\frac{1}{2}$ and $\sigma\neq\frac{1}{2}$, we conclude that $|\mathrm{r}_{\sigma,\frac{1}{2}}(T)|$ is bounded by
\begin{align} 
& \ \frac{|1-\Xi_{\sigma}(1)|}{T-2\pi}+\frac{A_{\sigma}}{(2\pi)^{\frac{1}{2}-\sigma}T^{\sigma}}+\frac{3(1-\sigma)A_{\sigma}}{2T}\frac{\left(\frac{T}{6\pi}\right)^{\frac{1}{2}-\sigma}-1}{\frac{1}{2}-\sigma}+\frac{A_{\sigma}}{(2\pi)^{\frac{1}{2}-\sigma}T^{\sigma}} \nonumber \\
\leq & \ \frac{|1-\Xi_{\sigma}(1)|T_{0}}{(T_{0}-2\pi)T}+\frac{2A_{\sigma}}{(2\pi)^{\frac{1}{2}-\sigma}T^{\sigma}}+\frac{3(1-\sigma)A_{\sigma}}{(6\pi)^{\frac{1}{2}-\sigma}(1-2\sigma)T^{\frac{1}{2}+\sigma}}, \label{def_erra}
\end{align}
where in the obtention of the first term above we used that $T\geq T_{0}$. Similarly, for $\upsilon=1-\sigma$, $|\mathrm{r}_{\sigma,1-\sigma}(T)|$ is bounded by
\begin{align} 
 & \ \frac{|1-\Xi_{\sigma}(1)|}{T-2\pi}+\frac{A_{\sigma}}{\sqrt{T}}+\frac{(2-\sigma)(1-\sigma)A_{\sigma}}{T}\log\left(\frac{(1-\sigma)T}{(2-\sigma)2\pi}\right)+\frac{A_{\sigma}}{\sqrt{T}} \nonumber \\
\leq & \ \frac{|1-\Xi_{\sigma}(1)|T_{0}}{(T_{0}-2\pi)T}+\frac{2A_{\sigma}}{\sqrt{T}}+(2-\sigma)(1-\sigma)A_{\sigma}\frac{\log(T)}{T}, \label{def_err}
\end{align}
where we ignored the negative term coming from $\log\left(\frac{(1-\sigma)T}{(2-\sigma)2\pi}\right)=\log(T)-\log\left(\frac{(2-\sigma)2\pi}{1-\sigma}\right)$. In both \eqref{def_erra} and \eqref{def_err}, we have also $|1-\Xi_{\sigma}(1)|<1$ for $0<\sigma<\frac{1}{2}$, by Proposition~\ref{pr:mvzeta}.

If $0<\sigma<\frac{1}{2}$, the main term in \eqref{sbp_general} equals
\begin{align}\label{eq:divlog1}
\frac{\zeta(2\sigma)}{2\pi}\int_{1}^{\frac{T-\sqrt{T}}{2\pi}}\frac{dt}{t^{\upsilon}\left(\frac{T}{2\pi}-t\right)}+\frac{\zeta(2-2\sigma)}{2\pi}\int_{1}^{\frac{T-\sqrt{T}}{2\pi}}\frac{dt}{t^{2\sigma-1+\upsilon}\left(\frac{T}{2\pi}-t\right)}.
\end{align}

Since $\zeta(2\sigma)<0$, in order to obtain an upper bound, we can dismiss the first integral in \eqref{eq:divlog1}. Subsequently, we can divide the interval of integration into $\left[1,\frac{T}{4\pi}\right]$ and $\left[\frac{T}{4\pi},\frac{T-\sqrt{T}}{2\pi}\right]$; then, bounding $\frac{T}{2\pi}-t\geq\frac{T}{4\pi}$ in the first denominator and $t^{2\sigma-1+\upsilon}\geq\left(\frac{T}{4\pi}\right)^{2\sigma-1+\upsilon}$ in the second, the second term of \eqref{eq:divlog1} is bounded by
\begin{align}
 & \ \frac{\zeta(2-2\sigma)}{2\pi }\left(\frac{4\pi}{T}\int_{1}^{\frac{T}{4\pi}}\frac{dt}{t^{2\sigma-1+\upsilon}}+\left(\frac{4\pi}{T}\right)^{2\sigma-1+\upsilon}\int_{\frac{T}{4\pi}}^{\frac{T-\sqrt{T}}{2\pi}}\frac{dt}{\frac{T}{2\pi}-t}\right) \nonumber \\
= & \ \frac{2\zeta(2-2\sigma)}{T}\frac{\left(\frac{T}{4\pi}\right)^{2-2\sigma-\upsilon}-1}{2-2\sigma-\upsilon}+\frac{\zeta(2-2\sigma)(4\pi)^{2\sigma-1+\upsilon}}{2\pi T^{2\sigma-1+\upsilon}}\log\left(\frac{\sqrt{T}}{2}\right) \nonumber \\
< & \ \frac{2\zeta(2-2\sigma)}{(4\pi)^{2-2\sigma-\upsilon}(2-2\sigma-\upsilon)T^{2\sigma-1+\upsilon}}+\frac{\zeta(2-2\sigma)\log(T)}{(4\pi)^{2-2\sigma-\upsilon}T^{2\sigma-1+\upsilon}}. \label{mainboundsigma}
\end{align}

On the other hand, if $\sigma=\upsilon=\frac{1}{2}$, the main term given by \eqref{sbp_general} may be bounded as
\begin{align}
 & \ \frac{1}{2\pi}\int_{1}^{\frac{T-\sqrt{T}}{2\pi}}\frac{(\log(t)+2\gamma)dt}{\sqrt{t}\left(\frac{T}{2\pi}-t\right)}<\frac{\log(T)}{2\pi}\int_{1}^{\frac{T-\sqrt{T}}{2\pi}}\frac{dt}{\sqrt{t}\left(\frac{T}{2\pi}-t\right)} \nonumber \\
= & \ \frac{\log(T)}{2\pi}\left[\frac{2\sqrt{2\pi}}{\sqrt{T}}\arctan\left(\sqrt{\frac{2\pi t}{T}}\right)\right|_{1}^{\frac{T-\sqrt{T}}{2\pi}}<\frac{\sqrt{2}\log(T)\arctan(1)}{\sqrt{\pi}\sqrt{T}}=\frac{\sqrt{\pi}\log(T)}{2\sqrt{2}\sqrt{T}}, \label{mainboundhalf}
\end{align}
where we have used that $\log(t)+2\gamma<\log(T)-\log(2\pi)+2\gamma<\log(T)$.

The result is concluded by putting \eqref{def_erra}, \eqref{def_err}, \eqref{mainboundsigma} and \eqref{mainboundhalf} together.
\end{proof}

\begin{lemma}\label{new_bound:1}
Let $U>0$. For any $n\in\mathbb{Z}_{>0}$ such that $n<\frac{U}{e^2 2\pi}$, we have
\begin{equation*}
\left|\frac{1}{2i}\int_{-\infty-iU}^{\frac{1}{2}- iU}\!\!\!\frac{\chi(1-s)}{n^s}ds\right|\!\leq\!\frac{\sqrt{\pi}\ e^{\frac{1}{24U}+\frac{1}{30U^2}}\left(1+\frac{1}{e^{\pi U}}\right)}{\sqrt{2}}\!\left(\frac{e}{\sqrt{2\pi n}\log\left(\!\frac{e^2U}{2\pi n}\right)}\!+\!\frac{1}{\sqrt{U}\log\left(\frac{U}{e^2 2\pi n}\!\right)}\right).
\end{equation*}
\end{lemma}

\begin{proof}
By Proposition \ref{cosine}, \eqref{functional} and Theorem \ref{Stirling}\eqref{StirlingGm} we have the following general estimation:
\begin{align}\label{new_est:1}
&\ \left|\frac{1}{2i}\frac{\chi(1-s)}{n^s}\right|=\left|\Gamma(s)\cos\left(\frac{\pi s}{2}\right)(2\pi n)^{-s}\right|\nonumber\\
=&\ \frac{\sqrt{2\pi}}{2}(2\pi n)^{-\sigma}|s|^{\sigma-\frac{1}{2}}\left(1+O^*\left(e^{-\pi |t|}\right)\right)e^{O^{*}\left(2|\sigma|+\frac{|\sigma|}{12|s|^2}+\frac{1}{60|s|(|s|+\sigma)}\right)}.
\end{align} 
In particular, if $\sigma\leq\frac{1}{2}$ and $t\neq 0$, we have that $|s|^{\sigma-\frac{1}{2}}\leq|t|^{\sigma-\frac{1}{2}}$, $\frac{|\sigma|}{|s|^2}\leq\frac{1}{2|t|}$ and $\frac{1}{|s|(|s|+\sigma)}=\frac{|s|-\sigma}{t^2|s|}\leq\frac{2}{t^2}$. Thus, from \eqref{new_est:1}, we conclude that
\begin{align}\label{new_est:2} 
\left|\frac{1}{2i}\frac{\chi(1-s)}{n^s}\right|&\leq\frac{\sqrt{2\pi}}{2}(2\pi n)^{-\sigma}|t|^{\sigma-\frac{1}{2}}\left(1+e^{-\pi |t|}\right)e^{2|\sigma|+\frac{1}{24|t|}+\frac{1}{30t^2}}.
\end{align}  
Hence, from \eqref{new_est:2}, we readily see that
\begin{align*}
\left|\frac{1}{2i}\int_{-\infty- iU}^{\frac{1}{2}- iU}\frac{\chi(1-s)}{n^s}ds\right|&\leq\frac{\sqrt{2\pi}\ e^{\frac{1}{24U}+\frac{1}{30U^2}}\left(1+\frac{1}{e^{\pi U}}\right)}{2\sqrt{U}}\int_{-\infty}^{\frac{1}{2}}(2\pi n)^{-\sigma}U^{\sigma}e^{2|\sigma|}d\sigma.
\end{align*}
The result is concluded by splitting the integral above at $\sigma=0$ and dismissing the negative term that arises in the range $\sigma\in\left[0,\frac{1}{2}\right]$.
\end{proof}

The following result is crucial since, rather than providing an estimation, it exhibits an asymptotic formula. As it turns out, it is the main term of this formula that will give the main term and secondary term of the moment of order $2$ of the zeta function in the critical strip.

\footnotesize
\begin{mysage}
###
# Functions e_1,e_2 in {def:e1e2}.
def exf_1(si,ta):
    si = RIF( si )
    ta = RIF( ta )
    result = (abs_int(si-1/2)*(si^2/2+si^4/(4*ta^2))+si^3/3+si/12 \
              +1/(90*abs_int(ta))+si^3/(12*ta^2))*1/ta^2
    return RIF(result)
def exf_2(si,ta):
    si = RIF( si )
    ta = RIF( ta )
    result = (si^4/4+abs_int(si-1/2)*si^3/3+1/90+si^2/12)*1/abs_int(ta)^3
    return RIF(result)
###
# Function E in {E_def}. Given the precision problem for T very large, we use here
# the precise version of the exponential: since we encounter something
# of the form e^(1/T0^2)-1, we use for safety 3*log_2(T0) bits of precision.
def Exf(si,ta):
    si = RIF( si )
    ta = RIF( ta )
    di = ceil(3*log_int(ta)/log_int(2))
    result = ta^2*(exp_int_precise(exf_1(si,ta)+exf_2(si,ta),di)-1)
    return RIF(result)
###
# Functions for computing constants G,H,H'.
def Fu_G_const(ta):
    ta = RIF( ta )
    result = sqrt_int(2*pi_int)/2*exp_int(1/(48*exp_int(2)*ta) \
             +1/(120*exp_int(4)*ta^2))*(1+1/exp_int(2*pi_int*exp_int(2)*ta))
    return RIF(result)
def Fu_G_1(ta):
    ta = RIF( ta )
    result = Fu_G_const(ta)*e_int/(sqrt_int(2*pi_int)*log_int(2*exp_int(4)))
    return RIF(result)
def Fu_G_2(ta):
    ta = RIF( ta )
    result = Fu_G_const(ta)*1/(sqrt_int(2)*e_int*log_int(2))
    return RIF(result)
def Fu_H_const(ta):
    ta = RIF( ta )
    result = (2*exp_int(2)-1)/2*(1+1/exp_int(pi_int*ta))
    return RIF(result)
def Fu_E_const(ta):
    ta = RIF( ta )
    result = Exf(1/2,ta)
    return RIF(result)
def Fu_H_1(ta):
    ta = RIF( ta )
    result = Fu_H_const(ta)*2/(2*exp_int(2)-1)
    return RIF(result)
def Fu_H_2(ta):
    ta = RIF( ta )
    result = Fu_H_const(ta)*1/(48*exp_int(2))
    return RIF(result)
def Fu_H_3(ta):
    ta = RIF( ta )
    result = Fu_H_const(ta)*Fu_E_const(ta)/(2*exp_int(2))
    return RIF(result)
def Fu_H_4(simin,simax,ta):
    simin = RIF( simin )
    simax = RIF( simax )
    ta = RIF( ta )
    result = 1/2*(1+1/exp_int(pi_int*ta))*(1+Exf(1-simin,ta)/ta^2)
    return RIF(result)
def Fu_Hp_1(ta):
    ta = RIF( ta )
    result = Fu_H_const(ta)*3/(2*exp_int(2)-1)
    return RIF(result)
def Fu_Hp_2(ta):
    ta = RIF( ta )
    result = Fu_H_const(ta)*1/(2*exp_int(2)-1)
    return RIF(result)
def Fu_Hp_3(ta):
    ta = RIF( ta )
    result = Fu_H_const(ta)*1/(48*exp_int(2))
    return RIF(result)
###
# Constants G,H,H'.
G_1 = roundup( Fu_G_1(T0) , digits )
G_2 = roundup( Fu_G_2(T0) , digits )
H_1 = roundup( Fu_H_1(T0) , digits )
H_2 = roundup( Fu_H_2(T0) , digits )
H_3 = roundup( Fu_H_3(T0) , digits )
H_4_i0 = roundup( Fu_H_4(0,1/2,T0) , digits )
H_4_i4 = roundup( Fu_H_4(1/4,1/2,T0) , digits )
Hp_1 = roundup( Fu_Hp_1(T0) , digits )
Hp_2 = roundup( Fu_Hp_2(T0) , digits )
Hp_3 = roundup( Fu_Hp_3(T0) , digits )
\end{mysage}
\normalsize

\begin{lemma}\label{In:estimation}
Let $T\geq T_0=\sage{T0}$. For any $n\in\mathbb{Z}_{>0}$ such that $n\leq\frac{T}{2\pi}$, we have the following estimation:
\begin{equation*}
I_n=2\pi+O^{*}\left(2|A_n|+2|B_n|+2|C_n|\right),
\end{equation*}
where
\begin{align*}
|A_n| & \leq\frac{G_{1}}{\sqrt{n}}+\frac{G_{2}}{\sqrt{T}}, & & \begin{aligned} G_{1} & =\sage{G_1}, \\ G_{2} & =\sage{G_2}, \end{aligned}
\end{align*}
and, if $n\leq\frac{T-\sqrt{T}}{2\pi}$, 
\begin{align}
|B_n| & \leq\frac{H_{1}}{\sqrt{n}\log\left(\frac{T}{2\pi n}\right)}+\frac{H_{2}}{T\sqrt{n}\log\left(\frac{T}{2\pi n}\right)}+\frac{H_{3}}{T\sqrt{n}}, & & \begin{aligned} H_{1} & =\sage{H_1}, & H_{2} & =\sage{H_2}, \\ H_{3} & =\sage{H_3}, & & \end{aligned} \label{Bn:1} \\
|C_n| & \leq\frac{H_{4}}{(2\pi)^{\frac{1}{2}-\tau}}\frac{T^{\frac{1}{2}-\tau}-(2\pi n)^{\frac{1}{2}-\tau}}{n^{1-\tau}\log\left(\frac{T}{2\pi n}\right)}, & & H_{4}=\sage{H_4_i4}; \label{Cn:1}
\end{align} 
otherwise, if $\frac{T-\sqrt{T}}{2\pi}<n\leq\frac{T}{2\pi}$,
\begin{align}
|B_n| & \leq\frac{H'_{1}\sqrt{T}}{\sqrt{n}}+\frac{H'_{2}}{\sqrt{n}}+\frac{H'_{3}}{\sqrt{T}\sqrt{n}}+\frac{H'_{4}}{T\sqrt{n}}, & & \begin{aligned} H'_{1} & =\sage{Hp_1}, & H'_{2} & =\sage{Hp_2}, \\ H'_{3} & =\sage{Hp_3}, & H'_{4} & =H_{3}, \end{aligned} \label{Bn:2} \\
|C_n| & \leq\left(\frac{1}{2}-\tau\right)\frac{H'_{5}}{(2\pi)^{\frac{1}{2}-\tau}}\frac{T^{\frac{1}{2}-\tau}}{n^{1-\tau}}, & & H'_{5}=H_4. \label{Cn:2}
\end{align}
\end{lemma}

\begin{remark}\label{crucial}
Observe that \eqref{Bn:1} and \eqref{Cn:1} are not as sharp as $n$ approaches $\frac{T}{2\pi}$. Indeed, if $2\pi n>T-\sqrt{T}$ then, by Lemma~\ref{logg}\eqref{logg1}, 
\begin{align*}\frac{1}{\log\left(\frac{T}{2\pi n}\right)}\geq\frac{2\pi n}{T-2\pi n}\gg\sqrt{T},\end{align*} 
and thus it is better to consider the bounds \eqref{Bn:2} and  \eqref{Cn:2}, respectively. Instead, if $2\pi n\leq T-\sqrt{T}$, we have that 
\begin{align*}\frac{1}{\log\left(\frac{T}{2\pi n}\right)}\leq\frac{T}{T-2\pi n}\ll\sqrt{T},\end{align*} 
and thus it is better to consider the bound \eqref{Bn:1} over the one given in \eqref{Bn:2}; moreover, as 
\begin{align*}\frac{T^{\frac{1}{2}-\tau}-(2\pi n)^{\frac{1}{2}-\tau}}{\log\left(\frac{T}{2\pi n}\right)}\ll_{\tau} T^{\frac{1}{2}-\tau},\end{align*} 
it is also better in this case to consider the bound $\eqref{Cn:1}$ over $\eqref{Cn:2}$.
\end{remark}

\begin{proof}
Let $U>T$. By using \eqref{functional}, the residue theorem and Theorem \ref{Stirling}\eqref{StirlingGm} we conclude, as in \cite[\S 7.4]{Tit86}, that for any $U>0$ we have $2\pi = I_{n} +A_n-\overline{A_n}+B_n-\overline{B_n}+C_n-\overline{C_n}$, where
\begin{align*}
A_n&=\frac{1}{i}\int_{-\infty-iU}^{\frac{1}{2}-iU}\Gamma(s)\cos\left(\frac{\pi s}{2}\right)(2\pi n)^{-s}ds,\\
B_n&=\frac{1}{i}\int_{\frac{1}{2}-iU}^{\frac{1}{2}-iT}\Gamma(s)\cos\left(\frac{\pi s}{2}\right)(2\pi n)^{-s}ds,\\
C_n&=\frac{1}{i}\int_{\frac{1}{2}-iT}^{1-\tau-iT}\Gamma(s)\cos\left(\frac{\pi s}{2}\right)(2\pi n)^{-s}ds.
\end{align*}
Furthermore, by selecting $U=2e^2T$ and using that $T\geq T_0$, $A_n$ can be estimated with the help of Lemma \ref{new_bound:1}, giving
\begin{align}\label{An:est1}
&|A_n|\leq
\frac{\sqrt{2\pi}}{2}e^{\frac{1}{48e^2T_0}+\frac{1}{120e^4T_{0}^2}}\left(1+\frac{1}{e^{2\pi e^2T_{0}}}\right)\left(\frac{e}{\sqrt{2\pi n}\log\left(\frac{2e^4T}{2\pi n}\right)}+\frac{1}{\sqrt{2}e\sqrt{T}\log\left(\frac{T}{\pi n}\right)}\right)\nonumber\\
&\leq\frac{\sqrt{2\pi}}{2}e^{\frac{1}{48e^2T_0}+\frac{1}{120e^4T_{0}^2}}\left(1+\frac{1}{e^{2\pi e^2T_{0}}}\right)\left(\frac{e}{\sqrt{2\pi}\log\left(2e^4\right)\sqrt{n}}+\frac{1}{\sqrt{2}e\log\left(2\right)\sqrt{T}}\right),
\end{align}  
where we have used that $n\leq\frac{T}{2\pi}$.

With respect to $B_n$, observe that $\frac{1}{2}-i[U,T]$ is a subset of the angular sector defined by $|\arg(s)|<\frac{\pi}{2}$; hence, we can use Theorem \ref{Stirling}\eqref{StirlingGam}-\eqref{StirlingGaI} with $\sigma\geq 0$, $t<0$ and $\theta=\frac{\pi}{2}$ (so that $F_{\theta}=\frac{1}{90}$), along with the definition of the complex cosine and write
\begin{align*}
\frac{1}{i}\Gamma(s)\cos\left(\frac{\pi s}{2}\right)(2\pi n)^{-s}&=\frac{e^{-\frac{\pi i}{2}}}{2}e^{\log(|\Gamma(s)|)+i\Im(\log(\Gamma(s)))}\left(e^{\frac{\pi}{2}is}+e^{-\frac{\pi}{2}is}\right)e^{-\log(2\pi n)s},
\end{align*} 
which, upon using that $\log(1+x)=x+O^*\left(\frac{x^2}{2}\right)$ for $x>0$, so that $\log(|s|)=\log(|t|)+\frac{\sigma^2}{2t^2}+O^*\left(\frac{\sigma^4}{4t^4}\right)$, that $\frac{1}{|s|^2}=\frac{1}{t^2}+O^*\left(\frac{\sigma^2}{t^4}\right)$ and that $\frac{1}{|s|^3}\leq\frac{1}{|t|^3}$, may be rewritten as 
\begin{align}\label{definitive}
\frac{(2\pi)^{\frac{1}{2}-\sigma}e^{-\frac{\pi i}{4}}}{2n^{\sigma}}|t|^{\sigma-\frac{1}{2}}e^{if(\sigma,t)}(1+e^{-\pi\sigma i}e^{\pi t})e^{\mathrm{r}_1(\sigma,t)+i\mathrm{r}_2(\sigma,t)},
\end{align}
where 
\begin{align}\label{def:e1e2}
f(\sigma,t)&=t\log\left(\frac{|t|}{2\pi n}\right)-t+\left(\frac{\sigma^2}{2}-\sigma\left(\sigma-\frac{1}{2}\right)-\frac{1}{12}\right)\frac{1}{t},\nonumber\\
|\mathrm{r}_1(\sigma,t)|&\leq\mathrm{e}_1(\sigma,t)=\left(\left|\sigma-\frac{1}{2}\right|\left(\frac{\sigma^2}{2}+\frac{\sigma^4}{4t^2}\right)+\frac{\sigma^3}{3}+\frac{\sigma}{12}+\frac{1}{90|t|}+\frac{\sigma^3}{12t^2}\right)\frac{1}{t^2},\nonumber\\
|\mathrm{r}_2(\sigma,t)|&\leq\mathrm{e}_2(\sigma,t)=\left(\frac{\sigma^4}{4}+\left|\sigma-\frac{1}{2}\right|\frac{\sigma^3}{3}+\frac{1}{90}+\frac{\sigma^2}{12}\right)\frac{1}{|t|^3}.
\end{align}
Observe that, if $|t|\geq T\geq T_0$, then, for any $\sigma\geq 0$, we have the following bounds:
\begin{align*}
\mathrm{e}_1(\sigma,t)&\leq\min\left\{\mathrm{e}_1(\sigma,T_0),\frac{\mathrm{e}_1(\sigma,T_0)T_0^2}{t^2}\right\},&\mathrm{e}_2(\sigma,t)&\leq\min\left\{\mathrm{e}_2(\sigma,T_0),\frac{\mathrm{e}_2(\sigma,T_0)T_{0}^2}{t^2}\right\}.
\end{align*} 
Therefore, by the complex series representation of the exponential function and using that $|\mathrm{r}_1(\sigma,t)+i\mathrm{r}_2(\sigma,t)|\leq|\mathrm{r}_1(\sigma,t)|+|\mathrm{r}_2(\sigma,t)|$, we may write
\begin{align}\label{def_bound}
e^{\mathrm{r}_1(\sigma,t)+i\mathrm{r}_2(\sigma,t)}=1+\frac{O^*(\mathrm{E}(\sigma,T_0))}{t^2},
\end{align} 
where
\begin{align}\label{E_def} 
\mathrm{E}(\sigma,T_0)=T_0^2(e^{\mathrm{e}_1(\sigma,T_0)+\mathrm{e}_2(\sigma,T_0)}-1).
\end{align}

On the other hand, with the help of \eqref{definitive} and \eqref{def_bound} with $\sigma=\frac{1}{2}$, and using that $|e^{-\pi\sigma i}e^{\pi t}|\leq e^{-\pi T_0}$ if $t\leq-T\leq -T_0$, we derive
\begin{align}\label{Bn:est}
&|B_n|=\left|\int_{-2e^2T}^{-T}\frac{e^{\frac{\pi i}{4}}}{2\sqrt{n}}e^{i f\left(\frac{1}{2},t\right)}(1+e^{-\frac{\pi i}{2}}e^{\pi t})e^{\mathrm{r}_1\left(\frac{1}{2},t\right)+i\mathrm{r}_2\left(\frac{1}{2},t\right)}dt\right|\nonumber\\
\leq\ &\frac{1}{2\sqrt{n}}\left(1+\frac{1}{e^{\pi T_0}}\right)\left(\left|\int_{T}^{2e^2T}e^{i f\left(\frac{1}{2},-t\right)}dt\right|+\mathrm{E}\left(\frac{1}{2},T_0\right)\left|\int_{T}^{2e^2T}\frac{e^{i f\left(\frac{1}{2},-t\right)}}{t^2}dt\right|\right), 
\end{align}  
where we have used the change of variables $t\leftrightarrow -t$ in both integrals above and where, for $t>0$, $f\left(\frac{1}{2},-t\right)=-t\log\left(\frac{t}{2\pi n}\right)+t-\frac{1}{24t}$. The second integral in \eqref{Bn:est} can be readily bounded by $\left(1-\frac{1}{2e^2}\right)\frac{1}{T}$; concerning the first one, we may write $f\left(\frac{1}{2},-t\right)=g(t)+h(t)$, with $g(t)=-t\log\left(\frac{t}{2\pi n}\right)+t$, $h(t)=-\frac{1}{24t}$. Moreover, $g'(t)\neq 0$ for $t\in(T,\infty)$, since $|g'(t)|=\log\left(\frac{t}{2\pi n}\right)>\log\left(\frac{T}{2\pi n}\right)$ and, by hypothesis, $\frac{T}{2\pi n}\geq 1$. We may use then the following identity:
\begin{align}\label{identity1}
\int_{V}^{W}e^{if\left(\frac{1}{2},-t\right)}dt=\left[\frac{e^{i(g(t)+h(t))}}{g'(t)}\right|_{V}^{W}-\int_{V}^{W}\frac{e^{i(g(t)+h(t))}h'(t)}{g'(t)}-\frac{e^{i(g(t)+h(t))}g''(t)}{g'(t)^2}dt,
\end{align}
valid for any $V,W$ such that $W>V$ and $n<\frac{V}{2\pi}$, and derive
\begin{align}\label{Bn:identity}
\left|\int_{V}^{W}e^{if\left(\frac{1}{2},-t\right)}dt\right|&\leq\frac{1}{\log\left(\frac{W}{2\pi n}\right)}+\frac{1}{\log\left(\frac{V}{2\pi n}\right)}+\int_{V}^{W}\frac{1}{24t^2\log\left(\frac{t}{2\pi n}\right)}+\frac{1}{t\log^2\left(\frac{t}{2\pi n}\right)}dt\nonumber\\
&\leq\frac{2}{\log\left(\frac{V}{2\pi n}\right)}+\frac{1}{24}\left(\frac{1}{V}-\frac{1}{W}\right)\frac{1}{\log\left(\frac{V}{2\pi n}\right)}.
\end{align} 
Thus, by using \eqref{Bn:identity} with $V=T$, $W=2e^2T$ and $n\leq\frac{T-\sqrt{T}}{2\pi}<\frac{T}{2\pi}$, we derive from \eqref{Bn:est} that $|B_n|$ is bounded by
\begin{align}\label{Bn:est1}
\frac{2e^2-1}{2\sqrt{n}}\left(1+\frac{1}{e^{\pi T_0}}\right)\left(\left(\frac{2}{2e^{2}-1}+\frac{1}{48e^{2}T}\right)\frac{1}{\log\left(\frac{T}{2\pi n}\right)}+\frac{\mathrm{E}\left(\frac{1}{2},T_0\right)}{2e^2T}\right).
\end{align}
As pointed out in Remark \ref{crucial}, we adopt the bound \eqref{Bn:est1} only for $n\leq\frac{T-\sqrt{T}}{2\pi}$ since, otherwise, it becomes too big in magnitude. When $\frac{T-\sqrt{T}}{2\pi}<n\leq\frac{T}{2\pi}$, we can obtain a better estimation. Indeed, by Lemma~\ref{logg}\eqref{logg2}, we have 
\begin{align*}
\frac{1}{\log\left(\frac{T+\sqrt{T}}{2\pi n}\right)}\leq\frac{1}{\log\left(\frac{T+\sqrt{T}}{T}\right)}\leq\sqrt{T}+\frac{1}{2}\end{align*}
 and we may derive that
\begin{align*}
&\phantom{xxxxxx}\left|\int_{T}^{2e^2T}e^{i f\left(\frac{1}{2},-t\right)}dt\right|\leq\left|\int_{T+\sqrt{T}}^{2e^2T}e^{i f\left(\frac{1}{2},-t\right)}dt\right|+\left|\int_{T}^{T+\sqrt{T}}e^{i f\left(\frac{1}{2},-t\right)}dt\right|\\
&\leq\left(2+\frac{1}{24}\left(\frac{(2e^2-1)\sqrt{T}-1}{2e^2(\sqrt{T}+1)T}\right)\right)\left(\sqrt{T}+\frac{1}{2}\right)+\sqrt{T}\leq 3\sqrt{T}+1+\frac{1}{24}\left(1-\frac{1}{2e^2}\right)\frac{1}{\sqrt{T}},
\end{align*}
using \eqref{Bn:identity} with $V=T+\sqrt{T}$, $W=2e^2T$. Therefore $|B_n|$ is also bounded by
\begin{align}\label{Bn:est2}
\frac{2e^2-1}{2\sqrt{n}}\left(1+\frac{1}{e^{\pi T_0}}\right)\left(\frac{3\sqrt{T}+1}{2e^{2}-1}+\frac{1}{48e^{2}\sqrt{T}}+\frac{\mathrm{E}\left(\frac{1}{2},T_0\right)}{2e^2T}\right).
\end{align}

With respect to $C_n$, observe that $[\frac{1}{2},1-\tau]-iT$ is also a subset of the angular sector defined by $|\arg(s)|<\frac{\pi}{2}$. By recalling \eqref{definitive}, \eqref{def_bound} and the fact that the function $\sigma\mapsto\mathrm{E}(\sigma,t)$ is increasing for $\sigma>\frac{1}{2}$, so that $\mathrm{E}(1-\tau,T_0)\leq\mathrm{E}\left(\frac{3}{4},T_0\right)$, we obtain that for any $s\in[\frac{1}{2},1-\tau]-iT$,
\begin{align}\label{definitive2}
\left|\frac{1}{i}\Gamma(s)\cos\left(\frac{\pi s}{2}\right)(2\pi n)^{-s}\right|\leq\left(1+\frac{1}{e^{\pi T_0}}\right)\left(1+\frac{\mathrm{E}\left(\frac{3}{4},T_0\right)}{T_0^2}\right)\sqrt{\frac{\pi }{2T}}\left(\frac{T}{2\pi n}\right)^{\sigma},
\end{align}
so that, upon integrating the bound given by \eqref{definitive2} on the variable $\sigma$, we obtain that $|C_n|$ is at most either
\begin{align}
& \left(1+\frac{1}{e^{\pi T_0}}\right)\left(1+\frac{\mathrm{E}\left(\frac{3}{4},T_0\right)}{T_0^2}\right)\frac{1}{2\sqrt{n}}\frac{\left(\frac{T}{2\pi n}\right)^{\frac{1}{2}-\tau}-1}{\log\left(\frac{T}{2\pi n}\right)}, & & \text{ if }n\leq\frac{T-\sqrt{T}}{2\pi},\text{ or } \label{Cn:est1} \\
& \left(\frac{1}{2}-\tau\right)\left(1+\frac{1}{e^{\pi T_0}}\right)\left(1+\frac{\mathrm{E}\left(\frac{3}{4},T_0\right)}{T_0^2}\right)\frac{1}{2\sqrt{n}}\left(\frac{T}{2\pi n}\right)^{\frac{1}{2}-\tau}, & & \text{ if }\frac{T-\sqrt{T}}{2\pi}<n\leq\frac{T}{2\pi}. \label{Cn:est2}
\end{align}
The bound \eqref{Cn:est2} is correct since, whenever $n\leq \frac{T}{2\pi}$, the function $\sigma\mapsto\left(\frac{T}{2\pi n}\right)^{\sigma}$ is increasing. Furthermore, observe that \eqref{Cn:est1} and \eqref{Cn:est2} vanish  when $\tau=\frac{1}{2}$. Subsequently, we define 
\begin{align*}
H_{4}=\frac{1}{2}\left(1+\frac{1}{e^{\pi T_0}}\right)\left(1+\frac{\mathrm{E}\left(\frac{3}{4},T_0\right)}{T_0^2}\right).
\end{align*}

Finally, in order to derive the constants of the statement, we combine and evaluate either estimations \eqref{An:est1}, \eqref{Bn:est1} and \eqref{Cn:est1}, if $n\leq\frac{T-\sqrt{T}}{2\pi}$, or estimations \eqref{An:est1}, \eqref{Bn:est2} and \eqref{Cn:est2}, when $\frac{T-\sqrt{T}}{2\pi}<n\leq\frac{T}{2\pi}$. 
\end{proof}

Remark~\ref{crucial} was well pointed out in \cite{Atk39}, by means of which it was possible to obtain an error term for $I$, defined in \eqref{I}, of order $\sqrt{T}\log^2(T)$; as it turns out, the proof we present, inspired in part by \cite[\S 7.4]{Tit86} by means of equation \eqref{def_bound}, allows us to improve the error term magnitude to $\sqrt{T}\log(T)$, presented in Proposition~\ref{pr:I} and proved below.

\begin{proofbold}{Proposition~\ref{pr:I}}

\footnotesize
\begin{mysage}
###
# Estimating I.
###
# Constants multiplying the sums in {eq:alot}.
Ieta_out_0 = RIF( 2*pi_int )
def Fu_Ieta_1_0(ta):
    ta = RIF( ta )
    result = 2*Fu_G_1(ta)
    return RIF(result)
def Fu_Ieta_1_m1(ta):
    ta = RIF( ta )
    result = 2*Fu_H_3(ta)
    return RIF(result)
def Fu_Ieta_2_m12(ta):
    ta = RIF( ta )
    result = 2*Fu_G_2(ta)
    return RIF(result)
def Fu_Ieta_3_1(ta):
    ta = RIF( ta )
    result = 2*Fu_H_1(ta)
    return RIF(result)
def Fu_Ieta_3_0(ta):
    ta = RIF( ta )
    result = 2*Fu_H_2(ta)
    return RIF(result)
def Fu_Ieta_4_32mt(simin,simax,ta):
    simin = RIF( simin )
    simax = RIF( simax )
    ta = RIF( ta )
    result = 2*Fu_H_4(simin,simax,ta)/exp_int(log_int(2*pi_int)*(1/2-simax))
    return RIF(result)
def Fu_Ieta_5_12(ta):
    ta = RIF( ta )
    result = 2*Fu_Hp_1(ta)
    return RIF(result)
def Fu_Ieta_5_0(ta):
    ta = RIF( ta )
    result = 2*Fu_Hp_2(ta)
    return RIF(result)
def Fu_Ieta_5_m12(ta):
    ta = RIF( ta )
    result = 2*Fu_Hp_3(ta)
    return RIF(result)
def Fu_Ieta_6_12mt(simin,simax,ta):
    simin = RIF( simin )
    simax = RIF( simax )
    ta = RIF( ta )
    result = 2*(1/2-simin)*Fu_H_4(simin,simax,ta)/exp_int(log_int(2*pi_int)*(1/2-simax))
    return RIF(result)
###
# All the following coefficients per order inside the sums
# are written as constants multiplying [1,1/tau,1/(1/2-tau)].
###
# Coefficients per order inside the outer sum in {eq:alot}.
def Fu_It_out_1mt(simin,simax):
    simin = RIF( simin )
    simax = RIF( simax )
    result = cw_prod( 1/exp_int(log_int(2*pi_int)*(1-simax)) , Fu_As(simin,simax) )
    return result
I0_out_1mt = Fu_It_out_1mt(0,1/2)
I4_out_1mt = Fu_It_out_1mt(1/4,1/2)
I2_out_12 = RIF( A12/sqrt_int(2*pi_int) )
###
# Coefficients per order inside the first sum in {eq:alot}.
def Fu_It_1_12(simin,simax):
    simin = RIF( simin )
    simax = RIF( simax )
    result = cw_prod( 1/sqrt_int(2*pi_int) , zeb_2s )
    return result
I0_1_12 = Fu_It_1_12(0,1/2)
I4_1_12 = Fu_It_1_12(1/4,1/2)
def Fu_It_1_32m2t(simin,simax):
    simin = RIF( simin )
    simax = RIF( simax )
    result = cw_prod( 1/((3/2-2*simax)*exp_int(log_int(2*pi_int)*(3/2-2*simax))) , zea_2m2s )
    return result
I0_1_32m2t = Fu_It_1_32m2t(0,1/2)
I4_1_32m2t = Fu_It_1_32m2t(1/4,1/2)
def Fu_It_1_12mt(simin,simax):
    simin = RIF( simin )
    simax = RIF( simax )
    result = cw_prod( 1/exp_int(log_int(2*pi_int)*(1/2-simax)) , Fu_C1s(simin,simax) )
    return result
I0_1_12mt = Fu_It_1_12mt(0,1/2)
I4_1_12mt = Fu_It_1_12mt(1/4,1/2)
I2_1_12l = RIF( 2/sqrt_int(2*pi_int) )
def Fu_I2_1_l(ta):
    return Fu_C12(ta)
###
# Coefficients per order inside the second sum in {eq:alot}.
def Fu_It_2_1(simin,simax):
    simin = RIF( simin )
    simax = RIF( simax )
    result = cw_prod( 1/(2*pi_int) , zeb_2s )
    return result
I0_2_1 = Fu_It_2_1(0,1/2)
I4_2_1 = Fu_It_2_1(1/4,1/2)
def Fu_It_2_2m2t(simin,simax):
    simin = RIF( simin )
    simax = RIF( simax )
    result = cw_prod( 1/(exp_int(log_int(2*pi_int)*(2-2*simax))*(2-2*simax)) , zea_2m2s )
    return result
I0_2_2m2t = Fu_It_2_2m2t(0,1/2)
I4_2_2m2t = Fu_It_2_2m2t(1/4,1/2)
def Fu_It_2_1mt(simin,simax):
    simin = RIF( simin )
    simax = RIF( simax )
    result = cw_prod( 1/exp_int(log_int(2*pi_int)*(1-simax)) , Fu_As(simin,simax) )
    return result
I0_2_1mt = Fu_It_2_1mt(0,1/2)
I4_2_1mt = Fu_It_2_1mt(1/4,1/2)
I2_2_1l = RIF( 1/(2*pi_int) )
I2_2_1 = RIF( (2*gamma_int-1)/(2*pi_int) )
I2_2_12 = RIF( A12/sqrt_int(2*pi_int) )
###
# Coefficients per order inside the third sum in {eq:alot}.
I0_3_12m2tl = F_1_s2_i0
I4_3_12m2tl = F_1_s2_i4
I0_3_12m2t = F_2_s2_i0
I4_3_12m2t = F_2_s2_i4
I0_3_mt = F_3_s2_i0
I4_3_mt = F_3_s2_i4
I0_3_m12mtl = F_4_s2_i0
I4_3_m12mtl = F_4_s2_i4
I0_3_m12mt = F_5_s2_i0
I4_3_m12mt = F_5_s2_i4
def Fu_I0_3_m1(ta):
    return Fu_F_6_s2_i0(ta)
def Fu_I4_3_m1(ta):
    return Fu_F_6_s2_i4(ta)
I2_3_m12l = F_1_22
I2_3_m12 = F_2_22 + F_3_22
I2_3_m1l = F_4_22
def Fu_I2_3_m1(ta):
    ta = RIF( ta )
    result = F_5_22 + Fu_F_6_22(ta)
    return RIF(result)
###
# Coefficients per order inside the fourth sum in {eq:alot}.
I0_4_mtl = F_1_s1_i0
I4_4_mtl = F_1_s1_i4
I0_4_mt = F_2_s1_i0
I4_4_mt = F_2_s1_i4
I0_4_m12 = F_3_s1_i0
I4_4_m12 = F_3_s1_i4
I0_4_m1l = F_4_s1_i0
I4_4_m1l = F_4_s1_i4
def Fu_I0_4_m1(ta):
    ta = RIF( ta )
    result = cw_sum( F_5_s1_i0 , Fu_F_6_s1_i0(ta) )
    return result
def Fu_I4_4_m1(ta):
    ta = RIF( ta )
    result = cw_sum( F_5_s1_i4 , Fu_F_6_s1_i4(ta) )
    return result
###
# Coefficients per order inside the fifth sum in {eq:alot}.
def Fu_It_5_0(simin,simax,ta):
    simin = RIF( simin )
    simax = RIF( simax )
    ta = RIF( ta )
    result = cw_prod( 1/exp_int(log_int(1-1/sqrt_int(ta))*1/2)*1/sqrt_int(2*pi_int) , zeb_2s )
    return result
def Fu_It_5_1m2t(simin,simax,ta):
    simin = RIF( simin )
    simax = RIF( simax )
    ta = RIF( ta )
    result = cw_prod( 1/exp_int(log_int(1-1/sqrt_int(ta))*1/2)*1/exp_int(log_int(2*pi_int)*(1/2-2*simax)) , zea_2m2s )
    return result
def Fu_It_5_12mt(simin,simax,ta):
    simin = RIF( simin )
    simax = RIF( simax )
    ta = RIF( ta )
    result = cw_prod( 1/exp_int(log_int(1-1/sqrt_int(ta))*1/2)*2/exp_int(log_int(2*pi_int)*(1/2-simax)) , Fu_As(simin,simax) )
    return result
def Fu_I0_5_0(ta):
    return Fu_It_5_0(0,1/2,ta)
def Fu_I4_5_0(ta):
    return Fu_It_5_0(1/4,1/2,ta)
def Fu_I0_5_1m2t(ta):
    return Fu_It_5_1m2t(0,1/2,ta)
def Fu_I4_5_1m2t(ta):
    return Fu_It_5_1m2t(1/4,1/2,ta)
def Fu_I0_5_12mt(ta):
    return Fu_It_5_12mt(0,1/2,ta)
def Fu_I4_5_12mt(ta):
    return Fu_It_5_12mt(1/4,1/2,ta)
def Fu_I2_5_l(ta):
    ta = RIF( ta )
    result = 1/(sqrt_int(2*pi_int)*sqrt_int(1-1/sqrt_int(ta)))
    return RIF(result)
def Fu_I2_5_0(ta):
    ta = RIF( ta )
    result = (1/sqrt_int(2*pi_int)*(sqrt_int(ta)/(sqrt_int(ta)-1) \
              -log_int(2*pi_int)+2*gamma_int-1)+2*A12)/sqrt_int(1-1/sqrt_int(ta))
    return RIF(result)
###
# Coefficients per order inside the sixth sum in {eq:alot}.
def Fu_It_6_tm12(simin,simax,ta):
    simin = RIF( simin )
    simax = RIF( simax )
    ta = RIF( ta )
    result = cw_prod( 1/exp_int(log_int(1-1/sqrt_int(ta))*(1-simax))*1/exp_int(log_int(2*pi_int)*simin) , zeb_2s )
    return result
def Fu_It_6_12mt(simin,simax,ta):
    simin = RIF( simin )
    simax = RIF( simax )
    ta = RIF( ta )
    result = cw_prod( 1/exp_int(log_int(1-1/sqrt_int(ta))*(1-simax))*exp_int(log_int(2*pi_int)*simax) , zea_2m2s )
    return result
def Fu_It_6_0(simin,simax,ta):
    simin = RIF( simin )
    simax = RIF( simax )
    ta = RIF( ta )
    result = cw_prod( 1/exp_int(log_int(1-1/sqrt_int(ta))*(1-simax))*2 , Fu_As(simin,simax) )
    return result
def Fu_I0_6_tm12(ta):
    return Fu_It_6_tm12(0,1/2,ta)
def Fu_I4_6_tm12(ta):
    return Fu_It_6_tm12(1/4,1/2,ta)
def Fu_I0_6_12mt(ta):
    return Fu_It_6_12mt(0,1/2,ta)
def Fu_I4_6_12mt(ta):
    return Fu_It_6_12mt(1/4,1/2,ta)
def Fu_I0_6_0(ta):
    return Fu_It_6_0(0,1/2,ta)
def Fu_I4_6_0(ta):
    return Fu_It_6_0(1/4,1/2,ta)
###
# Coefficients per order of the whole I, for tau in (0,1/2).
def Fu_I0__32m2tl(ta):
    ta = RIF( ta )
    x3 = cw_prod( Fu_Ieta_3_1(ta) , I0_3_12m2tl )
    x4 = cw_prod( Fu_Ieta_4_32mt(0,1/2,ta) , I0_4_mtl )
    x = cw_sum( x3 , x4 )
    return x
def Fu_I0__32m2t(ta):
    ta = RIF( ta )
    x1 = cw_prod( Fu_Ieta_1_0(ta) , I0_1_32m2t )
    x2 = cw_prod( Fu_Ieta_2_m12(ta) , I0_2_2m2t )
    x3 = cw_prod( Fu_Ieta_3_1(ta) , I0_3_12m2t )
    x4 = cw_prod( Fu_Ieta_4_32mt(0,1/2,ta) , I0_4_mt )
    x5 = cw_prod( Fu_Ieta_5_12(ta) , Fu_I0_5_1m2t(ta) )
    x = cw_sum( x1 , x2 )
    x = cw_sum( x , x3 )
    x = cw_sum( x , x4 )
    x = cw_sum( x , x5 )
    return x
def Fu_I0__1mt(ta):
    ta = RIF( ta )
    xo = cw_prod( Ieta_out_0 , I0_out_1mt )
    x3 = cw_prod( Fu_Ieta_3_1(ta) , I0_3_mt )
    x4 = cw_prod( Fu_Ieta_4_32mt(0,1/2,ta) , I0_4_m12 )
    x5 = cw_prod( Fu_Ieta_5_12(ta) , Fu_I0_5_12mt(ta) )
    x = cw_sum( xo , x3 )
    x = cw_sum( x , x4 )
    x = cw_sum( x , x5 )
    return x
def Fu_I0__1m2t(ta):
    ta = RIF( ta )
    x5 = cw_prod( Fu_Ieta_5_0(ta) , Fu_I0_5_1m2t(ta) )
    x6 = cw_prod( Fu_Ieta_6_12mt(0,1/2,ta) , Fu_I0_6_12mt(ta) )
    x = cw_sum( x5 , x6 )
    return x
def Fu_I0__12(ta):
    ta = RIF( ta )
    x1 = cw_prod( Fu_Ieta_1_0(ta) , I0_1_12 )
    x2 = cw_prod( Fu_Ieta_2_m12(ta) , I0_2_1 )
    x5 = cw_prod( Fu_Ieta_5_12(ta) , Fu_I0_5_0(ta) )
    x = cw_sum( x1 , x2 )
    x = cw_sum( x , x5 )
    return x
def Fu_I0__12mtl(ta):
    ta = RIF( ta )
    x3 = cw_prod( Fu_Ieta_3_1(ta) , I0_3_m12mtl )
    x4 = cw_prod( Fu_Ieta_4_32mt(0,1/2,ta) , I0_4_m1l )
    x = cw_sum( x3 , x4 )
    return x
def Fu_I0__12mt(ta):
    ta = RIF( ta )
    x1 = cw_prod( Fu_Ieta_1_0(ta) , I0_1_12mt )
    x2 = cw_prod( Fu_Ieta_2_m12(ta) , I0_2_1mt )
    x3 = cw_prod( Fu_Ieta_3_1(ta) , I0_3_m12mt )
    x4 = cw_prod( Fu_Ieta_4_32mt(0,1/2,ta) , Fu_I0_4_m1(ta) )
    x5 = cw_prod( Fu_Ieta_5_0(ta) , Fu_I0_5_12mt(ta) )
    x6 = cw_prod( Fu_Ieta_6_12mt(0,1/2,ta) , Fu_I0_6_0(ta) )
    x = cw_sum( x1 , x2 )
    x = cw_sum( x , x3 )
    x = cw_sum( x , x4 )
    x = cw_sum( x , x5 )
    x = cw_sum( x , x6 )
    return x
def Fu_I0__0(ta):
    ta = RIF( ta )
    x3 = cw_prod( Fu_Ieta_3_1(ta) , Fu_I0_3_m1(ta) )
    x5 = cw_prod( Fu_Ieta_5_0(ta) , Fu_I0_5_0(ta) )
    x6 = cw_prod( Fu_Ieta_6_12mt(0,1/2,ta) , Fu_I0_6_tm12(ta) )
    x = cw_sum( x3 , x5 )
    x = cw_sum( x , x6 )
    return x
def Fu_I0__12m2tl(ta):
    ta = RIF( ta )
    x3 = cw_prod( Fu_Ieta_3_0(ta) , I0_3_12m2tl )
    return x3
def Fu_I0__12m2t(ta):
    ta = RIF( ta )
    x1 = cw_prod( Fu_Ieta_1_m1(ta) , I0_1_32m2t )
    x3 = cw_prod( Fu_Ieta_3_0(ta) , I0_3_12m2t )
    x5 = cw_prod( Fu_Ieta_5_m12(ta) , Fu_I0_5_1m2t(ta) )
    x = cw_sum( x1 , x3 )
    x = cw_sum( x , x5 )
    return x
def Fu_I0__mt(ta):
    ta = RIF( ta )
    x3 = cw_prod( Fu_Ieta_3_0(ta) , I0_3_mt )
    x5 = cw_prod( Fu_Ieta_5_m12(ta) , Fu_I0_5_12mt(ta) )
    x = cw_sum( x3 , x5 )
    return x
def Fu_I0__m12(ta):
    ta = RIF( ta )
    x1 = cw_prod( Fu_Ieta_1_m1(ta) , I0_1_12 )
    x5 = cw_prod( Fu_Ieta_5_m12(ta) , Fu_I0_5_0(ta) )
    x = cw_sum( x1 , x5 )
    return x
def Fu_I0__m12mtl(ta):
    ta = RIF( ta )
    x3 = cw_prod( Fu_Ieta_3_0(ta) , I0_3_m12mtl )
    return x3
def Fu_I0__m12mt(ta):
    ta = RIF( ta )
    x1 = cw_prod( Fu_Ieta_1_m1(ta) , I0_1_12mt )
    x3 = cw_prod( Fu_Ieta_3_0(ta) , I0_3_m12mt )
    x = cw_sum( x1 , x3 )
    return x
def Fu_I0__m1(ta):
    ta = RIF( ta )
    x3 = cw_prod( Fu_Ieta_3_0(ta) , Fu_I0_3_m1(ta) )
    return x3
###
# Coefficients per order of the whole I, for tau in [1/4,1/2).
def Fu_I4__32m2tl(ta):
    ta = RIF( ta )
    x3 = cw_prod( Fu_Ieta_3_1(ta) , I4_3_12m2tl )
    x4 = cw_prod( Fu_Ieta_4_32mt(1/4,1/2,ta) , I4_4_mtl )
    x = cw_sum( x3 , x4 )
    return x
def Fu_I4__32m2t(ta):
    ta = RIF( ta )
    x1 = cw_prod( Fu_Ieta_1_0(ta) , I4_1_32m2t )
    x2 = cw_prod( Fu_Ieta_2_m12(ta) , I4_2_2m2t )
    x3 = cw_prod( Fu_Ieta_3_1(ta) , I4_3_12m2t )
    x4 = cw_prod( Fu_Ieta_4_32mt(1/4,1/2,ta) , I4_4_mt )
    x5 = cw_prod( Fu_Ieta_5_12(ta) , Fu_I4_5_1m2t(ta) )
    x = cw_sum( x1 , x2 )
    x = cw_sum( x , x3 )
    x = cw_sum( x , x4 )
    x = cw_sum( x , x5 )
    return x
def Fu_I4__1mt(ta):
    ta = RIF( ta )
    xo = cw_prod( Ieta_out_0 , I4_out_1mt )
    x3 = cw_prod( Fu_Ieta_3_1(ta) , I4_3_mt )
    x4 = cw_prod( Fu_Ieta_4_32mt(1/4,1/2,ta) , I4_4_m12 )
    x5 = cw_prod( Fu_Ieta_5_12(ta) , Fu_I4_5_12mt(ta) )
    x = cw_sum( xo , x3 )
    x = cw_sum( x , x4 )
    x = cw_sum( x , x5 )
    return x
def Fu_I4__1m2t(ta):
    ta = RIF( ta )
    x5 = cw_prod( Fu_Ieta_5_0(ta) , Fu_I4_5_1m2t(ta) )
    x6 = cw_prod( Fu_Ieta_6_12mt(1/4,1/2,ta) , Fu_I4_6_12mt(ta) )
    x = cw_sum( x5 , x6 )
    return x
def Fu_I4__12(ta):
    ta = RIF( ta )
    x1 = cw_prod( Fu_Ieta_1_0(ta) , I4_1_12 )
    x2 = cw_prod( Fu_Ieta_2_m12(ta) , I4_2_1 )
    x5 = cw_prod( Fu_Ieta_5_12(ta) , Fu_I4_5_0(ta) )
    x = cw_sum( x1 , x2 )
    x = cw_sum( x , x5 )
    return x
def Fu_I4__12mtl(ta):
    ta = RIF( ta )
    x3 = cw_prod( Fu_Ieta_3_1(ta) , I4_3_m12mtl )
    x4 = cw_prod( Fu_Ieta_4_32mt(1/4,1/2,ta) , I4_4_m1l )
    x = cw_sum( x3 , x4 )
    return x
def Fu_I4__12mt(ta):
    ta = RIF( ta )
    x1 = cw_prod( Fu_Ieta_1_0(ta) , I4_1_12mt )
    x2 = cw_prod( Fu_Ieta_2_m12(ta) , I4_2_1mt )
    x3 = cw_prod( Fu_Ieta_3_1(ta) , I4_3_m12mt )
    x4 = cw_prod( Fu_Ieta_4_32mt(1/4,1/2,ta) , Fu_I4_4_m1(ta) )
    x5 = cw_prod( Fu_Ieta_5_0(ta) , Fu_I4_5_12mt(ta) )
    x6 = cw_prod( Fu_Ieta_6_12mt(1/4,1/2,ta) , Fu_I4_6_0(ta) )
    x = cw_sum( x1 , x2 )
    x = cw_sum( x , x3 )
    x = cw_sum( x , x4 )
    x = cw_sum( x , x5 )
    x = cw_sum( x , x6 )
    return x
def Fu_I4__0(ta):
    ta = RIF( ta )
    x3 = cw_prod( Fu_Ieta_3_1(ta) , Fu_I4_3_m1(ta) )
    x5 = cw_prod( Fu_Ieta_5_0(ta) , Fu_I4_5_0(ta) )
    x6 = cw_prod( Fu_Ieta_6_12mt(1/4,1/2,ta) , Fu_I4_6_tm12(ta) )
    x = cw_sum( x3 , x5 )
    x = cw_sum( x , x6 )
    return x
def Fu_I4__12m2tl(ta):
    ta = RIF( ta )
    x3 = cw_prod( Fu_Ieta_3_0(ta) , I4_3_12m2tl )
    return x3
def Fu_I4__12m2t(ta):
    ta = RIF( ta )
    x1 = cw_prod( Fu_Ieta_1_m1(ta) , I4_1_32m2t )
    x3 = cw_prod( Fu_Ieta_3_0(ta) , I4_3_12m2t )
    x5 = cw_prod( Fu_Ieta_5_m12(ta) , Fu_I4_5_1m2t(ta) )
    x = cw_sum( x1 , x3 )
    x = cw_sum( x , x5 )
    return x
def Fu_I4__mt(ta):
    ta = RIF( ta )
    x3 = cw_prod( Fu_Ieta_3_0(ta) , I4_3_mt )
    x5 = cw_prod( Fu_Ieta_5_m12(ta) , Fu_I4_5_12mt(ta) )
    x = cw_sum( x3 , x5 )
    return x
def Fu_I4__m12(ta):
    ta = RIF( ta )
    x1 = cw_prod( Fu_Ieta_1_m1(ta) , I4_1_12 )
    x5 = cw_prod( Fu_Ieta_5_m12(ta) , Fu_I4_5_0(ta) )
    x = cw_sum( x1 , x5 )
    return x
def Fu_I4__m12mtl(ta):
    ta = RIF( ta )
    x3 = cw_prod( Fu_Ieta_3_0(ta) , I4_3_m12mtl )
    return x3
def Fu_I4__m12mt(ta):
    ta = RIF( ta )
    x1 = cw_prod( Fu_Ieta_1_m1(ta) , I4_1_12mt )
    x3 = cw_prod( Fu_Ieta_3_0(ta) , I4_3_m12mt )
    x = cw_sum( x1 , x3 )
    return x
def Fu_I4__m1(ta):
    ta = RIF( ta )
    x3 = cw_prod( Fu_Ieta_3_0(ta) , Fu_I4_3_m1(ta) )
    return x3
###
# Coefficients per order of the whole I, for tau=1/2.
def Fu_I2__12l(ta):
    ta = RIF( ta )
    result = Fu_Ieta_1_0(ta)*I2_1_12l + Fu_Ieta_2_m12(ta)*I2_2_1l \
             + Fu_Ieta_3_1(ta)*I2_3_m12l + Fu_Ieta_5_12(ta)*Fu_I2_5_l(ta)
    return RIF(result)
def Fu_I2__12(ta):
    ta = RIF( ta )
    result = Ieta_out_0*I2_out_12 + Fu_Ieta_2_m12(ta)*I2_2_1 \
             + Fu_Ieta_3_1(ta)*I2_3_m12 + Fu_Ieta_5_12(ta)*Fu_I2_5_0(ta)
    return RIF(result)
def Fu_I2__l(ta):
    ta = RIF( ta )
    result = Fu_Ieta_1_0(ta)*Fu_I2_1_l(ta) + Fu_Ieta_3_1(ta)*I2_3_m1l \
             + Fu_Ieta_5_0(ta)*Fu_I2_5_l(ta)
    return RIF(result)
def Fu_I2__0(ta):
    ta = RIF( ta )
    result = Fu_Ieta_2_m12(ta)*I2_2_12 + Fu_Ieta_3_1(ta)*Fu_I2_3_m1(ta) \
             + Fu_Ieta_5_0(ta)*Fu_I2_5_0(ta)
    return RIF(result)
def Fu_I2__m12l(ta):
    ta = RIF( ta )
    result = Fu_Ieta_1_m1(ta)*I2_1_12l + Fu_Ieta_3_0(ta)*I2_3_m12l \
             + Fu_Ieta_5_m12(ta)*Fu_I2_5_l(ta)
    return RIF(result)
def Fu_I2__m12(ta):
    ta = RIF( ta )
    result = Fu_Ieta_3_0(ta)*I2_3_m12 + Fu_Ieta_5_m12(ta)*Fu_I2_5_0(ta)
    return RIF(result)
def Fu_I2__m1l(ta):
    ta = RIF( ta )
    result = Fu_Ieta_1_m1(ta)*Fu_I2_1_l(ta) + Fu_Ieta_3_0(ta)*I2_3_m1l
    return RIF(result)
def Fu_I2__m1(ta):
    ta = RIF( ta )
    result = Fu_Ieta_3_0(ta)*Fu_I2_3_m1(ta)
    return RIF(result)
###
# Rounding coefficients to the order T^(3/2-2tau)*log(T) for tau in (0,1/2).
def Fu_I0__32m2tl_final(ta):
    ta = RIF( ta )
    x_32m2tl = Fu_I0__32m2tl(ta)
    x_32m2t = cw_prod( 1/log_int(ta) , Fu_I0__32m2t(ta) )
    x_1mt = cw_prod( 1/log_int(ta) , Fu_I0__1mt(ta) )
    x_1m2t = cw_prod( 1/(log_int(ta)*sqrt_int(ta)) , Fu_I0__1m2t(ta) )
    x_12 = cw_prod( 1/log_int(ta) , Fu_I0__12(ta) )
    x_12mtl = cw_prod( 1/sqrt_int(ta) , Fu_I0__12mtl(ta) )
    x_12mt = cw_prod( 1/(log_int(ta)*sqrt_int(ta)) , Fu_I0__12mt(ta) )
    x_0 = cw_prod( 1/(log_int(ta)*sqrt_int(ta)) , Fu_I0__0(ta) )
    x_12m2tl = cw_prod( 1/ta , Fu_I0__12m2tl(ta) )
    x_12m2t = cw_prod( 1/(log_int(ta)*ta) , Fu_I0__12m2t(ta) )
    x_mt = cw_prod( 1/(log_int(ta)*ta) , Fu_I0__mt(ta) )
    x_m12 = cw_prod( 1/(log_int(ta)*ta) , Fu_I0__m12(ta) )
    x_m12mtl = cw_prod( 1/exp_int(3/2*log_int(ta)) , Fu_I0__m12mtl(ta) )
    x_m12mt = cw_prod( 1/(log_int(ta)*exp_int(3/2*log_int(ta))) , Fu_I0__m12mt(ta) )
    x_m1 = cw_prod( 1/(log_int(ta)*exp_int(3/2*log_int(ta))) , Fu_I0__m1(ta) )
    x = cw_sum( x_32m2tl , x_32m2t )
    x = cw_sum( x , x_1mt )
    x = cw_sum( x , x_1m2t )
    x = cw_sum( x , x_12 )
    x = cw_sum( x , x_12mtl )
    x = cw_sum( x , x_12mt )
    x = cw_sum( x , x_0 )
    x = cw_sum( x , x_12m2tl )
    x = cw_sum( x , x_12m2t )
    x = cw_sum( x , x_mt )
    x = cw_sum( x , x_m12 )
    x = cw_sum( x , x_m12mtl )
    x = cw_sum( x , x_m12mt )
    x = cw_sum( x , x_m1 )
    return x
###
# Rounding coefficients to the order T^(3/2-2tau)*log(T) for tau in [1/4,1/2).
def Fu_I4__32m2tl_final(ta):
    ta = RIF( ta )
    x_32m2tl = Fu_I4__32m2tl(ta)
    x_32m2t = cw_prod( 1/log_int(ta) , Fu_I4__32m2t(ta) )
    x_1mt = cw_prod( 1/log_int(ta) , Fu_I4__1mt(ta) )
    x_1m2t = cw_prod( 1/(log_int(ta)*sqrt_int(ta)) , Fu_I4__1m2t(ta) )
    x_12 = cw_prod( 1/log_int(ta) , Fu_I4__12(ta) )
    x_12mtl = cw_prod( 1/sqrt_int(ta) , Fu_I4__12mtl(ta) )
    x_12mt = cw_prod( 1/(log_int(ta)*sqrt_int(ta)) , Fu_I4__12mt(ta) )
    x_0 = cw_prod( 1/(log_int(ta)*sqrt_int(ta)) , Fu_I4__0(ta) )
    x_12m2tl = cw_prod( 1/ta , Fu_I4__12m2tl(ta) )
    x_12m2t = cw_prod( 1/(log_int(ta)*ta) , Fu_I4__12m2t(ta) )
    x_mt = cw_prod( 1/(log_int(ta)*ta) , Fu_I4__mt(ta) )
    x_m12 = cw_prod( 1/(log_int(ta)*ta) , Fu_I4__m12(ta) )
    x_m12mtl = cw_prod( 1/exp_int(3/2*log_int(ta)) , Fu_I4__m12mtl(ta) )
    x_m12mt = cw_prod( 1/(log_int(ta)*exp_int(3/2*log_int(ta))) , Fu_I4__m12mt(ta) )
    x_m1 = cw_prod( 1/(log_int(ta)*exp_int(3/2*log_int(ta))) , Fu_I4__m1(ta) )
    x = cw_sum( x_32m2tl , x_32m2t )
    x = cw_sum( x , x_1mt )
    x = cw_sum( x , x_1m2t )
    x = cw_sum( x , x_12 )
    x = cw_sum( x , x_12mtl )
    x = cw_sum( x , x_12mt )
    x = cw_sum( x , x_0 )
    x = cw_sum( x , x_12m2tl )
    x = cw_sum( x , x_12m2t )
    x = cw_sum( x , x_mt )
    x = cw_sum( x , x_m12 )
    x = cw_sum( x , x_m12mtl )
    x = cw_sum( x , x_m12mt )
    x = cw_sum( x , x_m1 )
    return x
###
# Rounding coefficients to the order sqrt(T)*log(T) for tau=1/2.
def Fu_I2__12l_final(ta):
    ta = RIF( ta )
    result = Fu_I2__12l(ta) + Fu_I2__12(ta)/log_int(ta) \
             + Fu_I2__l(ta)/sqrt_int(ta) + Fu_I2__0(ta)/(log_int(ta)*sqrt_int(ta)) \
             + Fu_I2__m12l(ta)/ta + Fu_I2__m12(ta)/(log_int(ta)*ta) \
             + Fu_I2__m1l(ta)/exp_int(3/2*log_int(ta)) \
             + Fu_I2__m1(ta)/(log_int(ta)*exp_int(3/2*log_int(ta)))
    return RIF(result)
###
# Actual values for our choice of T0.
I0__32m2tl_final_vec = Fu_I0__32m2tl_final(T0)
I4__32m2tl_final_vec = Fu_I4__32m2tl_final(T0)
I4__32m2tl_final_c = roundup( I4__32m2tl_final_vec[0] \
                              + I4__32m2tl_final_vec[1]/(1/4) , digits )
I4__32m2tl_final_2 = roundup( I4__32m2tl_final_vec[2] , digits )
I2__12l_final = roundup( Fu_I2__12l_final(T0) , digits )
\end{mysage}
\normalsize

We have that $I=\sum_{n\leq\frac{T}{2\pi}}d_{1-2\tau}(n)\ I_{n}$. As per Lemma~\ref{In:estimation} and Remark~\ref{crucial}, we split that sum into two parts, according to whether $n\leq\frac{T-\sqrt{T}}{2\pi}$ or $\frac{T-\sqrt{T}}{2\pi}<n\leq\frac{T}{2\pi}$. For the first interval we use \eqref{Bn:1} and \eqref{Cn:1} with the simplification $\log\left(\frac{T}{2\pi n}\right)\geq\frac{T-2\pi n}{T}$ from Lemma~\ref{logg}\eqref{logg1}, and in the second interval we use \eqref{Bn:2} and \eqref{Cn:2}. Thus, $I$ is equal to
\begin{align}
& 2\pi\!\!\sum_{n\leq\frac{T}{2\pi}}\!\!d_{1-2\tau}(n)+2O^{*}\left(\eta_{1}(T)\!\!\sum_{n\leq\frac{T}{2\pi}}\!\!\frac{d_{1-2\tau}(n)}{\sqrt{n}}+\eta_{2}(T)\!\!\sum_{n\leq\frac{T}{2\pi}}\!\!d_{1-2\tau}(n)\right. \nonumber \\
& +\eta_{3}(T)\!\!\sum_{n\leq\frac{T-\sqrt{T}}{2\pi}}\!\!\frac{d_{1-2\tau}(n)}{\sqrt{n}(T-2\pi n)}+\mathds{1}_{\{\tau<\frac{1}{2}\}}(\tau)\eta_{4}(T,\tau)\!\!\sum_{n\leq\frac{T-\sqrt{T}}{2\pi}}\!\!\frac{d_{1-2\tau}(n)}{n^{1-\tau}(T-2\pi n)} \nonumber \\
& \left.+\eta_{5}(T)\!\!\sum_{\frac{T-\sqrt{T}}{2\pi}<n\leq\frac{T}{2\pi}}\!\!\frac{d_{1-2\tau}(n)}{\sqrt{n}}+\mathds{1}_{\{\tau<\frac{1}{2}\}}(\tau)\eta_{6}(T,\tau)\!\!\sum_{\frac{T-\sqrt{T}}{2\pi}<n\leq\frac{T}{2\pi}}\!\!\frac{d_{1-2\tau}(n)}{n^{1-\tau}}\right), \label{eq:alot}
\end{align}
with
\begin{align*}
\eta_{1}(T) & =G_{1}+\frac{H_{3}}{T}, & \eta_{3}(T) & =H_{1}T+H_{2}, & \eta_{5}(T) & =H'_{1}\sqrt{T}+H'_{2}+\frac{H'_{3}}{\sqrt{T}}, \\
\eta_{2}(T) & =\frac{G_{2}}{\sqrt{T}}, & \eta_{4}(T,\tau) & =\frac{H_{4}T^{\frac{3}{2}-\tau}}{(2\pi)^{\frac{1}{2}-\tau}}, & \eta_{6}(T,\tau) & =\left(\frac{1}{2}-\tau\right)\frac{H_{4}T^{\frac{1}{2}-\tau}}{(2\pi)^{\frac{1}{2}-\tau}}.
\end{align*}

Assume first that $0<\tau<\frac{1}{2}$. For the main term in \eqref{eq:alot}, we use Proposition~\ref{pr:sumdiv}, obtaining
\begin{equation*} 
\sum_{n\leq\frac{T}{2\pi}}d_{1-2\tau}(n)=\frac{\zeta(2\tau)}{2\pi}T+\frac{\zeta(2-2\tau)}{(2-2\tau)(2\pi)^{2-2\tau}}T^{2-2\tau}+O^{*}\left(\frac{A_{\tau}}{(2\pi)^{1-\tau}}T^{1-\tau}\right).
\end{equation*}

The first error term in \eqref{eq:alot} may be bounded by Proposition~\ref{pr:sumdivwei}. For the second error term we use Proposition~\ref{pr:sumdiv} again. For the third and fourth error terms, we apply directly Lemma~\ref{le:sumdivlog} with $\upsilon=\frac{1}{2}$ and $\upsilon=1-\tau$ respectively. Finally, the fifth and sixth summations in \eqref{eq:alot} can be bounded through Proposition~\ref{pr:sumdiv} using
\begin{align*}
 & \ \sum_{\frac{T-\sqrt{T}}{2\pi}<n\leq\frac{T}{2\pi}}\frac{d_{1-2\tau}(n)}{n^{\upsilon}}<\left(\frac{2\pi}{T-\sqrt{T}}\right)^{\upsilon}\sum_{\frac{T-\sqrt{T}}{2\pi}<n\leq\frac{T}{2\pi}}d_{1-2\tau}(n) \\
< & \ \left(\frac{2\pi}{T-\sqrt{T}}\right)^{\upsilon}\left(\zeta(2\tau)\frac{\sqrt{T}}{2\pi}+\frac{\zeta(2-2\tau)}{2-2\tau}\left(\left(\frac{T}{2\pi}\right)^{2-2\tau}\!\!\!\!-\left(\frac{T-\sqrt{T}}{2\pi}\right)^{2-2\tau}\right)+\frac{2A_{\tau}T^{1-\tau}}{(2\pi)^{1-\tau}}\right) \\
< & \ \frac{(2\pi)^{\upsilon}}{\left(1-\frac{1}{\sqrt{T_{0}}}\right)^{\upsilon}T^{\upsilon}}\left(\frac{\zeta(2\tau)}{2\pi}\sqrt{T}+\frac{\zeta(2-2\tau)}{(2\pi)^{1-2\tau}}T^{\frac{3}{2}-2\tau}+\frac{2A_{\tau}}{(2\pi)^{1-\tau}}T^{1-\tau}\right)
\end{align*}
where, as $2-2\tau<1$, we used Lemma~\ref{logg}\eqref{loggexp} as well as the fact that $T\geq T_{0}$. Subsequently, we can combine all lower order terms to the order $T^{\frac{3}{2}-2\tau}\log(T)$.

Assume now that $\tau=\frac{1}{2}$. In this case, the fourth and sixth error term in \eqref{eq:alot} disappear. For the main term in \eqref{eq:alot}, we use Proposition~\ref{pr:sumdiv}, obtaining
\begin{equation*}
\sum_{n\leq\frac{T}{2\pi}}d(n)=\frac{T\log(T)}{2\pi}+\frac{2\gamma-1-\log(2\pi)}{2\pi}T+O^{*}\left(\frac{A_{\frac{1}{2}}}{\sqrt{2\pi}}\sqrt{T}\right).
\end{equation*}

The first error term in \eqref{eq:alot} may be bounded by Proposition~\ref{pr:sumdivwei}. We can forgo the negative term therein since, as $T\geq T_{0}$, this term is smaller in absolute value than the positive one; moreover, we can also bound $\log\left(\frac{T}{2\pi}\right)$ by $\log(T)$.

In order to estimate the second error term in \eqref{eq:alot}, we use Proposition~\ref{pr:sumdiv} and the bound  $\log\left(\frac{T}{2\pi}\right)\leq\log(T)$ again, giving that it is at most
\begin{equation*}
\frac{1}{2\pi}T\log(T)+\frac{2\gamma-1}{2\pi}\ T+\frac{A_{\frac{1}{2}}}{\sqrt{2\pi}}\sqrt{T}.
\end{equation*}

For the third error term, we apply directly Lemma~\ref{le:sumdivlog} with $\sigma=\upsilon=\frac{1}{2}$. Finally, the fifth summation in \eqref{eq:alot} can be bounded through Proposition~\ref{pr:sumdiv} by
\begin{align*}
 & \ \sum_{\frac{T-\sqrt{T}}{2\pi}<n\leq\frac{T}{2\pi}}\frac{d(n)}{\sqrt{n}}<\frac{\sqrt{2\pi}}{\sqrt{1-\frac{1}{\sqrt{T_{0}}}}\sqrt{T}}\sum_{\frac{T-\sqrt{T}}{2\pi}<n\leq\frac{T}{2\pi}}d(n) \\
< & \ \frac{\sqrt{2\pi}}{\sqrt{1-\frac{1}{\sqrt{T_{0}}}}\sqrt{T}}\left(-\frac{T}{2\pi}\log\left(1-\frac{1}{\sqrt{T}}\right)+\frac{\sqrt{T}}{2\pi}\log\left(\frac{T}{2\pi}\right)+\frac{2\gamma-1}{2\pi}\ \sqrt{T}+\frac{2A_{\frac{1}{2}}}{\sqrt{2\pi}}\sqrt{T}\right) \\
\leq & \ \frac{\log(T)}{\sqrt{2\pi}\sqrt{1-\frac{1}{\sqrt{T_{0}}}}}+\frac{\frac{1}{\sqrt{2\pi}}\left(\frac{\sqrt{T_{0}}}{\sqrt{T_{0}}-1}-\log(2\pi)+2\gamma-1\right)+2A_{\frac{1}{2}}}{\sqrt{1-\frac{1}{\sqrt{T_{0}}}}},
\end{align*}
where, by Lemma~\ref{logg}\eqref{logg1} and since $T\geq T_{0}$, we used that 
\begin{align*}
-\sqrt{T}\log\left(1-\frac{1}{\sqrt{T}}\right)\leq\frac{\sqrt{T}}{\sqrt{T}-1}\leq\frac{\sqrt{T_{0}}}{\sqrt{T_{0}}-1}.\end{align*}
 
By combining all terms, and merging all error terms to the order $\sqrt{T}\log(T)$, we conclude the result.
\end{proofbold}

\subsection{The integrals $J_i$}\label{sec:J}

\textbf{Choice of parameter.} In Propositions~\ref{pr:J} and \ref{pr:K}, our choice will be $\lambda=\frac{\mathrm{c}}{\log(T)}$, where $\mathrm{c}=\sage{c_la}$ optimizes the arising constants.

A particular result that we need in this section is the following. 

\footnotesize
\begin{mysage}
###
# D_{1/2} = 2A_{1/2}+2gamma-1.
D12 = RIF( 2*A12+2*gamma_int-1 )
# D_{sigma} = 2A_{sigma}+1/(2-2sigma)+1/(1-2sigma)^2.
# We write it as constants multiplying [1,1/sigma,1/(1/2-sigma)^2].
# Recall that A_{sigma} has 1/(1/2-sigma) in the third slot instead.
def Fu_Ds(simin,simax):
    simin = RIF( simin )
    simax = RIF( simax )
    Ass = Fu_As(simin,simax)
    modifier = 1/2-simin
    result1 = 2*Ass[0]+1/(2-2*simax)
    result2 = 2*Ass[1]
    result3 = 2*Ass[2]*modifier+1/4
    return [RIF(result1),RIF(result2),RIF(result3)]
Ds_i0 = Fu_Ds(0,1/2)
Ds_i4 = Fu_Ds(1/4,1/2)
# We are going to use the bound in {pr:sumdivwei_B} for T/(2pi),
# and we need T0/(2pi)>=1 for the proof to be working.
# For a check about this, see at the end of the file.
\end{mysage}
\normalsize

\begin{proposition}\label{pr:sumdivwei_B}
Let $X\geq 1$ and $0<\sigma\leq\frac{1}{2}$. Recall the definition of $A_{\sigma}$ given in Lemma \ref{pr:sumdiv}. Then, if $\sigma<\frac{1}{2}$,
\begin{equation*}
\sum_{n\leq X}\frac{d_{1-2\sigma}(n)}{n^{2-2\sigma}}=\zeta(2-2\sigma)\log(X)+O^{*}(D_{\sigma}),
\end{equation*}
where $D_{\sigma}=2A_{\sigma}+\frac{1}{2-2\sigma}+\frac{1}{(1-2\sigma)^{2}}$, whereas
\begin{equation*}
\sum_{n\leq X}\frac{d_{0}(n)}{n}=\frac{1}{2}\log^2(X)+2\gamma\log(X)+O^{*}\left(D_{\frac{1}{2}}\right).
\end{equation*}
where $D_{\frac{1}{2}}=2A_{\frac{1}{2}}+2\gamma-1$.
\end{proposition}

\begin{proof}
Consider \eqref{sbpgeneral} with $\upsilon=2-2\sigma$. Assume first that $0<\sigma<\frac{1}{2}$. Similar to the proof of Proposition~\ref{pr:sumdivwei}, after integrating we keep the first term of \eqref{sbpgeneral} and merge the remaining ones to a constant order: note that the second term of order $X^{2\sigma-1}$ is of smaller order than a constant, unlike in Proposition~\ref{pr:sumdivwei}. By Proposition~\ref{pr:mvzeta} we have
\begin{equation*}
\frac{-3+6\sigma-4\sigma^{2}}{(1-2\sigma)^{2}(2-2\sigma)}<\zeta(2\sigma)+\frac{\zeta(2-2\sigma)}{2-2\sigma}+\frac{\zeta(2\sigma)}{1-2\sigma}-\frac{\zeta(2\sigma)}{(1-2\sigma)X^{1-2\sigma}}<1+\frac{1}{1-2\sigma},
\end{equation*}
and both sides are bounded in absolute value by $\frac{1}{2-2\sigma}+\frac{1}{(1-2\sigma)^{2}}$; moreover, as $X\geq 1$, the remainder coming from \eqref{sbpgeneral} is bounded by $2A_{\sigma}$. Therefore, $D_{\sigma}$ may be defined as in the statement.

On the other hand, if $\sigma=\frac{1}{2}$, we readily define $D_{\frac{1}{2}}$ upon observing from \eqref{sbpgeneral} and Proposition \ref{pr:sumdiv} that
\begin{equation*}
\sum_{n\leq X}\frac{d_{0}(n)}{n}=\frac{1}{2}\log^2(X)+2\gamma\log(X)+2\gamma-1+O^{*}\left(2A_{\frac{1}{2}}\right).
\end{equation*}
\end{proof}
 
Recall the definition of $J_1,J_2$ given in \eqref{definition:J1,J2}. We are now ready to bound them.
 
\begin{proofbold}{Proposition~\ref{pr:J}}

\footnotesize
\begin{mysage}
###
# Supremum of exponent k+lambda for T>=T0.
def Fu_extrambda(ci,ta,ka):
    ci = RIF( ci )
    ta = RIF( ta )
    ka = RIF( ka )
    result = ka+ci/log_int(ta)
    return RIF(result)
###
# Functions for computing constants R'(c,T0) and kappa(c,T0) below.
def Fu_RpJ(simin,simax,ci,ta):
    ci = RIF( ci )
    ta = RIF( ta )
    si = Fu_extrambda(ci,ta,2-2*simin)
    result = (si-1/2)*(si^2/2+si^4/(4*ta^2))+si/12+si^3/3+1/(90*ta)
    return RIF(result)
def Fu_kapJ(simin,simax,ci,ta):
    ci = RIF( ci )
    ta = RIF( ta )
    result = exp_int(Fu_RpJ(simin,simax,ci,ta)/ta^2) \
             +exp_int(Fu_RpJ(simin,simax,ci,ta)/ta^2-pi_int*ta)
    return RIF(result)
###
# Constants R'(c,T0) and kappa(c,T0) below.
Rpc_i4 = roundup( Fu_RpJ(1/4,1/2,c_la,T0) , digits )
Rpc_i2 = roundup( Fu_RpJ(1/2,1/2,c_la,T0) , digits )
kapc_i4 = roundup( Fu_kapJ(1/4,1/2,c_la,T0) , digits )
\end{mysage}
\normalsize

From \eqref{definition:J1,J2}, we may write that $J_{2}=L_1-L_2$, where
\begin{align*}
L_1 & =\frac{1}{2i}\int_{2-2\tau+\lambda+iT}^{1-\tau+iT}\chi(1-s)\zeta(s)\zeta(2\tau-1+s)ds, \\
L_2 & =\frac{1}{2i}\int_{2-2\tau+\lambda+iT}^{1-\tau+iT}\chi(1-s)\sum_{n\leq\frac{T}{2\pi}}\frac{d_{1-2\tau}(n)}{n^{s}}ds.
\end{align*}
First, note that $\left[1-\tau+iT,2-2\tau+\lambda+iT\right)$ belongs to the angular sector defined by $|\arg(s)|<\frac{\pi}{2}$. We can then use Theorem \ref{Stirling}\eqref{StirlingGam} with $\sigma\in[1-\tau,2-2\tau+\lambda)$, $t=T>0$ and $\theta=\frac{\pi}{2}$ (so that $F_{\theta}=\frac{1}{90}$), along with the estimation $\log(|s|)=\log(T)+\frac{\sigma^2}{2T^2}+O^*\left(\frac{\sigma^4}{4T^4}\right)$ to obtain that
\begin{align}\label{gam}  
|\Gamma(s)|&=\sqrt{2\pi}T^{\sigma-\frac{1}{2}}e^{-\frac{\pi}{2}T}e^{O^{*}\left(\frac{\mathrm{R}(\sigma,T_{0})}{T^2}\right)},
\end{align}
where, by using that $\frac{1}{|s|^2}\leq\frac{1}{T^2}$ and that $\frac{1}{T^3}\leq\frac{1}{T_0 T^2}$,
\begin{align*}  
\mathrm{R}(\sigma,T_{0})=\left(\sigma-\frac{1}{2}\right)\left(\frac{\sigma^2}{2}+\frac{\sigma^4}{4T_{0}^2}\right)+\frac{\sigma}{12}+\frac{\sigma^3}{3}+\frac{1}{90T_{0}}.
\end{align*}
Moreover, as $\frac{1}{2}\leq 1-\tau\leq\sigma<2-2\tau+\lambda\leq\frac{3}{2}+\frac{\mathrm{c}}{\log(T_{0})}$, we have the uniform bound 
\begin{align*}\mathrm{R}(\sigma,T_{0})\leq\mathrm{R}\left(\frac{3}{2}+\frac{\mathrm{c}}{\log(T_{0})},T_{0}\right)\leq\mathrm{R}'\left(\frac{3}{2},\mathrm{c},T_{0}\right)=\sage{Rpc_i4},\end{align*} 
and specifically $\mathrm{R'}(1,\mathrm{c},T_{0})\leq\sage{Rpc_i2}$.

Therefore, by \eqref{functional}, \eqref{gam} and Proposition~\ref{cosine}, we conclude that for any $s=\sigma+it\in\left[1-\tau+iT,2-2\tau+\lambda+iT\right)$,  
\begin{align*}
|\chi(1-s)|=\left|2(2\pi)^{-s}\cos\left(\frac{\pi s}{2}\right)\Gamma(s)\right|=(2\pi)^{\frac{1}{2}-\sigma}\left(1+O^*\left(\frac{1}{e^{\pi T}}\right)\right)T^{\sigma-\frac{1}{2}}e^{O^{*}\left(\frac{\mathrm{R}'(3/2,\mathrm{c},T_{0})}{T^2}\right)},
\end{align*}
so that, since $T\geq T_{0}$, we derive the uniform bound on $\left[1-\tau+iT,2-2\tau+\lambda+iT\right)$,
\begin{align}\label{bound:chi}
|\chi(1-s)|\leq\kappa(\mathrm{c},T_0)\left(\frac{T}{2\pi}\right)^{\sigma-\frac{1}{2}},
\end{align}
where $\kappa(\mathrm{c},T_0)=\sage{kapc_i4}\geq e^{\frac{\mathrm{R}'(3/2,\mathrm{c},T_{0})}{{T_0}^2}}+e^{\frac{\mathrm{R}'(3/2,\mathrm{c},T_{0})}{{T_0}^2}-\pi T_0}$.

\footnotesize
\begin{mysage}
###
# Estimating J.
###
# All the following coefficients per order inside J
# are written as constants multiplying [1,1/tau,1/(1/2-tau)^2].
###
# Coefficients per order inside L_1.
def Fu_Jt_L1_1mtlq(simin,simax,ci,ta):
    simin = RIF( simin )
    simax = RIF( simax )
    ci = RIF( ci )
    ta = RIF( ta )
    result1 = Fu_kapJ(simin,simax,ci,ta)/2*Fu_omega(ta)^2*simax \
             *1/exp_int(log_int(2*pi_int)*(1-simax))
    result2 = 0
    result3 = 0
    return [RIF(result1),RIF(result2),RIF(result3)]
def Fu_Jt_L1_32m2tl(simin,simax,ci,ta):
    simin = RIF( simin )
    simax = RIF( simax )
    ci = RIF( ci )
    ta = RIF( ta )
    term1 = 2*Fu_omega(ta)*(1+log_int(2*pi_int)/log_int(ta/(2*pi_int))) \
            *1/exp_int(log_int(2*pi_int)*(1/2-simax))
    term2 = (exp_int(ci)-1)*(1+log_int(2*pi_int)/log_int(ta/(2*pi_int))) \
            *1/exp_int(log_int(2*pi_int)*(1/2-simax))
    if simin>=1/2:
        result1 = Fu_kapJ(simin,simax,ci,ta)/2*term2 \
                  *1/exp_int(log_int(2*pi_int)*(1-simax))
    else:
        result1 = Fu_kapJ(simin,simax,ci,ta)/2*(term1+term2) \
                  *1/exp_int(log_int(2*pi_int)*(1-simax))
    result2 = 0
    result3 = 0
    return [RIF(result1),RIF(result2),RIF(result3)]
###
# Coefficients per order inside L_2.
def Fu_Jt_L2_const(simin,simax,ci,ta):
    simin = RIF( simin )
    simax = RIF( simax )
    ci = RIF( ci )
    ta = RIF( ta )
    result = Fu_kapJ(simin,simax,ci,ta)*exp_int(ci)/2*(1-simin+ci/log_int(ta)) \
             *1/exp_int(log_int(2*pi_int)*(3/2-2*simax))
    return RIF(result)
def Fu_Jt_L2_32m2tl(simin,simax,ci,ta):
    simin = RIF( simin )
    simax = RIF( simax )
    ci = RIF( ci )
    ta = RIF( ta )
    modifier = 1/2-simin
    ze = [zea_2m2s[0],zea_2m2s[1],zea_2m2s[2]*modifier]
    result = cw_prod( Fu_Jt_L2_const(simin,simax,ci,ta) , ze )
    return result
def Fu_Jt_L2_32m2t(simin,simax,ci,ta):
    simin = RIF( simin )
    simax = RIF( simax )
    ci = RIF( ci )
    ta = RIF( ta )
    result = cw_prod( Fu_Jt_L2_const(simin,simax,ci,ta) , Fu_Ds(simin,simax) )
    return result
def Fu_J2_L2_12lq(ci,ta):
    ci = RIF( ci )
    ta = RIF( ta )
    result = Fu_Jt_L2_const(1/2,1/2,ci,ta)*1/2
    return RIF(result)
def Fu_J2_L2_12l(ci,ta):
    ci = RIF( ci )
    ta = RIF( ta )
    result = Fu_Jt_L2_const(1/2,1/2,ci,ta) \
             *max(0,2*gamma_int-log_int(2*pi_int))
    return RIF(result)
def Fu_J2_L2_12(ci,ta):
    ci = RIF( ci )
    ta = RIF( ta )
    result = Fu_Jt_L2_const(1/2,1/2,ci,ta)*(1/2*log_int(2*pi_int)^2 \
             -2*gamma_int*log_int(2*pi_int)+D12)
    return RIF(result)
###
# Coefficients per order of the whole J.
# For tau in (0,1/2), we round coefficients to the orders
# T^(3/2-2tau)*log(T) and T^(1/2-tau)*log(T)^2.
# For tau=1/2, we round coefficients to the orders
# T^(1/2)*log(T)^2 and T^(1/2)*log(T).
def Fu_Jt__32m2tl_final(simin,simax,ci,ta):
    simin = RIF( simin )
    simax = RIF( simax )
    ci = RIF( ci )
    ta = RIF( ta )
    x1_32m2tl = Fu_Jt_L1_32m2tl(simin,simax,ci,ta)
    x2_32m2tl = Fu_Jt_L2_32m2tl(simin,simax,ci,ta)
    x2_32m2t = cw_prod( 1/log_int(ta) , Fu_Jt_L2_32m2t(simin,simax,ci,ta) )
    x = cw_sum( x1_32m2tl , x2_32m2tl )
    x = cw_sum( x , x2_32m2t )
    return x
def Fu_Jt__1mtlq_final(simin,simax,ci,ta):
    simin = RIF( simin )
    simax = RIF( simax )
    ci = RIF( ci )
    ta = RIF( ta )
    x1_1mtlq = Fu_Jt_L1_1mtlq(simin,simax,ci,ta)
    return x1_1mtlq
def Fu_J2__12lq_final(ci,ta):
    ci = RIF( ci )
    ta = RIF( ta )
    x1_12lq = Fu_Jt_L1_1mtlq(1/2,1/2,ci,ta)
    if x1_12lq[2]!=0:
        return 'Error'
    x1_12lq = x1_12lq[0]+x1_12lq[1]/(1/2)
    x2_12lq = Fu_J2_L2_12lq(ci,ta)
    result = x1_12lq + x2_12lq
    return RIF(result)
def Fu_J2__12l_final(ci,ta):
    ci = RIF( ci )
    ta = RIF( ta )
    x1_12l = Fu_Jt_L1_32m2tl(1/2,1/2,ci,ta)
    if x1_12l[2]!=0:
        return 'Error'
    x1_12l = x1_12l[0]+x1_12l[1]/(1/2)
    x2_12l = Fu_J2_L2_12l(ci,ta)
    x2_12 = Fu_J2_L2_12(ci,ta)
    result = x1_12l + x2_12l + x2_12/log_int(ta)
    return RIF(result)
###
# Actual values for our choice of c,T0.
# There is the potential for a term of order T^(1-tau)*log(T)^2*1/(1/2-tau)^2,
# so we check at the end of the file that this term actually does not exist.
J0__32m2tl_final_vec = Fu_Jt__32m2tl_final(0,1/2,c_la,T0)
J0__1mtlq_final_vec = Fu_Jt__1mtlq_final(0,1/2,c_la,T0)
J4__32m2tl_final_vec = Fu_Jt__32m2tl_final(1/4,1/2,c_la,T0)
J4__32m2tl_final_c = roundup( J4__32m2tl_final_vec[0] \
                           + J4__32m2tl_final_vec[1]/(1/4) , digits )
J4__32m2tl_final_2q = roundup( J4__32m2tl_final_vec[2] , digits )
J4__1mtlq_final_vec = Fu_Jt__1mtlq_final(1/4,1/2,c_la,T0)
J4__1mtlq_final_c = roundup( J4__1mtlq_final_vec[0] \
                           + J4__1mtlq_final_vec[1]/(1/4) , digits )
J4__1mtlq_final_2q = roundup( J4__1mtlq_final_vec[2] , digits )
J2__12lq_final = roundup( Fu_J2__12lq_final(c_la,T0) , digits )
J2__12l_final = roundup( Fu_J2__12l_final(c_la,T0) , digits )
\end{mysage}
\normalsize

Secondly, we may derive an upper bound for $L_{1}$ by using the convexity bounds of $\zeta$ and the definition of $\omega$ given in Corollary~\ref{convexity}. Together with \eqref{bound:chi}, we conclude that, for all $s\in\left[1-\tau+iT,2-2\tau+\lambda+iT\right)$,
\begin{align}\label{estimation:abs2}
\left|\chi(1-s)\zeta(s)\zeta(2\tau-1+s)\right|\leq\kappa(\mathrm{c},T_0)\log^2(T)\left(\mathds{1}_{[1-\tau,1)}(\sigma)\ \omega^2 \left(\frac{T}{2\pi}\right)^{\frac{1}{2}-\frac{2\tau-1}{2}}\right.&\nonumber\\
\left.+\mathds{1}_{[1,2-2\tau)}(\sigma)\omega\left(\frac{T}{2\pi}\right)^{\frac{\sigma-(2\tau-1)}{2}} + \mathds{1}_{[2-2\tau,2-2\tau+\lambda)}(\sigma)\left(\frac{T}{2\pi}\right)^{\sigma-\frac{1}{2}}\right)&
\end{align}
so that, by integrating \eqref{estimation:abs2}, we have 
\begin{align}\label{estimation:L1}
|L_1|\leq& \ \frac{\kappa(\mathrm{c},T_0)}{2}\left(\omega^2\tau\log(T)+2\omega\left(1+\frac{\log(2\pi)}{\log\left(\frac{T_{0}}{2\pi}\right)}\right)\left(\left(\frac{T}{2\pi}\right)^{\frac{1}{2}-\tau}-1\right)\right. \nonumber\\
& \ \left.+(e^{\mathrm{c}}-1)\left(1+\frac{\log(2\pi)}{\log\left(\frac{T_{0}}{2\pi}\right)}\right)\left(\frac{T}{2\pi}\right)^{\frac{1}{2}-\tau}\right) \left(\frac{T}{2\pi}\right)^{1-\tau}\log(T) 
\end{align}
where we have used that $\left(\frac{T}{2\pi}\right)^\lambda<T^\lambda=e^{\mathrm{c}}$ and that $\frac{\log(T)}{\log\left(\frac{T}{2\pi}\right)}\leq 1+\frac{\log(2\pi)}{\log\left(\frac{T_{0}}{2\pi}\right)}$. Note that if $\tau=\frac{1}{2}$, the middle term in \eqref{estimation:L1} vanishes.

On the other hand, with respect to $L_2$, we recall \eqref{bound:chi} and Lemma~\ref{logg}\eqref{logg1}, as well as the facts that  $\lambda\leq\frac{\mathrm{c}}{\log(T_0)}$ and $\left(\frac{T}{2\pi n}\right)^{\lambda}<T^{\lambda}=e^{\mathrm{c}}$, and derive
\begin{align}
|L_{2}|&\leq\frac{\kappa(\mathrm{c},T_0)}{2}\sqrt{\frac{2\pi}{T}}\sum_{n\leq \frac{T}{2\pi}}d_{1-2\tau}(n)\int_{1-\tau}^{2-2\tau+\lambda}\left(\frac{T}{2\pi n}\right)^{\sigma}d\sigma\nonumber\\
&\leq\frac{\kappa(\mathrm{c},T_0)}{2}\left(\frac{T}{2\pi }\right)^{\frac{1}{2}-\tau}\sum_{n\leq\frac{T}{2\pi}}\frac{d_{1-2\tau}(n)}{n^{1-\tau}}\frac{\left(\frac{T}{2\pi n}\right)^{1-\tau+\lambda}-1}{\log\left(\frac{T}{2\pi n}\right)}\nonumber\\
&\leq\frac{\kappa(\mathrm{c},T_0)e^{\mathrm{c}}}{2}\left(1-\tau+\frac{\mathrm{c}}{\log(T_{0})}\right)\left(\frac{T}{2\pi }\right)^{\frac{3}{2}-2\tau}\sum_{n\leq\frac{T}{2\pi}}\frac{d_{1-2\tau}(n)}{n^{2-2\tau}}\label{estimation:L2}.
\end{align} 
We then apply Proposition~\ref{pr:sumdivwei_B} to the inner sum of \eqref{estimation:L2}: when $0<\tau<\frac{1}{2}$, we simplify the main term coefficient with the help of Proposition~\ref{pr:mvzeta}, whereas, when $\tau=\frac{1}{2}$, by ignoring negative coefficients, we merge the remainder terms to the order $\sqrt{T}\log(T)$. Thereupon, by using that $\tau\leq\frac{1}{2}$, $(2\pi)^{-\frac{3}{2}+2\tau}\leq\frac{1}{\sqrt{2\pi}}$ and that $(2\pi)^{-1+\tau}\leq\frac{1}{\sqrt{2\pi}}$ we combine the resulting bound with \eqref{estimation:L1}, obtaining the result.

Finally, observe that \eqref{estimation:L1} and \eqref{estimation:L2} also hold when bounding $J_{1}$, in which case $T$ is replaced by $|T|$.
\end{proofbold}

\subsection{The integral $K$}\label{sec:K}

Let us give the following tail estimation of an arithmetical function involving $d_{a}$ and the parameter $\lambda$.

\begin{proposition}\label{pr:sumdivwei_C}
Let $X\geq 1$, $0<\sigma\leq\frac{1}{2}$ and $\lambda>0$. Then, if $\sigma<\frac{1}{2}$,
\begin{align*}
\sum_{n>X}\frac{d_{1-2\sigma}(n)}{n^{2-2\sigma+\lambda}}= & \ \frac{\zeta(2-2\sigma+\lambda)}{\lambda X^{\lambda}}+O^{*}\left(\frac{\frac{1}{2-2\sigma}+\frac{1}{(1-2\sigma)(1-2\sigma+\lambda)}}{X^{\lambda}}\right. \\
 & \ \left.+\frac{\zeta(2-2\sigma+\lambda)+\frac{1}{1-2\sigma+\lambda}+1}{X^{1+\lambda}}\right),
\end{align*}
whereas 
\begin{equation*}
\sum_{n>X}\frac{d(n)}{n^{1+\lambda}}=\frac{\log(X)}{\lambda X^{\lambda}}+O^*\left(\left(1+\frac{\zeta(1+\lambda)+\gamma}{\lambda}+\frac{\zeta(1+\lambda)+\frac{2}{3\lambda}}{X}\right)\frac{1}{X^{\lambda}}\right).
\end{equation*}
\end{proposition}

\begin{proof}
Let $0<\sigma\leq\frac{1}{2}$. Observe that
\begin{equation}\label{convergence}
\sum_{n}\frac{d_{1-2\sigma}(n)}{n^{2-2\sigma+\lambda}}=\sum_{n}\sum_{d|n}\frac{1}{d^{1+\lambda}}\left(\frac{d}{n}\right)^{2-2\sigma+\lambda}=\zeta(1+\lambda)\zeta(2-2\sigma+\lambda).
\end{equation}
On the other hand, by Lemma~\ref{le:sums}\eqref{le:sums<0}, we have 
\begin{align*}
\sum_{n\leq X}\frac{d_{1-2\sigma}(n)}{n^{2-2\sigma+\lambda}}= & \ \sum_{d\leq X}\frac{1}{d^{1+\lambda}}\!\left(\zeta(2-2\sigma+\lambda)-\frac{d^{1-2\sigma+\lambda}}{(1-2\sigma+\lambda)X^{1-2\sigma+\lambda}}+O^{*}\!\left(\frac{d^{2-2\sigma+\lambda}}{X^{2-2\sigma+\lambda}}\right)\right) \nonumber \\
= & \ \zeta(2-2\sigma+\lambda)\left(\zeta(1+\lambda)-\frac{1}{\lambda X^{\lambda}}+O^*\left(\frac{1}{X^{1+\lambda}}\right)\right) \nonumber \\
 & \ -\frac{1}{(1-2\sigma+\lambda)X^{1-2\sigma+\lambda}}\sum_{d\leq X}\frac{1}{d^{2\sigma}}+O^{*}\left(\frac{1}{X^{2-2\sigma+\lambda}}\sum_{d\leq X}d^{1-2\sigma}\right).
\end{align*}
Then, for $0<\sigma<\frac{1}{2}$ we use Lemma~\ref{le:sums}\eqref{le:sums<0}-\eqref{le:sums>0} on the remaining sums, while for $\sigma=\frac{1}{2}$ we use Lemma~\ref{le:sums}\eqref{le:sums=-1} and $\sum_{d\leq X}1\leq X$. Subsequently, we subtract the resulting expression from \eqref{convergence}. When $0<\sigma<\frac{1}{2}$, the negative summand coming from the first term of Lemma~\ref{le:sums}\eqref{le:sums<0} is smaller in absolute value than the positive summand coming from the second term, by Proposition~\ref{pr:mvzeta}: thus, when combining everything into the error term, we can forget about the former. The result follows.
\end{proof}

\footnotesize
\begin{mysage}
###
# Constants P for the case sigma=1/2.
P2_1 = RIF( 1/(4*pi_int) )
P2_2 = RIF( (2*gamma_int-log_int(pi_int))/(4*pi_int) )
def Fu_P2_3(ta):
    ta = RIF( ta )
    result = A12*sqrt_int(2/pi_int)*sqrt_int(1+1/sqrt_int(ta))
    return RIF(result)
P2_4 = RIF( A12/(2*sqrt_int(2*pi_int)) )
P2_5 = RIF( 4*A12*sqrt_int(pi_int) )
###
# Constants P for the case 0<sigma<1/2.
# We write them as constants multiplying [1,1/sigma,1/(1/2-sigma)].
Ps_1 = [ RIF( 0 ) , RIF( 0 ) , RIF( 0 ) ]
def Fu_Ps_2(simin,simax):
    simin = RIF( simin )
    simax = RIF( simax )
    result = cw_prod( 1/(4*pi_int) , zea_2m2s )
    return result
def Fu_Ps_3(simin,simax):
    simin = RIF( simin )
    simax = RIF( simax )
    result = cw_prod( 2/exp_int(log_int(2*pi_int)*simin) , Fu_As(simin,simax) )
    return result
def Fu_Ps_4(simin,simax):
    simin = RIF( simin )
    simax = RIF( simax )
    result = cw_prod( (1-simin)/exp_int(log_int(2*pi_int)*simin) , Fu_As(simin,simax) )
    return result
def Fu_Ps_5(simin,simax):
    simin = RIF( simin )
    simax = RIF( simax )
    result = cw_prod( 2/exp_int(log_int(pi_int)*simin) , Fu_As(simin,simax) )
    return result
\end{mysage}
\normalsize

\begin{lemma}\label{le:sumKlog}
Let $T\geq T_{0}=\sage{T0}$ and $0<\sigma\leq\frac{1}{2}$. Then
\begin{equation*}
\sum_{\frac{T+\sqrt{T}}{2\pi}<n\leq\frac{T}{\pi}}\frac{d_{1-2\sigma}(n)}{n^{1-2\sigma}(2\pi n-T)}\leq P_{1,\sigma}\log^{2}(T)+P_{2,\sigma}\log(T)+\frac{P_{3,\sigma}}{T^{\frac{1}{2}-\sigma}}+\frac{P_{4,\sigma}\log(T)}{T^{1-\sigma}}+\frac{P_{5,\sigma}}{T^{1-\sigma}},
\end{equation*}
where for $0<\sigma<\frac{1}{2}$
\begin{align*}
P_{1,\sigma} & =0, & P_{2,\sigma} & =\frac{\zeta(2-2\sigma)}{4\pi}, & P_{3,\sigma} & =\frac{2A_{\sigma}}{(2\pi)^{\sigma}}, & P_{4,\sigma} & =\frac{A_{\sigma}(1-\sigma)}{(2\pi)^{\sigma}}, & P_{5,\sigma} & =\frac{2A_{\sigma}}{\pi^{\sigma}},
\end{align*}
with $A_{\sigma}$ as in Proposition~\ref{pr:sumdiv}, whereas for $\sigma=\frac{1}{2}$
\begin{align*}
P_{1,\frac{1}{2}} & =\frac{1}{4\pi}, & P_{2,\frac{1}{2}} & =\frac{2\gamma-\log(\pi)}{4\pi}, & P_{3,\frac{1}{2}} & =\frac{A_{\frac{1}{2}}\sqrt{2}}{\sqrt{\pi}}\sqrt{1+\frac{1}{\sqrt{T_{0}}}}, \\
P_{4,\frac{1}{2}} & =\frac{A_{\frac{1}{2}}}{2\sqrt{2\pi}}, & P_{5,\frac{1}{2}} & =4A_{\frac{1}{2}}\sqrt{\pi}. & & 
\end{align*}
\end{lemma}

\begin{proof}
We follow the same reasoning as in Lemma~\ref{le:sumdivlog}. By Proposition~\ref{pr:sumdiv}, we write 
\begin{align*}\sum_{\frac{T+\sqrt{T}}{2\pi}<n\leq t}d_{1-2\sigma}(n)=\mathrm{M}_{\sigma}(t)-\mathrm{M}_{\sigma}\left(\frac{T+\sqrt{T}}{2\pi}\right)+\Xi_{\sigma}(t)\end{align*} 
with $|\Xi_{\sigma}(t)|\leq 2A_{\sigma}t^{1-\sigma}$ and $\Xi_{\sigma}\left(\frac{T+\sqrt{T}}{2\pi}\right)=0$. For $0<\sigma<\frac{1}{2}$, the sum in the statement is bounded as
\begin{align}
& \frac{\zeta(2\sigma)}{2\pi}\int_{\frac{T+\sqrt{T}}{2\pi}}^{\frac{T}{\pi}}\frac{dt}{t^{1-2\sigma}\left(t-\frac{T}{2\pi}\right)}+\frac{\zeta(2-2\sigma)}{2\pi}\int_{\frac{T+\sqrt{T}}{2\pi}}^{\frac{T}{\pi}}\frac{dt}{t-\frac{T}{2\pi}} \nonumber \\
& +\frac{\Xi_{\sigma}\left(\frac{T}{\pi}\right)}{\left(\frac{T}{\pi}\right)^{1-2\sigma}T}-\frac{1}{2\pi}\int_{\frac{T+\sqrt{T}}{2\pi}}^{\frac{T}{\pi}}\Xi_{\sigma}(t)\left(\frac{1}{t^{1-2\sigma}\left(t-\frac{T}{2\pi}\right)}\right)'dt. \label{eq:Kloglong_s}
\end{align}
Since $\zeta(2\sigma)<0$, the first term in \eqref{eq:Kloglong_s} can be ignored. The integral in the second term is equal to $\frac{1}{2}\log(T)$. As for the terms involving $\Xi_{\sigma}$, they can be bounded by
\begin{align}
 & \ \frac{2A_{\sigma}}{\pi^{\sigma}T^{1-\sigma}}+\left|\frac{1}{2\pi}\left[\frac{2A_{\sigma}t^{\sigma}}{t-\frac{T}{2\pi}}\right|_{\frac{T+\sqrt{T}}{2\pi}}^{\frac{T}{\pi}}\right|+\left|\frac{1}{2\pi}\int_{\frac{T+\sqrt{T}}{2\pi}}^{\frac{T}{\pi}}\frac{2A_{\sigma}(1-\sigma)}{t^{1-\sigma}\left(t-\frac{T}{2\pi}\right)}dt\right| \nonumber\\
< & \ \frac{2A_{\sigma}}{\pi^{\sigma}T^{1-\sigma}}+\frac{A_{\sigma}}{\pi}(2\pi)^{1-\sigma}\left(\frac{(T+\sqrt{T})^{\sigma}}{\sqrt{T}}-\frac{2^{\sigma}}{T^{1-\sigma}}\right)+\frac{1}{2\pi\left(\frac{T}{2\pi}\right)^{1-\sigma}}\int_{\frac{T+\sqrt{T}}{2\pi}}^{\frac{T}{\pi}}\frac{2A_{\sigma}(1-\sigma)}{t-\frac{T}{2\pi}}dt \nonumber\\
< & \ \frac{2A_{\sigma}}{\pi^{\sigma}T^{1-\sigma}}+\frac{2A_{\sigma}}{(2\pi)^{\sigma}T^{\frac{1}{2}-\sigma}}+\frac{A_{\sigma}(1-\sigma)\log(T)}{(2\pi)^{\sigma}T^{1-\sigma}}, \label{eq:Xi}
\end{align}
where we used Lemma~\ref{logg}\eqref{loggexp} to show that $(T+\sqrt{T})^{\sigma}\leq T^{\sigma}(1+\sigma T^{-\frac{1}{2}})<T^{\sigma}+2^{\sigma}T^{\sigma-\frac{1}{2}}$. By putting \eqref{eq:Xi} back into \eqref{eq:Kloglong_s}, we obtain the result.

We can proceed similarly for $\sigma=\frac{1}{2}$. By Proposition~\ref{pr:sumdiv}, the sum in the statement is bounded as
\begin{equation}\label{eq:Kloglong}
\frac{1}{2\pi}\int_{\frac{T+\sqrt{T}}{2\pi}}^{\frac{T}{\pi}}\frac{(\log(t)+2\gamma)dt}{t-\frac{T}{2\pi}}+\frac{1}{2\pi}\int_{\frac{T+\sqrt{T}}{2\pi}}^{\frac{T}{\pi}}\frac{2A_{\frac{1}{2}}\sqrt{t}dt}{\left(t-\frac{T}{2\pi}\right)^{2}}+\frac{4A_{\frac{1}{2}}\sqrt{\pi}}{\sqrt{T}}.
\end{equation}
The first integral, coming from the main term, can in turn be easily bounded since by definition $\log(t)+2\gamma\leq\log(T)+2\gamma-\log(\pi)$ for $t\leq\frac{T}{\pi}$, which implies
\begin{align}\label{eq:Klogmain}
\frac{1}{2\pi}\int_{\frac{T+\sqrt{T}}{2\pi}}^{\frac{T}{\pi}}\frac{(\log(t)+2\gamma)dt}{t-\frac{T}{2\pi}} & \leq\frac{\log(T)+2\gamma-\log(\pi)}{2\pi}\left[\log\left(t-\frac{T}{2\pi}\right)\right|_{\frac{T+\sqrt{T}}{2\pi}}^{\frac{T}{\pi}} \nonumber \\
 & =\frac{1}{4\pi}\log^{2}(T)+\frac{2\gamma-\log(\pi)}{4\pi}\log(T).
\end{align}
For the second integrand in \eqref{eq:Kloglong} we use instead that, for $x>a>0$,
\begin{equation*}
\int\frac{\sqrt{x}dx}{(x-a)^{2}}=\frac{\sqrt{x}}{a-x}-\frac{1}{2\sqrt{a}}\log\left(\frac{\sqrt{x}+\sqrt{a}}{\sqrt{x}-\sqrt{a}}\right).
\end{equation*}
Then, the second term in \eqref{eq:Kloglong} is equal to
\begin{align}
 & \frac{A_{\frac{1}{2}}}{\pi}\left[-\frac{\sqrt{t}}{t-\frac{T}{2\pi}}-\sqrt{\frac{\pi}{2T}}\log\left(\frac{\sqrt{\frac{2\pi t}{T}}+1}{\sqrt{\frac{2\pi t}{T}}-1}\right)\right|_{\frac{T+\sqrt{T}}{2\pi}}^{\frac{T}{\pi}} \nonumber \\
\leq & \frac{A_{\frac{1}{2}}}{\sqrt{2\pi}}\!\left(2\sqrt{1+\frac{1}{\sqrt{T_{0}}}}-\frac{2\sqrt{2}}{\sqrt{T}}+\frac{\log\left(1+\sqrt{1+\frac{1}{\sqrt{T_{0}}}}\right)}{\sqrt{T}}+\frac{\log(T)}{2\sqrt{T}}+\frac{\log\left(\frac{7}{3}\right)}{\sqrt{T}}-\frac{\log\left(\frac{\sqrt{2}+1}{\sqrt{2}-1}\right)}{\sqrt{T}}\right) \nonumber \\
< & \frac{A_{\frac{1}{2}}\sqrt{2}}{\sqrt{\pi}}\sqrt{1+\frac{1}{\sqrt{T_{0}}}}+\frac{A_{\frac{1}{2}}}{2\sqrt{2\pi}}\frac{\log(T)}{\sqrt{T}}, \label{eq:Klogerr}
\end{align}
where in the second line we used the mean value theorem to obtain that for $T\geq T_0$,
\begin{align*}\sqrt{T+\sqrt{T}}-\sqrt{T}\geq\frac{\sqrt{T}}{2\sqrt{T+\sqrt{T}}}>\frac{3}{7}\Rightarrow -\log\left(\sqrt{1+\frac{1}{\sqrt{T}}}-1\right)\leq\frac{1}{2}\log(T)+\log\left(\frac{7}{3}\right),\end{align*} 
and where in the last line we dropped all the terms of order $\frac{1}{\sqrt{T}}$ since they amount to a negative contribution to the bound, given the choice of $T_{0}$. The result is concluded by putting \eqref{eq:Klogmain} and \eqref{eq:Klogerr} back into \eqref{eq:Kloglong}.
\end{proof}

\footnotesize
\begin{mysage}
###
# Functions for computing constants Q,R,R'.
def Fu_Q(simin,simax,ci,ta):
    simin = RIF( simin )
    simax = RIF( simax )
    ci = RIF( ci )
    ta = RIF( ta )
    if 2-2*simin+ci/log_int(ta)<=2:
        gammaterm = 1
    else:
        gammaterm = gam_int(2-2*simin+ci/log_int(ta))
    result = exp_int(pi_int/2)*gammaterm/exp_int(log_int(2*pi_int)*(2-2*simax))
    return RIF(result)
def Fu_R_const(simin,simax):
    simin = RIF( simin )
    simax = RIF( simax )
    result = (exp_int(pi_int)+1)/(2*exp_int(pi_int)*exp_int(log_int(2*pi_int)*(3/2-2*simax)))
    return RIF(result)
def Fu_FK(simin,simax,ci,ta):
    simin = RIF( simin )
    simax = RIF( simax )
    ci = RIF( ci )
    ta = RIF( ta )
    result = Exf(Fu_extrambda(ci,ta,2-2*simin),ta)
    return RIF(result)
def Fu_R_1(simin,simax,ci,ta):
    simin = RIF( simin )
    simax = RIF( simax )
    ci = RIF( ci )
    ta = RIF( ta )
    result = Fu_R_const(simin,simax)*exp_int(ci)/2 \
             *((2-2*simin+ci/log_int(ta))*(1-2*simin+ci/log_int(ta))+25/6)
    return RIF(result)
def Fu_R_2(simin,simax,ci,ta):
    simin = RIF( simin )
    simax = RIF( simax )
    ci = RIF( ci )
    ta = RIF( ta )
    result = Fu_R_const(simin,simax)*Fu_FK(simin,simax,ci,ta)/(1/2-ci/log_int(ta))
    return RIF(result)
def Fu_Rp_1(simin,simax,ci,ta):
    simin = RIF( simin )
    simax = RIF( simax )
    ci = RIF( ci )
    ta = RIF( ta )
    result = Fu_R_const(simin,simax)*exp_int(ci)/2 \
             *((2-2*simin+ci/log_int(ta))*(1-2*simin+ci/log_int(ta))+37/6)
    return RIF(result)
###
# Constants F,Q,R,R'.
FK4 = roundup( Fu_FK(1/4,1/2,c_la,T0) , digits )
FK2 = roundup( Fu_FK(1/2,1/2,c_la,T0) , digits )
Q2 = roundup( Fu_Q(1/2,1/2,c_la,T0) , digits )
Q4 = roundup( Fu_Q(1/4,1/2,c_la,T0) , digits )
R2_1 = roundup( Fu_R_1(1/2,1/2,c_la,T0) , digits )
R4_1 = roundup( Fu_R_1(1/4,1/2,c_la,T0) , digits )
R2_2 = roundup( Fu_R_2(1/2,1/2,c_la,T0) , digits )
R4_2 = roundup( Fu_R_2(1/4,1/2,c_la,T0) , digits )
Rp2_1 = roundup( Fu_Rp_1(1/2,1/2,c_la,T0) , digits )
Rp4_1 = roundup( Fu_Rp_1(1/4,1/2,c_la,T0) , digits )
###
# In the proof we assume that lambda<1/2.
# For a check about this, see the end of the file.
\end{mysage}
\normalsize

\begin{lemma}\label{Kn:estimation} Let $T>T_0=\sage{T0}$ and $\lambda=\frac{\mathrm{c}}{\log(T)}$. For any $n\in\mathbb{Z}_{>0}$ such that $n>\frac{T}{2\pi}$, we have the following estimation:
\begin{align*}
|K_{n}|=\left|\frac{1}{2i}\int_{2-2\tau+\lambda-iT}^{2-2\tau+\lambda+iT}\frac{\chi(1-s)}{n^s}ds\right|\leq 2|X_{n}|+2|Y_{n}|,
\end{align*}
with
\begin{align}
|Y_{n}| & \leq\frac{Q}{n^{2-2\tau+\lambda}}, & & \nonumber \\
|X_{n}| & \leq\frac{R_{1}T^{\frac{3}{2}-2\tau}}{n^{2-2\tau+\lambda}\log\left(\frac{2\pi n}{T}\right)}+\frac{R_{2}T^{1-2\tau}}{n^{2-2\tau+\lambda}} & & \text{if $n>\frac{T+\sqrt{T}}{2\pi}$,} \label{Xn:1} \\
|X_{n}| & \leq\frac{R'_{1}T^{2-2\tau}}{n^{2-2\tau+\lambda}}+\frac{R'_{2}T^{1-2\tau}}{n^{2-2\tau+\lambda}} & & \text{if $\frac{T}{2\pi}<n\leq\frac{T+\sqrt{T}}{2\pi}$,} \label{Xn:2}
\end{align}
where for $\tau=\frac{1}{2}$
\begin{align*}
Q & =\sage{Q2}, & R_{1} & =\sage{R2_1}, & R_{2} & =\sage{R2_2}, & R'_{1} & =\sage{Rp2_1}, & R'_{2} & =R_{2},
\end{align*}
whereas for $\frac{1}{4}\leq\tau<\frac{1}{2}$
\begin{align*}
Q & =\sage{Q4}, & R_{1} & =\sage{R4_1}, & R_{2} & =\sage{R4_2}, & R'_{1} & =\sage{Rp4_1}, & R'_{2} & =R_{2}.
\end{align*}
\end{lemma}

\begin{remark}\label{crucial2} 
Observe that if \eqref{Xn:1} is not as sharp as $n$ approaches $\frac{T}{2\pi}$ from the right, for if $T<2\pi n\leq T+\sqrt{T}$, then by Lemma~\ref{logg}\eqref{logg1} 
\begin{align*}\frac{\sqrt{T}}{\log\left(\frac{2\pi n}{T}\right)}\geq\frac{T^{\frac{3}{2}}}{2\pi n-T}\gg T,\end{align*}
 so that it is better to consider the bound \eqref{Xn:2}. Instead, if $2\pi n>T+\sqrt{T}$, we have 
 \begin{align*}\frac{\sqrt{T}}{\log\left(\frac{2\pi n}{T}\right)}\leq\frac{2\pi n\sqrt{T}}{2\pi n-T}\ll T,\end{align*} 
 and thus, in this case, it is better to consider the bound \eqref{Xn:1} over the one given in \eqref{Xn:2}. 
\end{remark}

\footnotesize
\begin{mysage}
###
# Constant F(c,T0) and kappa(c,T0) below.
F_const = roundup( Exf(Fu_extrambda(c_la,T0,1),T0) , digits )
\end{mysage}
\normalsize

\begin{proof}
By \eqref{functional} and using that $\chi(1-\overline{s})=\overline{\chi(1-s)}$, we may write $K_n=X_{n}+Y_{n}-\overline{X_{n}}-\overline{Y_{n}}$, where
\begin{align}
X_{n}&=\frac{1}{i}\int_{1}^{T}\Gamma(2-2\tau+\lambda+it)\cos\left(\frac{\pi(2-2\tau+\lambda+it)}{2}\right)(2\pi n)^{-(2-2\tau)-\lambda-it}dt,\label{eq:Xn}\\
Y_{n}&=\frac{1}{i}\int_{0}^{1}\Gamma(2-2\tau+\lambda+it)\cos\left(\frac{\pi(2-2\tau+\lambda+it)}{2}\right)(2\pi n)^{-(2-2\tau)-\lambda-it}dt.\label{eq:Yn} 
\end{align} 
By Theorem~\ref{Stirling}\eqref{StirlingGam}-\eqref{StirlingGaI}, we can obtain an expression analogous to \eqref{definitive}: we change variables via $t\leftrightarrow -t$, which does not change the absolute value of the integrands above, and, since $t>0$ for all $s\in 2-2\tau+\lambda+i(0,T]$, we have
\begin{equation*}
\frac{1}{2i}\frac{\chi(1-s)}{n^s}=\frac{e^{-\frac{\pi i}{4}}t^{\frac{3}{2}-2\tau+\lambda}e^{if_{\lambda,\tau}(t)}}{2(2\pi)^{\frac{3}{2}-2\tau+\lambda}n^{2-2\tau+\lambda}}\left(1+O^*\left(\frac{1}{e^{\pi t}}\right)\right)\left(1+\frac{O^*(\mathrm{F}(\mathrm{c},T_0,\tau))}{t^2}\right),
\end{equation*}
where
\begin{equation*}
f_{\lambda,\tau}(t)=f(2-2\tau+\lambda,-t)=t\log\left(\frac{2\pi n}{t}\right)+t+\frac{1}{2}\left((2-2\tau+\lambda)(1-2\tau+\lambda)+\frac{1}{6}\right)\frac{1}{t}
\end{equation*}
and where $\mathrm{F}(\mathrm{c},T_0,\tau)$ is a numerical upper bound of $\mathrm{E}(2-2\tau+\lambda,T_{0})$,  defined as 
\begin{equation*}
\mathrm{E}(2-2\tau+\lambda,T_{0})\leq T_{0}^2\left(e^{\mathrm{e}_1\left(2-2\tau+\frac{\mathrm{c}}{\log(T_0)},T_0\right)+\mathrm{e}_2\left(2-2\tau+\frac{\mathrm{c}}{\log(T_0)},T_0\right)}-1\right),
\end{equation*}
where $\mathrm{e_1}$, $\mathrm{e}_2$ are defined in \eqref{def:e1e2} and $\mathrm{E}$ is defined in \eqref{E_def}, and where we have used that $\lambda\leq\frac{\mathrm{c}}{\log(T_0)}$. Since $\sigma\mapsto\mathrm{E}(\sigma,T_{0})$ is increasing, we can define $\mathrm{F}(\mathrm{c},T_{0},\tau)\leq\mathrm{F}(\mathrm{c},T_{0},\frac{1}{4})=\sage{FK4}$ and $\mathrm{F}(\mathrm{c},T_{0},\frac{1}{2})=\sage{FK2}$.

Therefore, from \eqref{eq:Xn} we conclude that 
\begin{align}
&|X_{n}|=\left|\int_{1}^{T}\frac{e^{-\frac{\pi i}{4}}t^{\frac{3}{2}-2\tau+\lambda}e^{if_{\lambda,\tau}(t)}}{2(2\pi)^{\frac{3}{2}-2\tau+\lambda}n^{2-2\tau+\lambda}}\left(1+O^*\left(\frac{1}{e^{\pi t}}\right)\right)\left(1+\frac{O^*(\mathrm{F}(\mathrm{c},T_0,\tau))}{t^2}\right),dt\right|\nonumber\\
\leq\ &\frac{e^{\pi}+1}{2e^{\pi}(2\pi)^{\frac{3}{2}-2\tau}n^{2-2\tau+\lambda}}\left(\left|\int_{1}^{T}t^{\frac{3}{2}-2\tau+\lambda}e^{i f_{\lambda,\tau}\left(t\right)}dt\right|+\mathrm{F}\left(\mathrm{c},T_0,\tau\right)\left|\int_{1}^{T}\frac{t^{\frac{3}{2}-2\tau+\lambda}e^{i f_{\lambda,\tau}\left(t\right)}}{t^2}dt\right|\right), \label{Xn:est}
\end{align}
where we have used that for $t\in[1,T]$, $\frac{1}{e^{\pi t}}\leq\frac{1}{e^{\pi}}$ and that $\frac{3}{2}-2\tau+\lambda>\frac{3}{2}-2\tau$.

As $2\tau-1\leq 0$ and $\lambda<\frac{1}{2}$, the second integral in \eqref{Xn:est} is readily bounded by 
\begin{equation}\label{eq:inte2}
\int_{1}^{T}t^{-\frac{1}{2}-2\tau+\lambda}dt\leq T^{1-2\tau}\int_{1}^{T}t^{-\frac{3}{2}+\lambda}dt\leq\frac{T^{1-2\tau}}{\frac{1}{2}-\lambda}.
\end{equation}
As for the first integral, we may write $f_{\lambda,\tau}(t)=g(t)+h_{\lambda,\tau}(t)$, where  $g(t)=t\log\left(\frac{2\pi n}{t}\right)+t$ and $h_{\lambda,\tau}(t)=\frac{1}{2}\left((2-2\tau+\lambda)(1-2\tau+\lambda)+\frac{1}{6}\right)\frac{1}{t}$, similarly to the obtention of identity \eqref{identity1}. Moreover, we have that $g'(t)\neq 0$ for $t\in(0,T]$, since $|g'(t)|=\log\left(\frac{2\pi n}{t}\right)\geq\log\left(\frac{2\pi n}{T}\right)$ and, by hypothesis, $\frac{2\pi n}{T}>1$; then, for any $V,W$ such that $V<W$ and $\frac{2\pi n}{W}>1$, we may use
\begin{equation*}
\int_{V}^{W}l(t)e^{if_{\lambda,\tau}\left(t\right)}dt=\left[\frac{l(t)e^{if_{\lambda,\tau}(t)}}{ig'(t)}\right|_{V}^{W}-\int_{V}^{W}e^{if_{\lambda,\tau}(t)}\left(\left(\frac{l(t)}{ig'(t)}\right)'+\frac{l(t)h_{\lambda,\tau}'(t)}{g'(t)}\right)dt,
\end{equation*}
where $l(t)=t^{\frac{3}{2}-2\tau+\lambda}$, and derive
\begin{align}\label{ident}
&\left|\int_{V}^{W}t^{\frac{3}{2}-2\tau+\lambda}e^{i f_{\lambda,\tau}\left(t\right)}dt\right|\leq\frac{2W^{\frac{3}{2}-2\tau+\lambda}}{\log\left(\frac{2\pi n}{W}\right)}+\frac{(2-2\tau+\lambda)(1-2\tau+\lambda)+\frac{1}{6}}{2}\int_{V}^{W}\frac{t^{-\frac{1}{2}-2\tau+\lambda}}{\log\left(\frac{2\pi n}{t}\right)}dt\nonumber\\
&\phantom{xxxxxxxxxxx}\leq\left(2+\frac{(2-2\tau+\lambda)(1-2\tau+\lambda)+\frac{1}{6}}{2}\left(\frac{1}{V}-\frac{1}{W}\right)\right)\frac{W^{\frac{3}{2}-2\tau+\lambda}}{\log\left(\frac{2\pi n}{W}\right)},
\end{align}
where we have used that $\frac{3}{2}-2\tau>0$ and that $t\mapsto\frac{l(t)}{g'(t)}$ is increasing. Hence, by selecting $V=1$, $W=T$ inside \eqref{ident} and by using that $T^{\lambda}=e^{\mathrm{c}}$, $\lambda\leq\frac{\mathrm{c}}{\log(T_{0})}$, we observe from \eqref{Xn:est} and \eqref{eq:inte2} that
\begin{align}\label{Xn:est1}
|X_{n}|\leq\ &\frac{e^{\pi}+1}{2e^{\pi}(2\pi)^{\frac{3}{2}-2\tau}n^{2-2\tau+\lambda}}\left(\frac{e^{\mathrm{c}}}{2}\left((2-2\tau+\lambda)(1-2\tau+\lambda)+\frac{25}{6}\right)\frac{T^{\frac{3}{2}-2\tau}}{\log\left(\frac{2\pi n}{T}\right)}\right.\nonumber\\
&\phantom{xxxxxxxxxxxxxxxxxxxxx}\left.+\frac{\mathrm{F}\left(\mathrm{c},T_0,\tau\right)}{\frac{1}{2}-\frac{\mathrm{c}}{\log(T_0)}}T^{1-2\tau}\right).
\end{align}

As pointed out in Remark \ref{crucial2}, we can do better than \eqref{Xn:est1} when $n\in\left(\frac{T}{2\pi},\frac{T+\sqrt{T}}{2\pi}\right]$. In this range, by Lemma~\ref{logg}\eqref{logg2}, we have that 
\begin{align*}\frac{T^{\frac{3}{2}-2\tau+\lambda}}{\log\left(\frac{2\pi n}{T-\sqrt{T}}\right)}\leq\frac{T^{\frac{3}{2}-2\tau+\lambda}}{\log\left(\frac{T}{T-\sqrt{T}}\right)}\leq T^{2-2\tau+\lambda},\end{align*} 
so
\begin{align}
 & \ \left|\int_{1}^{T}t^{\frac{3}{2}-2\tau+\lambda}e^{i f_{\lambda}\left(t\right)}dt\right|\leq\left|\int_{1}^{T-\sqrt{T}}t^{\frac{3}{2}-2\tau+\lambda}e^{i f_{\lambda}\left(t\right)}dt\right|+\left|\int_{T-\sqrt{T}}^{T}t^{\frac{3}{2}-2\tau+\lambda}e^{i f_{\lambda}\left(t\right)}dt\right|\nonumber\\
\leq & \ \frac{1}{2}\left((2-2\tau+\lambda)(1-2\tau+\lambda)+\frac{25}{6}\right)T^{2-2\tau+\lambda}+T^{2-2\tau+\lambda}, \label{eq:intermest2}
\end{align}
where, in the first integral above, since $\frac{2\pi n}{T-\sqrt{T}}>1$, we used \eqref{ident} with $V=1$, $W=T-\sqrt{T}$ and, in the second one, we have bounded trivially. Thus, by using that $T^{\lambda}=e^{\mathrm{c}}$, $\lambda\leq\frac{\mathrm{c}}{\log(T_{0})}$, we plug \eqref{eq:inte2} and \eqref{eq:intermest2} into \eqref{Xn:est} and obtain
\begin{align}\label{Xn:est2}
|X_{n}|&\leq\frac{e^{\pi}+1}{2e^{\pi}(2\pi)^{\frac{3}{2}-2\tau}n^{2-2\tau+\lambda}}\left(\frac{e^{\mathrm{c}}}{2}\left((2-2\tau+\lambda)(1-2\tau+\lambda)+\frac{37}{6}\right)T^{2-2\tau}\right.\nonumber\\
&\phantom{xxxxxxxxxxxxxxxxxxxxxx}\left.+\frac{\mathrm{F}\left(\mathrm{c},T_0,\tau\right)}{\frac{1}{2}-\frac{\mathrm{c}}{\log(T_0)}}T^{1-2\tau}\right).
\end{align}

Finally, for $t\in[0,1]$, $\lambda<\frac{1}{2}$ and $\tau\geq\frac{1}{4}$, we can bound $|\Gamma(2-2\tau+\lambda+it)|$ by $\Gamma(2)=1$, and $\cos\left(\frac{\pi(2-2\tau+\lambda+it)}{2}\right)$ by $e^{\frac{\pi}{2}}$, since $e^{\frac{-\pi x}{2}}\leq e^{\frac{\pi}{2}}$ for $x\in[-1,1]$, so that by \eqref{eq:Yn}
\begin{equation}\label{Yn:est}
|Y_{n}|\leq\frac{e^{\frac{\pi}{2}}}{(2\pi)^{2-2\tau+\lambda}}\frac{1}{n^{2-2\tau+\lambda}}.
\end{equation} 
The results is concluded by combining $\eqref{Xn:est1}$, \eqref{Xn:est2} and $\eqref{Yn:est}$ together and using that $|K_n|\leq 2|X_{n}|+2|Y_{n}|$.
\end{proof}

\begin{proofbold}{Proposition~\ref{pr:K}}

\footnotesize
\begin{mysage}
###
# Estimating K.
###
# Constants multiplying the sums in {eq:Kalot}.
def Fu_Kxi_1_0(simin,simax,ci,ta):
    simin = RIF( simin )
    simax = RIF( simax )
    ci = RIF( ci )
    ta = RIF( ta )
    result = 2*Fu_Q(simin,simax,ci,ta)
    return RIF(result)
def Fu_Kxi_1_1m2t(simin,simax,ci,ta):
    simin = RIF( simin )
    simax = RIF( simax )
    ci = RIF( ci )
    ta = RIF( ta )
    result = 2*Fu_R_2(simin,simax,ci,ta)
    return RIF(result)
def Fu_Kxi_2_2m2t(simin,simax,ci,ta):
    simin = RIF( simin )
    simax = RIF( simax )
    ci = RIF( ci )
    ta = RIF( ta )
    result = 2*Fu_Rp_1(simin,simax,ci,ta)
    return RIF(result)
def Fu_Kxi_3_32m2t(simin,simax,ci,ta):
    simin = RIF( simin )
    simax = RIF( simax )
    ci = RIF( ci )
    ta = RIF( ta )
    result = 2*Fu_R_1(simin,simax,ci,ta)*exp_int(log_int(2*pi_int)*(1+ci/log_int(ta)))/exp_int(ci)
    return RIF(result)
def Fu_Kxi_4_32m2t(simin,simax,ci,ta):
    simin = RIF( simin )
    simax = RIF( simax )
    ci = RIF( ci )
    ta = RIF( ta )
    result = 2*Fu_R_1(simin,simax,ci,ta)/log_int(2)
    return RIF(result)
###
# All the following coefficients per order inside the sums for 0<tau<1/2
# are written as constants multiplying [1,1/tau,1/(1/2-tau)].
# The coefficients for tau=1/2 are still constants.
###
# Coefficients per order inside the first sum in {eq:Kalot}.
def Fu_K0_1_lq(ci,ta):
    ci = RIF( ci )
    ta = RIF( ta )
    result1 = exp_int(log_int(2*pi_int)*(ci/log_int(ta)))*(2+ci/log_int(ta)) \
             /(ci^2*exp_int(ci))
    result2 = 0
    result3 = 0
    return [RIF(result1),RIF(result2),RIF(result3)]
def Fu_K0_1_l(ci,ta):
    ci = RIF( ci )
    ta = RIF( ta )
    result1 = 0
    result2 = 0
    result3 = exp_int(log_int(2*pi_int)*(ci/log_int(ta)))*(1/2)/(ci*exp_int(ci))
    return [RIF(result1),RIF(result2),RIF(result3)]
def Fu_K0_1_0(ci,ta):
    ci = RIF( ci )
    ta = RIF( ta )
    result1 = exp_int(log_int(2*pi_int)*(ci/log_int(ta)))/exp_int(ci)
    result2 = 0
    result3 = 0
    return [RIF(result1),RIF(result2),RIF(result3)]
def Fu_K0_1_m1l(ci,ta):
    ci = RIF( ci )
    ta = RIF( ta )
    result1 = exp_int(log_int(2*pi_int)*(1+ci/log_int(ta)))*(3+ci/log_int(ta)) \
              /(ci*exp_int(ci))
    result2 = 0
    result3 = 0
    return [RIF(result1),RIF(result2),RIF(result3)]
def Fu_K0_1_m1(ci,ta):
    ci = RIF( ci )
    ta = RIF( ta )
    result1 = exp_int(log_int(2*pi_int)*(1+ci/log_int(ta)))/exp_int(ci)
    result2 = 0
    result3 = 0
    return [RIF(result1),RIF(result2),RIF(result3)]
def Fu_K4_1_lq(ci,ta):
    ci = RIF( ci )
    ta = RIF( ta )
    result1 = exp_int(log_int(2*pi_int)*(ci/log_int(ta)))*(2-2*1/4+ci/log_int(ta)) \
             /(ci^2*exp_int(ci))
    result2 = 0
    result3 = 0
    return [RIF(result1),RIF(result2),RIF(result3)]
def Fu_K4_1_l(ci,ta):
    ci = RIF( ci )
    ta = RIF( ta )
    result1 = 0
    result2 = 0
    result3 = exp_int(log_int(2*pi_int)*(ci/log_int(ta)))*(1/2)/(ci*exp_int(ci))
    return [RIF(result1),RIF(result2),RIF(result3)]
def Fu_K4_1_0(ci,ta):
    ci = RIF( ci )
    ta = RIF( ta )
    result1 = exp_int(log_int(2*pi_int)*(ci/log_int(ta)))/exp_int(ci)
    result2 = 0
    result3 = 0
    return [RIF(result1),RIF(result2),RIF(result3)]
def Fu_K4_1_m1l(ci,ta):
    ci = RIF( ci )
    ta = RIF( ta )
    result1 = exp_int(log_int(2*pi_int)*(1+ci/log_int(ta)))*(3-2*1/4+ci/log_int(ta)) \
              /(ci*exp_int(ci))
    result2 = 0
    result3 = 0
    return [RIF(result1),RIF(result2),RIF(result3)]
def Fu_K4_1_m1(ci,ta):
    ci = RIF( ci )
    ta = RIF( ta )
    result1 = exp_int(log_int(2*pi_int)*(1+ci/log_int(ta)))/exp_int(ci)
    result2 = 0
    result3 = 0
    return [RIF(result1),RIF(result2),RIF(result3)]
def Fu_K2_1_lq(ci,ta):
    ci = RIF( ci )
    ta = RIF( ta )
    result = exp_int(log_int(2*pi_int)*(ci/log_int(ta)))*(ci+1)/(ci^2*exp_int(ci))
    return RIF(result)
def Fu_K2_1_0(ci,ta):
    ci = RIF( ci )
    ta = RIF( ta )
    result = exp_int(log_int(2*pi_int)*(ci/log_int(ta)))/exp_int(ci)
    return RIF(result)
def Fu_K2_1_m1l(ci,ta):
    ci = RIF( ci )
    ta = RIF( ta )
    result = 5*exp_int(log_int(2*pi_int)*(1+ci/log_int(ta)))/(3*ci*exp_int(ci))
    return RIF(result)
def Fu_K2_1_m1(ci,ta):
    ci = RIF( ci )
    ta = RIF( ta )
    result = exp_int(log_int(2*pi_int)*(1+ci/log_int(ta)))/exp_int(ci)
    return RIF(result)
###
# Coefficients per order inside the second sum in {eq:Kalot}.
def Fu_K0_2_m12(ci,ta):
    ci = RIF( ci )
    ta = RIF( ta )
    result = cw_prod( exp_int(log_int(2*pi_int)*(ci/log_int(ta)))/exp_int(ci) , zea_2m2s )
    return result
def Fu_K0_2_tm1(ci,ta):
    ci = RIF( ci )
    ta = RIF( ta )
    result = cw_prod( 2*exp_int(log_int(2*pi_int)*(1+ci/log_int(ta)))/exp_int(ci) , As_i0 )
    return result
def Fu_K0_2_tm32(ci,ta):
    ci = RIF( ci )
    ta = RIF( ta )
    result = cw_prod( exp_int(log_int(2*pi_int)*(1+ci/log_int(ta)))/exp_int(ci) , As_i0 )
    return result
def Fu_K4_2_m12(ci,ta):
    ci = RIF( ci )
    ta = RIF( ta )
    result = cw_prod( exp_int(log_int(2*pi_int)*(ci/log_int(ta)))/exp_int(ci) , zea_2m2s )
    return result
def Fu_K4_2_tm1(ci,ta):
    ci = RIF( ci )
    ta = RIF( ta )
    result = cw_prod( 2*exp_int(log_int(2*pi_int)*(3/4+ci/log_int(ta)))/exp_int(ci) , As_i4 )
    return result
def Fu_K4_2_tm32(ci,ta):
    ci = RIF( ci )
    ta = RIF( ta )
    result = cw_prod( exp_int(log_int(2*pi_int)*(3/4+ci/log_int(ta)))*(3/4)/exp_int(ci) , As_i4 )
    return result
def Fu_K2_2_m12l(ci,ta):
    ci = RIF( ci )
    ta = RIF( ta )
    result = exp_int(log_int(2*pi_int)*(ci/log_int(ta)))/exp_int(ci)
    return RIF(result)
def Fu_K2_2_m12(ci,ta):
    ci = RIF( ci )
    ta = RIF( ta )
    result = ((2*gamma_int-log_int(2*pi_int))*exp_int(log_int(2*pi_int)*(ci/log_int(ta))) \
             +2*exp_int(log_int(2*pi_int)*(1/2+ci/log_int(ta)))*A12)/exp_int(ci)
    return RIF(result)
def Fu_K2_2_m1(ci,ta):
    ci = RIF( ci )
    ta = RIF( ta )
    result = (exp_int(log_int(2*pi_int)*(ci/log_int(ta))) \
              +1/2*exp_int(log_int(2*pi_int)*(1/2+ci/log_int(ta)))*A12)/exp_int(ci)
    return RIF(result)
###
# Coefficients per order inside the third sum in {eq:Kalot}.
K0_3_lq = Ps_1
K0_3_l = Fu_Ps_2(0,1/2)
K0_3_tm12 = Fu_Ps_3(0,1/2)
K0_3_tm1l = Fu_Ps_4(0,1/2)
K0_3_tm1 = Fu_Ps_5(0,1/2)
K4_3_lq = Ps_1
K4_3_l = Fu_Ps_2(1/4,1/2)
K4_3_tm12 = Fu_Ps_3(1/4,1/2)
K4_3_tm1l = Fu_Ps_4(1/4,1/2)
K4_3_tm1 = Fu_Ps_5(1/4,1/2)
K2_3_lq = P2_1
K2_3_l = P2_2
def Fu_K2_3_0(ta):
    return Fu_P2_3(ta)
K2_3_m12l = P2_4
K2_3_m12 = P2_5
###
# Coefficients per order inside the fourth sum in {eq:Kalot}.
def Fu_K0_4_lq(ci,ta):
    ci = RIF( ci )
    ta = RIF( ta )
    result1 = exp_int(log_int(pi_int)*(ci/log_int(ta)))*(2+ci/log_int(ta)) \
              /(ci^2*exp_int(ci))
    result2 = 0
    result3 = 0
    return [RIF(result1),RIF(result2),RIF(result3)]
def Fu_K0_4_l(ci,ta):
    ci = RIF( ci )
    ta = RIF( ta )
    result1 = 0
    result2 = 0
    result3 = exp_int(log_int(pi_int)*(ci/log_int(ta)))*(1/2)/(ci*exp_int(ci))
    return [RIF(result1),RIF(result2),RIF(result3)]
def Fu_K0_4_0(ci,ta):
    ci = RIF( ci )
    ta = RIF( ta )
    result1 = exp_int(log_int(pi_int)*(ci/log_int(ta)))/exp_int(ci)
    result2 = 0
    result3 = 0
    return [RIF(result1),RIF(result2),RIF(result3)]
def Fu_K0_4_m1l(ci,ta):
    ci = RIF( ci )
    ta = RIF( ta )
    result1 = exp_int(log_int(pi_int)*(1+ci/log_int(ta)))*(3+ci/log_int(ta)) \
              /(ci*exp_int(ci))
    result2 = 0
    result3 = 0
    return [RIF(result1),RIF(result2),RIF(result3)]
def Fu_K0_4_m1(ci,ta):
    ci = RIF( ci )
    ta = RIF( ta )
    result1 = exp_int(log_int(pi_int)*(1+ci/log_int(ta)))/exp_int(ci)
    result2 = 0
    result3 = 0
    return [RIF(result1),RIF(result2),RIF(result3)]
def Fu_K4_4_lq(ci,ta):
    ci = RIF( ci )
    ta = RIF( ta )
    result1 = exp_int(log_int(pi_int)*(ci/log_int(ta)))*(2-2*1/4+ci/log_int(ta)) \
             /(ci^2*exp_int(ci))
    result2 = 0
    result3 = 0
    return [RIF(result1),RIF(result2),RIF(result3)]
def Fu_K4_4_l(ci,ta):
    ci = RIF( ci )
    ta = RIF( ta )
    result1 = 0
    result2 = 0
    result3 = exp_int(log_int(pi_int)*(ci/log_int(ta)))*(1/2)/(ci*exp_int(ci))
    return [RIF(result1),RIF(result2),RIF(result3)]
def Fu_K4_4_0(ci,ta):
    ci = RIF( ci )
    ta = RIF( ta )
    result1 = exp_int(log_int(pi_int)*(ci/log_int(ta)))/exp_int(ci)
    result2 = 0
    result3 = 0
    return [RIF(result1),RIF(result2),RIF(result3)]
def Fu_K4_4_m1l(ci,ta):
    ci = RIF( ci )
    ta = RIF( ta )
    result1 = exp_int(log_int(pi_int)*(1+ci/log_int(ta)))*(3-2*1/4+ci/log_int(ta)) \
              /(ci*exp_int(ci))
    result2 = 0
    result3 = 0
    return [RIF(result1),RIF(result2),RIF(result3)]
def Fu_K4_4_m1(ci,ta):
    ci = RIF( ci )
    ta = RIF( ta )
    result1 = exp_int(log_int(pi_int)*(1+ci/log_int(ta)))/exp_int(ci)
    result2 = 0
    result3 = 0
    return [RIF(result1),RIF(result2),RIF(result3)]
def Fu_K2_4_lq(ci,ta):
    ci = RIF( ci )
    ta = RIF( ta )
    result = exp_int(log_int(pi_int)*(ci/log_int(ta)))*(ci+1)/(ci^2*exp_int(ci))
    return RIF(result)
def Fu_K2_4_l(ci,ta):
    ci = RIF( ci )
    ta = RIF( ta )
    result = exp_int(log_int(pi_int)*(ci/log_int(ta))) \
             *(gamma_int+1-log_int(pi_int))/(ci*exp_int(ci))
    return RIF(result)
def Fu_K2_4_0(ci,ta):
    ci = RIF( ci )
    ta = RIF( ta )
    result = exp_int(log_int(pi_int)*(ci/log_int(ta)))/exp_int(ci)
    return RIF(result)
def Fu_K2_4_m1l(ci,ta):
    ci = RIF( ci )
    ta = RIF( ta )
    result = 5*exp_int(log_int(pi_int)*(1+ci/log_int(ta)))/(3*ci*exp_int(ci))
    return RIF(result)
def Fu_K2_4_m1(ci,ta):
    ci = RIF( ci )
    ta = RIF( ta )
    result = exp_int(log_int(pi_int)*(1+ci/log_int(ta)))/exp_int(ci)
    return RIF(result)
###
# Coefficients per order of the whole K, for tau in (0,1/2).
def Fu_K0__32m2tlq(ci,ta):
    ci = RIF( ci )
    ta = RIF( ta )
    x3 = cw_prod( Fu_Kxi_3_32m2t(0,1/2,ci,ta) , K0_3_lq )
    x4 = cw_prod( Fu_Kxi_4_32m2t(0,1/2,ci,ta) , Fu_K0_4_lq(ci,ta) )
    x = cw_sum( x3 , x4 )
    return x
def Fu_K0__32m2tl(ci,ta):
    ci = RIF( ci )
    ta = RIF( ta )
    x3 = cw_prod( Fu_Kxi_3_32m2t(0,1/2,ci,ta) , K0_3_l )
    x4 = cw_prod( Fu_Kxi_4_32m2t(0,1/2,ci,ta) , Fu_K0_4_l(ci,ta) )
    x = cw_sum( x3 , x4 )
    return x
def Fu_K0__32m2t(ci,ta):
    ci = RIF( ci )
    ta = RIF( ta )
    x2 = cw_prod( Fu_Kxi_2_2m2t(0,1/2,ci,ta) , Fu_K0_2_m12(ci,ta) )
    x4 = cw_prod( Fu_Kxi_4_32m2t(0,1/2,ci,ta) , Fu_K0_4_0(ci,ta) )
    x = cw_sum( x2 , x4 )
    return x
def Fu_K0__1mt(ci,ta):
    ci = RIF( ci )
    ta = RIF( ta )
    x2 = cw_prod( Fu_Kxi_2_2m2t(0,1/2,ci,ta) , Fu_K0_2_tm1(ci,ta) )
    x3 = cw_prod( Fu_Kxi_3_32m2t(0,1/2,ci,ta) , K0_3_tm12 )
    x = cw_sum( x2 , x3 )
    return x
def Fu_K0__1m2tlq(ci,ta):
    ci = RIF( ci )
    ta = RIF( ta )
    x1 = cw_prod( Fu_Kxi_1_1m2t(0,1/2,ci,ta) , Fu_K0_1_lq(ci,ta) )
    return x1
def Fu_K0__1m2tl(ci,ta):
    ci = RIF( ci )
    ta = RIF( ta )
    x1 = cw_prod( Fu_Kxi_1_1m2t(0,1/2,ci,ta) , Fu_K0_1_l(ci,ta) )
    return x1
def Fu_K0__1m2t(ci,ta):
    ci = RIF( ci )
    ta = RIF( ta )
    x1 = cw_prod( Fu_Kxi_1_1m2t(0,1/2,ci,ta) , Fu_K0_1_0(ci,ta) )
    return x1
def Fu_K0__12mtl(ci,ta):
    ci = RIF( ci )
    ta = RIF( ta )
    x3 = cw_prod( Fu_Kxi_3_32m2t(0,1/2,ci,ta) , K0_3_tm1l )
    return x3
def Fu_K0__12mt(ci,ta):
    ci = RIF( ci )
    ta = RIF( ta )
    x2 = cw_prod( Fu_Kxi_2_2m2t(0,1/2,ci,ta) , Fu_K0_2_tm32(ci,ta) )
    x3 = cw_prod( Fu_Kxi_3_32m2t(0,1/2,ci,ta) , K0_3_tm1 )
    x = cw_sum( x2 , x3 )
    return x
def Fu_K0__12m2tl(ci,ta):
    ci = RIF( ci )
    ta = RIF( ta )
    x4 = cw_prod( Fu_Kxi_4_32m2t(0,1/2,ci,ta) , Fu_K0_4_m1l(ci,ta) )
    return x4
def Fu_K0__12m2t(ci,ta):
    ci = RIF( ci )
    ta = RIF( ta )
    x4 = cw_prod( Fu_Kxi_4_32m2t(0,1/2,ci,ta) , Fu_K0_4_m1(ci,ta) )
    return x4
def Fu_K0__lq(ci,ta):
    ci = RIF( ci )
    ta = RIF( ta )
    x1 = cw_prod( Fu_Kxi_1_0(0,1/2,ci,ta) , Fu_K0_1_lq(ci,ta) )
    return x1
def Fu_K0__l(ci,ta):
    ci = RIF( ci )
    ta = RIF( ta )
    x1 = cw_prod( Fu_Kxi_1_0(0,1/2,ci,ta) , Fu_K0_1_l(ci,ta) )
    return x1
def Fu_K0__0(ci,ta):
    ci = RIF( ci )
    ta = RIF( ta )
    x1 = cw_prod( Fu_Kxi_1_0(0,1/2,ci,ta) , Fu_K0_1_0(ci,ta) )
    return x1
def Fu_K0__m2tl(ci,ta):
    ci = RIF( ci )
    ta = RIF( ta )
    x1 = cw_prod( Fu_Kxi_1_1m2t(0,1/2,ci,ta) , Fu_K0_1_m1l(ci,ta) )
    return x1
def Fu_K0__m2t(ci,ta):
    ci = RIF( ci )
    ta = RIF( ta )
    x1 = cw_prod( Fu_Kxi_1_1m2t(0,1/2,ci,ta) , Fu_K0_1_m1(ci,ta) )
    return x1
def Fu_K0__m1l(ci,ta):
    ci = RIF( ci )
    ta = RIF( ta )
    x1 = cw_prod( Fu_Kxi_1_0(0,1/2,ci,ta) , Fu_K0_1_m1l(ci,ta) )
    return x1
def Fu_K0__m1(ci,ta):
    ci = RIF( ci )
    ta = RIF( ta )
    x1 = cw_prod( Fu_Kxi_1_0(0,1/2,ci,ta) , Fu_K0_1_m1(ci,ta) )
    return x1
###
# Coefficients per order of the whole K, for tau in [1/4,1/2).
def Fu_K4__32m2tlq(ci,ta):
    ci = RIF( ci )
    ta = RIF( ta )
    x3 = cw_prod( Fu_Kxi_3_32m2t(1/4,1/2,ci,ta) , K4_3_lq )
    x4 = cw_prod( Fu_Kxi_4_32m2t(1/4,1/2,ci,ta) , Fu_K4_4_lq(ci,ta) )
    x = cw_sum( x3 , x4 )
    return x
def Fu_K4__32m2tl(ci,ta):
    ci = RIF( ci )
    ta = RIF( ta )
    x3 = cw_prod( Fu_Kxi_3_32m2t(1/4,1/2,ci,ta) , K4_3_l )
    x4 = cw_prod( Fu_Kxi_4_32m2t(1/4,1/2,ci,ta) , Fu_K4_4_l(ci,ta) )
    x = cw_sum( x3 , x4 )
    return x
def Fu_K4__32m2t(ci,ta):
    ci = RIF( ci )
    ta = RIF( ta )
    x2 = cw_prod( Fu_Kxi_2_2m2t(1/4,1/2,ci,ta) , Fu_K4_2_m12(ci,ta) )
    x4 = cw_prod( Fu_Kxi_4_32m2t(1/4,1/2,ci,ta) , Fu_K4_4_0(ci,ta) )
    x = cw_sum( x2 , x4 )
    return x
def Fu_K4__1mt(ci,ta):
    ci = RIF( ci )
    ta = RIF( ta )
    x2 = cw_prod( Fu_Kxi_2_2m2t(1/4,1/2,ci,ta) , Fu_K4_2_tm1(ci,ta) )
    x3 = cw_prod( Fu_Kxi_3_32m2t(1/4,1/2,ci,ta) , K4_3_tm12 )
    x = cw_sum( x2 , x3 )
    return x
def Fu_K4__1m2tlq(ci,ta):
    ci = RIF( ci )
    ta = RIF( ta )
    x1 = cw_prod( Fu_Kxi_1_1m2t(1/4,1/2,ci,ta) , Fu_K4_1_lq(ci,ta) )
    return x1
def Fu_K4__1m2tl(ci,ta):
    ci = RIF( ci )
    ta = RIF( ta )
    x1 = cw_prod( Fu_Kxi_1_1m2t(1/4,1/2,ci,ta) , Fu_K4_1_l(ci,ta) )
    return x1
def Fu_K4__1m2t(ci,ta):
    ci = RIF( ci )
    ta = RIF( ta )
    x1 = cw_prod( Fu_Kxi_1_1m2t(1/4,1/2,ci,ta) , Fu_K4_1_0(ci,ta) )
    return x1
def Fu_K4__12mtl(ci,ta):
    ci = RIF( ci )
    ta = RIF( ta )
    x3 = cw_prod( Fu_Kxi_3_32m2t(1/4,1/2,ci,ta) , K4_3_tm1l )
    return x3
def Fu_K4__12mt(ci,ta):
    ci = RIF( ci )
    ta = RIF( ta )
    x2 = cw_prod( Fu_Kxi_2_2m2t(1/4,1/2,ci,ta) , Fu_K4_2_tm32(ci,ta) )
    x3 = cw_prod( Fu_Kxi_3_32m2t(1/4,1/2,ci,ta) , K4_3_tm1 )
    x = cw_sum( x2 , x3 )
    return x
def Fu_K4__12m2tl(ci,ta):
    ci = RIF( ci )
    ta = RIF( ta )
    x4 = cw_prod( Fu_Kxi_4_32m2t(1/4,1/2,ci,ta) , Fu_K4_4_m1l(ci,ta) )
    return x4
def Fu_K4__12m2t(ci,ta):
    ci = RIF( ci )
    ta = RIF( ta )
    x4 = cw_prod( Fu_Kxi_4_32m2t(1/4,1/2,ci,ta) , Fu_K4_4_m1(ci,ta) )
    return x4
def Fu_K4__lq(ci,ta):
    ci = RIF( ci )
    ta = RIF( ta )
    x1 = cw_prod( Fu_Kxi_1_0(1/4,1/2,ci,ta) , Fu_K4_1_lq(ci,ta) )
    return x1
def Fu_K4__l(ci,ta):
    ci = RIF( ci )
    ta = RIF( ta )
    x1 = cw_prod( Fu_Kxi_1_0(1/4,1/2,ci,ta) , Fu_K4_1_l(ci,ta) )
    return x1
def Fu_K4__0(ci,ta):
    ci = RIF( ci )
    ta = RIF( ta )
    x1 = cw_prod( Fu_Kxi_1_0(1/4,1/2,ci,ta) , Fu_K4_1_0(ci,ta) )
    return x1
def Fu_K4__m2tl(ci,ta):
    ci = RIF( ci )
    ta = RIF( ta )
    x1 = cw_prod( Fu_Kxi_1_1m2t(1/4,1/2,ci,ta) , Fu_K4_1_m1l(ci,ta) )
    return x1
def Fu_K4__m2t(ci,ta):
    ci = RIF( ci )
    ta = RIF( ta )
    x1 = cw_prod( Fu_Kxi_1_1m2t(1/4,1/2,ci,ta) , Fu_K4_1_m1(ci,ta) )
    return x1
def Fu_K4__m1l(ci,ta):
    ci = RIF( ci )
    ta = RIF( ta )
    x1 = cw_prod( Fu_Kxi_1_0(1/4,1/2,ci,ta) , Fu_K4_1_m1l(ci,ta) )
    return x1
def Fu_K4__m1(ci,ta):
    ci = RIF( ci )
    ta = RIF( ta )
    x1 = cw_prod( Fu_Kxi_1_0(1/4,1/2,ci,ta) , Fu_K4_1_m1(ci,ta) )
    return x1
###
# Coefficients per order of the whole K, for tau=1/2.
def Fu_K2__12lq(ci,ta):
    ci = RIF( ci )
    ta = RIF( ta )
    result = Fu_Kxi_3_32m2t(1/2,1/2,ci,ta)*K2_3_lq \
             + Fu_Kxi_4_32m2t(1/2,1/2,ci,ta)*Fu_K2_4_lq(ci,ta)
    return RIF(result)
def Fu_K2__12l(ci,ta):
    ci = RIF( ci )
    ta = RIF( ta )
    result = Fu_Kxi_2_2m2t(1/2,1/2,ci,ta)*Fu_K2_2_m12l(ci,ta) \
             + Fu_Kxi_3_32m2t(1/2,1/2,ci,ta)*K2_3_l + Fu_Kxi_4_32m2t(1/2,1/2,ci,ta)*Fu_K2_4_l(ci,ta)
    return RIF(result)
def Fu_K2__12(ci,ta):
    ci = RIF( ci )
    ta = RIF( ta )
    result = Fu_Kxi_2_2m2t(1/2,1/2,ci,ta)*Fu_K2_2_m12(ci,ta) \
             + Fu_Kxi_3_32m2t(1/2,1/2,ci,ta)*Fu_K2_3_0(ta) \
             + Fu_Kxi_4_32m2t(1/2,1/2,ci,ta)*Fu_K2_4_0(ci,ta)
    return RIF(result)
def Fu_K2__lq(ci,ta):
    ci = RIF( ci )
    ta = RIF( ta )
    result = (Fu_Kxi_1_0(1/2,1/2,ci,ta)+Fu_Kxi_1_1m2t(1/2,1/2,ci,ta)) \
             *Fu_K2_1_lq(ci,ta)
    return RIF(result)
def Fu_K2__l(ci,ta):
    ci = RIF( ci )
    ta = RIF( ta )
    result = Fu_Kxi_3_32m2t(1/2,1/2,ci,ta)*K2_3_m12l
    return RIF(result)
def Fu_K2__0(ci,ta):
    ci = RIF( ci )
    ta = RIF( ta )
    result = (Fu_Kxi_1_0(1/2,1/2,ci,ta)+Fu_Kxi_1_1m2t(1/2,1/2,ci,ta)) \
             *Fu_K2_1_0(ci,ta) \
             + Fu_Kxi_2_2m2t(1/2,1/2,ci,ta)*Fu_K2_2_m1(ci,ta) \
             + Fu_Kxi_3_32m2t(1/2,1/2,ci,ta)*K2_3_m12
    return RIF(result)
def Fu_K2__m12l(ci,ta):
    ci = RIF( ci )
    ta = RIF( ta )
    result = Fu_Kxi_4_32m2t(1/2,1/2,ci,ta)*Fu_K2_4_m1l(ci,ta)
    return RIF(result)
def Fu_K2__m12(ci,ta):
    ci = RIF( ci )
    ta = RIF( ta )
    result = Fu_Kxi_4_32m2t(1/2,1/2,ci,ta)*Fu_K2_4_m1(ci,ta)
    return RIF(result)
def Fu_K2__m1l(ci,ta):
    ci = RIF( ci )
    ta = RIF( ta )
    result = (Fu_Kxi_1_0(1/2,1/2,ci,ta)+Fu_Kxi_1_1m2t(1/2,1/2,ci,ta)) \
             *Fu_K2_1_m1l(ci,ta)
    return RIF(result)
def Fu_K2__m1(ci,ta):
    ci = RIF( ci )
    ta = RIF( ta )
    result = (Fu_Kxi_1_0(1/2,1/2,ci,ta)+Fu_Kxi_1_1m2t(1/2,1/2,ci,ta)) \
             *Fu_K2_1_m1(ci,ta)
    return RIF(result)
###
# Rounding coefficients to the order T^(3/2-2tau)*log(T)^2 for tau in (0,1/2).
def Fu_K0__32m2tlq_final(ci,ta):
    ci = RIF( ci )
    ta = RIF( ta )
    x_32m2tlq = Fu_K0__32m2tlq(ci,ta)
    x_32m2tl = cw_prod( 1/log_int(ta) , Fu_K0__32m2tl(ci,ta) )
    x_32m2t = cw_prod( 1/log_int(ta)^2 , Fu_K0__32m2t(ci,ta) )
    x_1mt = cw_prod( 1/log_int(ta)^2 , Fu_K0__1mt(ci,ta) )
    x_1m2tlq = cw_prod( 1/sqrt_int(ta) , Fu_K0__1m2tlq(ci,ta) )
    x_1m2tl = cw_prod( 1/(log_int(ta)*sqrt_int(ta)) , Fu_K0__1m2tl(ci,ta) )
    x_1m2t = cw_prod( 1/(log_int(ta)^2*sqrt_int(ta)) , Fu_K0__1m2t(ci,ta) )
    x_12mtl = cw_prod( 1/(log_int(ta)*sqrt_int(ta)) , Fu_K0__12mtl(ci,ta) )
    x_12mt = cw_prod( 1/(log_int(ta)^2*sqrt_int(ta)) , Fu_K0__12mt(ci,ta) )
    x_12m2tl = cw_prod( 1/(log_int(ta)*ta) , Fu_K0__12m2tl(ci,ta) )
    x_12m2t = cw_prod( 1/(log_int(ta)^2*ta) , Fu_K0__12m2t(ci,ta) )
    x_lq = cw_prod( 1/sqrt_int(ta) , Fu_K0__lq(ci,ta) )
    x_l = cw_prod( 1/(log_int(ta)*sqrt_int(ta)) , Fu_K0__l(ci,ta) )
    x_0 = cw_prod( 1/(log_int(ta)^2*sqrt_int(ta)) , Fu_K0__0(ci,ta) )
    x_m2tl = cw_prod( 1/(log_int(ta)*exp_int(3/2*log_int(ta))) , Fu_K0__m2tl(ci,ta) )
    x_m2t = cw_prod( 1/(log_int(ta)^2*exp_int(3/2*log_int(ta))) , Fu_K0__m2t(ci,ta) )
    x_m1l = cw_prod( 1/(log_int(ta)*exp_int(3/2*log_int(ta))) , Fu_K0__m1l(ci,ta) )
    x_m1 = cw_prod( 1/(log_int(ta)^2*exp_int(3/2*log_int(ta))) , Fu_K0__m1(ci,ta) )
    x = cw_sum( x_32m2tlq , x_32m2tl )
    x = cw_sum( x , x_32m2t )
    x = cw_sum( x , x_1mt )
    x = cw_sum( x , x_1m2tlq )
    x = cw_sum( x , x_1m2tl )
    x = cw_sum( x , x_1m2t )
    x = cw_sum( x , x_12mtl )
    x = cw_sum( x , x_12mt )
    x = cw_sum( x , x_12m2tl )
    x = cw_sum( x , x_12m2t )
    x = cw_sum( x , x_lq )
    x = cw_sum( x , x_l )
    x = cw_sum( x , x_0 )
    x = cw_sum( x , x_m2tl )
    x = cw_sum( x , x_m2t )
    x = cw_sum( x , x_m1l )
    x = cw_sum( x , x_m1 )
    return x
###
# Rounding coefficients to the order T^(3/2-2tau)*log(T)^2 for tau in [1/4,1/2).
def Fu_K4__32m2tlq_final(ci,ta):
    ci = RIF( ci )
    ta = RIF( ta )
    x_32m2tlq = Fu_K4__32m2tlq(ci,ta)
    x_32m2tl = cw_prod( 1/log_int(ta) , Fu_K4__32m2tl(ci,ta) )
    x_32m2t = cw_prod( 1/log_int(ta)^2 , Fu_K4__32m2t(ci,ta) )
    x_1mt = cw_prod( 1/log_int(ta)^2 , Fu_K4__1mt(ci,ta) )
    x_1m2tlq = cw_prod( 1/sqrt_int(ta) , Fu_K4__1m2tlq(ci,ta) )
    x_1m2tl = cw_prod( 1/(log_int(ta)*sqrt_int(ta)) , Fu_K4__1m2tl(ci,ta) )
    x_1m2t = cw_prod( 1/(log_int(ta)^2*sqrt_int(ta)) , Fu_K4__1m2t(ci,ta) )
    x_12mtl = cw_prod( 1/(log_int(ta)*sqrt_int(ta)) , Fu_K4__12mtl(ci,ta) )
    x_12mt = cw_prod( 1/(log_int(ta)^2*sqrt_int(ta)) , Fu_K4__12mt(ci,ta) )
    x_12m2tl = cw_prod( 1/(log_int(ta)*ta) , Fu_K4__12m2tl(ci,ta) )
    x_12m2t = cw_prod( 1/(log_int(ta)^2*ta) , Fu_K4__12m2t(ci,ta) )
    x_lq = cw_prod( 1/sqrt_int(ta) , Fu_K4__lq(ci,ta) )
    x_l = cw_prod( 1/(log_int(ta)*sqrt_int(ta)) , Fu_K4__l(ci,ta) )
    x_0 = cw_prod( 1/(log_int(ta)^2*sqrt_int(ta)) , Fu_K4__0(ci,ta) )
    x_m2tl = cw_prod( 1/(log_int(ta)*exp_int(3/2*log_int(ta))) , Fu_K4__m2tl(ci,ta) )
    x_m2t = cw_prod( 1/(log_int(ta)^2*exp_int(3/2*log_int(ta))) , Fu_K4__m2t(ci,ta) )
    x_m1l = cw_prod( 1/(log_int(ta)*exp_int(3/2*log_int(ta))) , Fu_K4__m1l(ci,ta) )
    x_m1 = cw_prod( 1/(log_int(ta)^2*exp_int(3/2*log_int(ta))) , Fu_K4__m1(ci,ta) )
    x = cw_sum( x_32m2tl , x_32m2t )
    x = cw_sum( x , x_1mt )
    x = cw_sum( x , x_1m2tl )
    x = cw_sum( x , x_1m2t )
    x = cw_sum( x , x_12mtl )
    x = cw_sum( x , x_12mt )
    x = cw_sum( x , x_12m2tl )
    x = cw_sum( x , x_12m2t )
    x = cw_sum( x , x_l )
    x = cw_sum( x , x_0 )
    x = cw_sum( x , x_m2tl )
    x = cw_sum( x , x_m2t )
    x = cw_sum( x , x_m1l )
    x = cw_sum( x , x_m1 )
    return x
###
# Rounding coefficients to the order sqrt(T)*log(T) for tau=1/2.
# Since T0>=e^2, we can use log(T)/sqrt(T)<=log(T0)/sqrt(T0).
def Fu_K2__12lq_final(ci,ta):
    return Fu_K2__12lq(ci,ta)
def Fu_K2__12l_final(ci,ta):
    ci = RIF( ci )
    ta = RIF( ta )
    result = Fu_K2__12l(ci,ta) + Fu_K2__12(ci,ta)/log_int(ta) \
             + Fu_K2__lq(ci,ta)*log_int(ta)/sqrt_int(ta) \
             + Fu_K2__l(ci,ta)/sqrt_int(ta) \
             + Fu_K2__0(ci,ta)/(log_int(ta)*sqrt_int(ta)) + Fu_K2__m12l(ci,ta)/ta \
             + Fu_K2__m12(ci,ta)/(log_int(ta)*ta) \
             + Fu_K2__m1l(ci,ta)/exp_int(3/2*log_int(ta)) \
             + Fu_K2__m1(ci,ta)/(log_int(ta)*exp_int(3/2*log_int(ta)))
    return RIF(result)
###
# Actual values for our choice of c,T0.
K0__32m2tlq_final_vec = Fu_K0__32m2tlq_final(c_la,T0)
K4__32m2tlq_final_vec = Fu_K4__32m2tlq_final(c_la,T0)
K4__32m2tlq_final_c = roundup( K4__32m2tlq_final_vec[0] \
                              + K4__32m2tlq_final_vec[1]/(1/4) , digits )
K4__32m2tlq_final_2 = roundup( K4__32m2tlq_final_vec[2] , digits )
K2__12lq_final = roundup( Fu_K2__12lq_final(c_la,T0) , digits )
K2__12l_final = roundup( Fu_K2__12l_final(c_la,T0) , digits )
\end{mysage}
\normalsize

We have that $K=\sum_{n>\frac{T}{2\pi}}d_{1-2\tau}(n)K_{n}$. As per Lemma \ref{Kn:estimation} and Remark \ref{crucial2}, we split that sum into three parts, according to whether $\frac{T}{2\pi}<n\leq\frac{T+\sqrt{T}}{2\pi}$, $\frac{T+\sqrt{T}}{2\pi}<n\leq\frac{T}{\pi}$ or $n>\frac{T}{\pi}$. For the first interval we use \eqref{Xn:2}, in the second interval we use \eqref{Xn:1} with the simplification $\log\left(\frac{2\pi n}{T}\right)\geq\frac{2\pi n-T}{2\pi n}$ from Lemma~\ref{logg}\eqref{logg1}, while in the third interval we use \eqref{Xn:1} as well with $\log\left(\frac{2\pi n}{T}\right)\geq\log(2)$. Thus,
\begin{align}
K\leq & \ \xi_{1}(T,\tau)\sum_{n>\frac{T}{2\pi}}\frac{d_{1-2\tau}(n)}{n^{2-2\tau+\lambda}}+\xi_{2}(T,\tau)\!\!\!\!\sum_{\frac{T}{2\pi}<n\leq\frac{T+\sqrt{T}}{2\pi}}\frac{d_{1-2\tau}(n)}{n^{2-2\tau+\lambda}} \nonumber \\
& \ +\xi_{3}(T,\tau)\!\!\!\!\sum_{\frac{T+\sqrt{T}}{2\pi}<n\leq\frac{T}{\pi}}\frac{d_{1-2\tau}(n)}{n^{1-2\tau}(2\pi n-T)}+\xi_{4}(T,\tau)\sum_{n>\frac{T}{\pi}}\frac{d_{1-2\tau}(n)}{n^{2-2\tau+\lambda}}. \label{eq:Kalot}
\end{align}
with
\begin{align*}
\xi_{1}(T,\tau) & =2Q+2R_{2}T^{1-2\tau}, & \xi_{3}(T,\tau) & =\frac{2R_{1}(2\pi)^{1+\lambda}}{e^{\mathrm{c}}}T^{\frac{3}{2}-2\tau}, \\
\xi_{2}(T,\tau) & =2R'_{1}T^{2-2\tau}, & \xi_{4}(T,\tau) & =\frac{2R_{1}}{\log(2)}T^{\frac{3}{2}-2\tau}.
\end{align*}

We bound the first summation in \eqref{eq:Kalot} via Proposition~\ref{pr:sumdivwei_C}. For $\frac{1}{4}\leq\tau<\frac{1}{2}$, choosing $\lambda=\frac{\mathrm{c}}{\log(T)}$, we have
\begin{align*}
\sum_{n>\frac{T}{2\pi}}\frac{d_{1-2\tau}(n)}{n^{2-2\tau+\lambda}}\leq & \ \frac{\zeta(2-2\tau+\lambda)}{\lambda\left(\frac{T}{2\pi}\right)^{\lambda}}+\frac{\frac{1}{2-2\tau}+\frac{1}{(1-2\tau)(1-2\tau+\lambda)}}{\left(\frac{T}{2\pi}\right)^{\lambda}}+\frac{\zeta(2-2\tau+\lambda)+\frac{1}{1-2\tau+\lambda}+1}{\left(\frac{T}{2\pi}\right)^{1+\lambda}} \\
\leq & \ \frac{(2\pi)^{\lambda}\left(\frac{3}{2}+\lambda\right)}{\mathrm{c}^{2}e^{\mathrm{c}}}\log^{2}(T)+\frac{(2\pi)^{\lambda}}{(1-2\tau)\mathrm{c}e^{\mathrm{c}}}\log(T)+\frac{(2\pi)^{\lambda}}{e^{\mathrm{c}}} \\
 & \ +\frac{(2\pi)^{1+\lambda}\left(\frac{5}{2}+\lambda\right)\log(T)}{\mathrm{c}e^{\mathrm{c}}T}+\frac{(2\pi)^{1+\lambda}}{e^{\mathrm{c}}T}.
\end{align*}
where we used $T^{\lambda}=e^{\mathrm{c}}$ and Proposition~\ref{pr:mvzeta}. For $\tau=\frac{1}{2}$, we get instead
\begin{align}
\sum_{n>\frac{T}{2\pi}}\frac{d(n)}{n^{1+\lambda}}\leq & \ \frac{\log\left(\frac{T}{2\pi}\right)}{\lambda\left(\frac{T}{2\pi}\right)^{\lambda}}+\left(1+\frac{\zeta(1+\lambda)+\gamma}{\lambda}+\frac{\zeta(1+\lambda)+\frac{2}{3\lambda}}{\frac{T}{2\pi}}\right)\frac{1}{\left(\frac{T}{2\pi}\right)^{\lambda}} \nonumber \\
\leq & \ \frac{(2\pi)^{\lambda}(\mathrm{c}+1)}{\mathrm{c}^{2}e^{\mathrm{c}}}\log^{2}(T)+\frac{(2\pi)^{\lambda}(\gamma+1-\log(2\pi))}{\mathrm{c}e^{\mathrm{c}}}\log(T)+\frac{(2\pi)^{\lambda}}{e^{\mathrm{c}}} \nonumber \\
 & \ +\frac{5(2\pi)^{1+\lambda}\log(T)}{3\mathrm{c}e^{\mathrm{c}}T}+\frac{(2\pi)^{1+\lambda}}{e^{\mathrm{c}}T}, \label{eq:Kalot1}
\end{align}
where we can then forget the term of order $\log(T)$ since $\gamma+1-\log(2\pi)<0$. The last summation in \eqref{eq:Kalot} is bounded analogously, replacing every $2\pi$ with $\pi$; this time, $\gamma+1-\log(\pi)>0$ means that we cannot forget the term of order $\log(T)$ when $\tau=\frac{1}{2}$.

The second sum in \eqref{eq:Kalot} can be bounded through Proposition~\ref{pr:sumdiv} as follows: for $\frac{1}{4}\leq\tau<\frac{1}{2}$, we have
\begin{align*}
 & \ \sum_{\frac{T}{2\pi}<n\leq\frac{T+\sqrt{T}}{2\pi}}\frac{d_{1-2\tau}(n)}{n^{2-2\tau+\lambda}}\leq\frac{(2\pi)^{2-2\tau+\lambda}}{e^{\mathrm{c}}T^{2-2\tau}}\sum_{\frac{T}{2\pi}<n\leq\frac{T+\sqrt{T}}{2\pi}}d_{1-2\tau}(n) \\
\leq & \ \frac{(2\pi)^{2-2\tau+\lambda}}{e^{\mathrm{c}}T^{2-2\tau}}\left(\frac{\zeta(2\tau)\sqrt{T}}{2\pi}+\frac{\zeta(2-2\tau)}{2-2\tau}\left(\frac{T}{2\pi}\right)^{2-2\tau}\frac{2-2\tau}{\sqrt{T}}+A_{\tau}\left(\frac{T}{2\pi}\right)^{1-\tau}\left(2+\frac{1-\tau}{\sqrt{T}}\right)\right) \\
\leq & \ \frac{(2\pi)^{\lambda}}{e^{\mathrm{c}}\sqrt{T}}\left(1+\frac{1}{1-2\tau}\right)+\frac{2A_{\tau}(2\pi)^{\frac{3}{4}+\lambda}}{e^{\mathrm{c}}T^{1-\tau}}+\frac{3A_{\tau}(2\pi)^{\frac{3}{4}+\lambda}}{4e^{\mathrm{c}}T^{\frac{3}{2}-\tau}},
\end{align*}
where we dropped the negative first term in the proposition and simplified through Lemma~\ref{logg}\eqref{loggexp} and Proposition~\ref{pr:mvzeta}. Similarly, for $\tau=\frac{1}{2}$ we obtain
\begin{align*}
 & \ \sum_{\frac{T}{2\pi}<n\leq\frac{T+\sqrt{T}}{2\pi}}\frac{d(n)}{n^{1+\lambda}}\leq\frac{(2\pi)^{1+\lambda}}{e^{\mathrm{c}}T}\sum_{\frac{T}{2\pi}<n\leq\frac{T+\sqrt{T}}{2\pi}}d(n) \\
\leq & \ \frac{(2\pi)^{\lambda}\log(T)}{e^{\mathrm{c}}\sqrt{T}}+\frac{(2\gamma-\log(2\pi))(2\pi)^{\lambda}+2(2\pi)^{\frac{1}{2}+\lambda}A_{\frac{1}{2}}}{e^{\mathrm{c}}\sqrt{T}}+\frac{(2\pi)^{\lambda}+\frac{1}{2}(2\pi)^{\frac{1}{2}+\lambda}A_{\frac{1}{2}}}{e^{\mathrm{c}}T}.
\end{align*}

Finally, we use Lemma~\ref{le:sumKlog} directly on the third sum. Putting all terms together, we conclude the proof.
\end{proofbold}

\section{Numerical considerations}\label{sec:numerical}

\textbf{Bounding an integral in an area.} In Corollary~\ref{co:main}, we bound $\int_{0}^{\sage{T0}}|\zeta(\tau+it)|^{2}dt$ for all $\frac{1}{2}<\tau\leq\frac{3}{4}$. This is computed directly via Sage \cite{Sag19}. First, we retrieve bounds for $|\zeta(\tau+it)|^{2}$ in small square areas covering the rectangle $\mathrm{R}=\left[\frac{1}{2},\frac{3}{4}\right]+[0,\sage{T0}]i$. Such bounds define a piecewise constant real-valued function  $\mathrm{p}$ on $\mathrm{R}$, whose integral on any path contained in  $\mathrm{R}$ is an upper bound for the integral of $|\zeta(\tau+it)|^{2}$ on the same path. Then, we only need to check the finitely many possibilities that arise from the definition of  $\mathrm{p}$.

\footnotesize
\begin{mysage}
# Function to compute an upper bound for the constant i'.
# See also similar considerations for i in {co:main}.
# In the following, we assert that ipint_const is the output value
# for ip_upper_bound(1/2,1,100,1/400). For a verification,
# change the value of check_ip to True at the beginning of the paper.
# Warning: if one proceeds with the verification with these inputs,
# the computation of this value takes about 10 minutes!
def Fu_zetasq_mod(s):
    s = CBF(s)
    zetasq = ((s.zeta())^2).abs()
    zetasq_mod = zetasq*(1-s.real())
    return RIF(CBF(zetasq_mod).real())
def ip_upper_bound(simin,simax,ta,prec):
    if prec>=1/40:
        return 'Error'
    if check_ip==False:
        if simin>=1/2:
            if simax<=1:
                if ta<=100:
                    if prec>=1/400:
                        return 159.693638459782
    simin = RIF( simin )
    simax = RIF( simax )
    ta = RIF( ta )
    Sgen = RIF( 0 )
    si = simin+prec
    while si-prec<simax:
        S = RIF( 0 )
        point = CBF( si+prec*i )
        while RIF(point.imag()-prec)<ta:
            sphere = CBF(point).add_error(prec)
            if sphere.contains_exact(1):
                maxim = 5
                S = S+maxim
            else:
                maxim = Fu_zetasq_mod(sphere)
                S = S+maxim*2*prec
            point = point+2*prec*i
        if S.upper()>Sgen.upper():
            Sgen = S
        si = si+2*prec
    return Sgen.upper()
###
# Constant i' below.
ipint_const = roundup( ip_upper_bound(1/2,1,T0,1/400) , digits )
###
\end{mysage}
\normalsize

If we are interested in widening the range of $\tau$ (see below) from $\left(\frac{1}{2},\frac{3}{4}\right)$ to $\left(\frac{1}{2},1\right)$, we have to deal separately with the pole of $\zeta$ in $1$. Considering the Laurent expansion of $\zeta$ and the bounds on its coefficients given by Lavrik \cite[Lemma~4]{Lav76}, we can bound $|\zeta(s)|^{2}$ by $\frac{5}{(1-\tau)^{2}}$ for $s=\tau+it$ such that $|1-s|\leq\frac{1}{10}$. Thus, $\int_{0}^{1-\tau}|\zeta(\tau+it)|^{2}dt\leq\frac{5}{1-\tau}$ for $1-\tau\leq\frac{1}{20}$. Then, in the rest of the rectangle we bound the function $|\zeta(\tau+it)|^{2}(1-\tau)$ numerically as we did in $\mathbf{i}$ for $|\zeta(\tau+it)|^{2}$, so that $\int_{0}^{\sage{T0}}|\zeta(\tau+it)|^{2}dt\leq\mathbf{i}'=\frac{\sage{ipint_const}}{1-\tau}$.
\medskip

\textbf{Widening the range of $\tau$.} As mentioned in \S\ref{Int}, the strategy in proving Theorem~\ref{th:main} may be extended to the whole critical strip. We chose not to do so because the error term of order $T^{\frac{3}{2}-2\tau}\log^{2}(T)$ in the range $0<\tau<\frac{1}{4}$ becomes larger than the second main term of order $T$, and restricting the range of $\tau$ allows us to approximate constants more tightly, yielding a better quantitative result.

\footnotesize
\begin{mysage}
###
# Putting together the estimates for the range 0<tau<1/2.
# We have already computed vectors of coefficients for I,J,K
# in this range, and we need only sum them together with
# the term pi/2: as we did for tau>=1/4, we want I+J1+J2+K+pi/2.
# Note that the constants in I,K feature [1,1/tau,1/(1/2-tau)]
# while the constant in J features [1,1/tau,1/(1/2-tau)^2]:
# thus, we can absorb all third entries into 1/(1/2-tau)^2,
# multiplying by the appropriate max(1/2-tau).
modifier = 1/2-0
I0_c = RIF( I0__32m2tl_final_vec[0]/log_int(T0) )
I0_0 = RIF( I0__32m2tl_final_vec[1]/log_int(T0) )
I0_2q = RIF( I0__32m2tl_final_vec[2]*modifier/log_int(T0) )
J0_c = RIF( J0__32m2tl_final_vec[0]/log_int(T0) \
            + J0__1mtlq_final_vec[0] )
J0_0 = RIF( J0__32m2tl_final_vec[1]/log_int(T0)
            + J0__1mtlq_final_vec[1] )
J0_2q = RIF( J0__32m2tl_final_vec[2]/log_int(T0)
             + J0__1mtlq_final_vec[2] )
K0_c = RIF( K0__32m2tlq_final_vec[0] )
K0_0 = RIF( K0__32m2tlq_final_vec[1] )
K0_2q = RIF( K0__32m2tlq_final_vec[2]*modifier )
other0_c = RIF( pi/2/(sqrt_int(T0)*log_int(T0)^2) )
other0_0 = RIF( 0 )
other0_2q = RIF( 0 )
Main0_2q = RIF( I0_2q+2*J0_2q+K0_2q+other0_2q )
Main0_0 = RIF( I0_0+2*J0_0+K0_0+other0_0 )
Main0_c = RIF( I0_c+2*J0_c+K0_c+other0_c )
Main0_2q = roundup( Main0_2q , digits )
Main0_0 = roundup( Main0_0 , digits )
Main0_c = roundup( Main0_c , digits )
###
# Coefficients per order of the process analogous to {co:main}.
# We need to use i' instead of i,
# and |z|<=(1+1/4*1/(1/2-tau)^2+1/4*1/(1-tau)^2)*T0
# instead of |z|<=(1+1/8*1/(1/2-tau)^2)*T0.
# Also, note that we want 1/(1-tau)^2 here,
# instead of the 1/tau in the range (0,1/2).
###
# Coefficients per order of the whole result.
def Fu_Cor0__12lq_2q(ta):
    ta = RIF( ta )
    result = exp_int(log_int(2*pi_int)*(1-2*1/4))*Main0_2q*4*(1-1/4)
    return RIF(result)
def Fu_Cor0__12lq_1q(ta):
    ta = RIF( ta )
    modifier = 1-1/2
    result = exp_int(log_int(2*pi_int)*(1-2*1/4))*Main0_0*modifier*4*(1-1/4)
    return RIF(result)
def Fu_Cor0__12lq_c(ta):
    ta = RIF( ta )
    result = exp_int(log_int(2*pi_int)*(1-2*1/4))*Main0_c*4*(1-1/4)
    return RIF(result)
def Fu_Cor0__lq_2q(ta):
    ta = RIF( ta )
    result = exp_int(log_int(2*pi_int)*(1-2*1/4))*Main0_2q \
             *2/3*(3-2*1/4)*Fu_Z(1/4,1/2,ta)*exp_int(log_int(ta)*(-3/2))
    return RIF(result)
def Fu_Cor0__lq_1q(ta):
    ta = RIF( ta )
    modifier = 1-1/2
    result = exp_int(log_int(2*pi_int)*(1-2*1/4))*Main0_0*modifier \
             *2/3*(3-2*1/4)*Fu_Z(1/4,1/2,ta)*exp_int(log_int(ta)*(-3/2))
    return RIF(result)
def Fu_Cor0__lq_c(ta):
    ta = RIF( ta )
    result = exp_int(log_int(2*pi_int)*(1-2*1/4))*Main0_c \
             *2/3*(3-2*1/4)*Fu_Z(1/4,1/2,ta)*exp_int(log_int(ta)*(-3/2))
    return RIF(result)
def Fu_Cor0__0_2q(ta):
    ta = RIF( ta )
    result = 1/4*ta*(1+Fu_Z(1/4,1/2,ta)/ta^2)
    return RIF(result)
def Fu_Cor0__0_1q(ta):
    ta = RIF( ta )
    modifier = 1-1/2
    result = ip_upper_bound(1/2,3/4,ta,1/400)*modifier+1/4*ta*(1+Fu_Z(1/4,1/2,ta)/ta^2)
    return RIF(result)
def Fu_Cor0__0_c(ta):
    ta = RIF( ta )
    result = ta*(1+Fu_Z(1/4,1/2,ta)/ta^2)
    return RIF(result)
def Fu_Cor0__m32lq_2q(ta):
    ta = RIF( ta )
    result = exp_int(log_int(2*pi_int)*(1-2*1/4))*Main0_2q*2*Fu_Z(1/4,1/2,ta)
    return RIF(result)
def Fu_Cor0__m32lq_1q(ta):
    ta = RIF( ta )
    modifier = 1-1/2
    result = exp_int(log_int(2*pi_int)*(1-2*1/4))*Main0_0*modifier*2*Fu_Z(1/4,1/2,ta)
    return RIF(result)
def Fu_Cor0__m32lq_c(ta):
    ta = RIF( ta )
    result = exp_int(log_int(2*pi_int)*(1-2*1/4))*Main0_c*2*Fu_Z(1/4,1/2,ta)
    return RIF(result)
###
# Rounding coefficients to the order T^(1/2)*log(T)^2.
def Fu_Cor0__final_2q(ta):
    ta = RIF( ta )
    x_12lq = Fu_Cor0__12lq_2q(ta)
    x_lq = Fu_Cor0__lq_2q(ta)
    x_0 = Fu_Cor0__0_2q(ta)
    x_m32lq = Fu_Cor0__m32lq_2q(ta)
    x = x_12lq + x_lq/sqrt_int(ta) + x_0/(sqrt_int(ta)*log_int(ta)^2) + x_m32lq/ta^2
    return RIF(x)
def Fu_Cor0__final_1q(ta):
    ta = RIF( ta )
    x_12lq = Fu_Cor0__12lq_1q(ta)
    x_lq = Fu_Cor0__lq_1q(ta)
    x_0 = Fu_Cor0__0_1q(ta)
    x_m32lq = Fu_Cor0__m32lq_1q(ta)
    x = x_12lq + x_lq/sqrt_int(ta) + x_0/(sqrt_int(ta)*log_int(ta)^2) + x_m32lq/ta^2
    return RIF(x)
def Fu_Cor0__final_c(ta):
    ta = RIF( ta )
    x_12lq = Fu_Cor0__12lq_c(ta)
    x_lq = Fu_Cor0__lq_c(ta)
    x_0 = Fu_Cor0__0_c(ta)
    x_m32lq = Fu_Cor0__m32lq_c(ta)
    x = x_12lq + x_lq/sqrt_int(ta) + x_0/(sqrt_int(ta)*log_int(ta)^2) + x_m32lq/ta^2
    return RIF(x)
###
# Actual values for our choice of T0.
Cor0__final_2q = roundup( Fu_Cor0__final_2q(T0) , digits )
Cor0__final_1q = roundup( Fu_Cor0__final_1q(T0) , digits )
Cor0__final_c = roundup( Fu_Cor0__final_c(T0) , digits )
\end{mysage}
\normalsize

For the interest of the reader, however, we report here a version of the main result valid for the whole strip. If $T\geq T_{0}=\sage{T0}$ and $0<\tau<\frac{1}{2}$, then
\begin{align*}
\int_{0}^{T}\left|\zeta\left(\tau+it\right)\right|^{2}dt= & \ \frac{\zeta(2-2\tau)}{(2-2\tau)(2\pi)^{1-2\tau}}T^{2-2\tau}+\zeta(2\tau)T \\
& \ +O^{*}\left(\left(\frac{\sage{Main0_2q}}{\left(\frac{1}{2}-\tau\right)^{2}}+\frac{\sage{Main0_0}}{\tau}+\sage{Main0_c}\right)T^{\frac{3}{2}-2\tau}\log^{2}(T)\right).
\end{align*}
A process like the one in the proof of Corollary~\ref{co:main} holds in this range too. Thus, if $T\geq T_{0}=\sage{T0}$ and $\frac{1}{2}<\tau<1$, then
\begin{align*}
\int_{0}^{T}\left|\zeta\left(\tau+it\right)\right|^{2}dt= & \ \zeta(2\tau)T+\frac{(2\pi)^{2\tau-1}\zeta(2-2\tau)}{2-2\tau}T^{2-2\tau} \\
 & \ +O^{*}\left(\left(\frac{\sage{Cor0__final_2q}}{\left(\tau-\frac{1}{2}\right)^{2}}+\frac{\sage{Cor0__final_1q}}{(1-\tau)^{2}}+\sage{Cor0__final_c}\right)\sqrt{T}\log^{2}(T)\right).
\end{align*}
provided that, in the corresponding proof, we use the value $\mathbf{i}'$ above instead of $\mathbf{i}$. For the aforementioned reasons, the error terms in the bounds in \cite{DHA19} are asymptotically worse than the ones of our main result for $\frac{1}{4}<\tau<\frac{3}{4}$, but better than the ones presented in this section for $\tau\leq\frac{1}{4}$ and $\tau\geq\frac{3}{4}$.
\medskip

\textbf{Increasing $T_{0}$.} The choice of $T_{0}=\sage{T0}$ was made for the sake of convenience. Indeed, we needed to choose $T\geq 50$, because we relied upon Theorem~\ref{convexitygeneral} to bound $|\zeta(s)|$ on horizontal lines in \S\ref{sec:J}, and we have asked for various largeness conditions to simplify many computations. For instance, during the proof of Proposition~\ref{pr:I} we required the negative term from Proposition~\ref{pr:sumdivwei} to be smaller in absolute value than the positive one for $X=\frac{T}{2\pi}$; furthermore, for the purpose of properly rounding constants, we asked for some functions such as $t\mapsto\frac{\log(t)}{\sqrt{t}}$ to be decreasing in the interval $[T_{0},\infty)$ so that we are able to absorb the terms that are asymptotically of smaller order, via inequalities like $\frac{\log(T)}{\sqrt{T}}\leq\frac{\log(T_{0})}{\sqrt{T_{0}}}$ for $T\geq T_{0}$.

One can repeat the same calculations with a higher $T_{0}$ and expect to improve on the error terms in Theorems~\ref{th:mainpure} and \ref{th:main}. Say that the following are the error terms in the various cases.
\begin{center}
\begin{tabular}{|l|l|l|}
\hline
$\tau$ & Theorem~\ref{th:mainpure} & Theorem~\ref{th:main} \\ \hline
$\frac{1}{2}$ & $\mathfrak{e}_{1}\sqrt{T}\log^{2}(T)$ & $\mathfrak{m}_{11}\sqrt{T}\log^{2}(T)+\mathfrak{m}_{12}\sqrt{T}\log(T)$ \\
$\left[\frac{1}{4},\frac{1}{2}\right)$ & $\frac{\mathfrak{e}_{2}}{(1/2-\tau)^{2}}T^{\frac{3}{2}-2\tau}\log^{2}(T)$ & $\left(\frac{\mathfrak{m}_{21}}{(1/2-\tau)^{2}}+\mathfrak{m}_{22}\right)T^{\frac{3}{2}-2\tau}\log^{2}(T)$ \\ \hline
\end{tabular}
\end{center}
Then, see the table below for different values of $T_{0}$.

\footnotesize
\begin{mysage}
###
# Functions for computing the main constants of {th:main}:
# our integral of interest is in particular I+J1+J2+K+pi/2.
###
# Case tau=1/2.
def mainco2_b1(ci,ta):
    ci = RIF( ci )
    ta = RIF( ta )
    result = 2*Fu_J2__12lq_final(ci,ta) + Fu_K2__12lq_final(ci,ta)
    return RIF(result)
def mainco2_b2(ci,ta):
    ci = RIF( ci )
    ta = RIF( ta )
    result = Fu_I2__12l_final(ta) + 2*Fu_J2__12l_final(ci,ta) \
             + Fu_K2__12l_final(ci,ta) + pi_int/2/(sqrt_int(ta)*log_int(ta))
    return RIF(result)
###
# Case tau in [1/4,1/2).
def mainco4_b1(ci,ta):
    ci = RIF( ci )
    ta = RIF( ta )
    modifier = 1/2-1/4
    I4 = Fu_I4__32m2tl_final(ta)[2]
    J4 = Fu_Jt__32m2tl_final(1/4,1/2,ci,ta)[2]
    K4 = Fu_K4__32m2tlq_final(ci,ta)[2]
    result = I4*modifier/log_int(ta) + 2*J4/log_int(ta) + K4*modifier
    return RIF(result)
def mainco4_b2(ci,ta):
    ci = RIF( ci )
    ta = RIF( ta )
    modifier = 1/(1/4)
    I4__32m2tl_both = Fu_I4__32m2tl_final(ta)
    I4 = I4__32m2tl_both[0] + I4__32m2tl_both[1]*modifier
    J4__32m2tl_both = Fu_Jt__32m2tl_final(1/4,1/2,ci,ta)
    J4 = J4__32m2tl_both[0] + J4__32m2tl_both[1]*modifier
    J4p = Fu_Jt__1mtlq_final(1/4,1/2,ci,ta)
    if J4p[1]!=0 or J4p[2]!=0:
        return 'Error'
    K4__32m2tl_both = Fu_K4__32m2tlq_final(ci,ta)
    K4 = K4__32m2tl_both[0] + K4__32m2tl_both[1]*modifier
    other = pi_int/2/(sqrt_int(ta)*log_int(ta)^2)
    result = I4/log_int(ta) + 2*J4/log_int(ta) + 2*J4p[0] + K4 + other
    return RIF(result)
###
# Functions for computing the main constants of {th:mainpure}:
# we absorb the constants of the previous functions in one.
###
# For 1/2, we absorb everything to the order
# sqrt(T)*log(T)^2.
def mainco2(ci,ta):
    ci = RIF( ci )
    ta = RIF( ta )
    result = mainco2_b1(ci,ta) + mainco2_b2(ci,ta)/log_int(ta)
    return RIF(result)
###
# For [1/4,1/2), we absorb everything to the order
# T^(3/2-2tau)*log(T)^2/(1/2-tau)^2.
def mainco4(ci,ta):
    ci = RIF( ci )
    ta = RIF( ta )
    modifier = 1/2-1/4
    result = mainco4_b1(ci,ta) + mainco4_b2(ci,ta)*modifier^2
    return RIF(result)
\end{mysage}
\normalsize

\begin{center}
\begin{tabular}{|l|l|l|l|l|l|l|l|}
\hline
$T_{0}$ & \!\!\!\!\! & $\mathfrak{e}_{1}(T_{0})$ & $\mathfrak{m}_{11}(T_{0})$ & $\mathfrak{m}_{12}(T_{0})$ & $\mathfrak{e}_{2}(T_{0})$ & $\mathfrak{m}_{21}(T_{0})$ & $\mathfrak{m}_{22}(T_{0})$ \\ \hline
$10^{3}$ & \!\!\!\!\! & $\sage{roundup(mainco2(c_la,10^3),digits)}$ & $\sage{roundup(mainco2_b1(c_la,10^3),digits)}$ & $\sage{roundup(mainco2_b2(c_la,10^3),digits)}$ & $\sage{roundup(mainco4(c_la,10^3),digits)}$ & $\sage{roundup(mainco4_b1(c_la,10^3),digits)}$ & $\sage{roundup(mainco4_b2(c_la,10^3),digits)}$  \\ \hline
$10^{4}$ & \!\!\!\!\! & $\sage{roundup(mainco2(c_la,10^4),digits)}$ & $\sage{roundup(mainco2_b1(c_la,10^4),digits)}$ & $\sage{roundup(mainco2_b2(c_la,10^4),digits)}$ & $\sage{roundup(mainco4(c_la,10^4),digits)}$ & $\sage{roundup(mainco4_b1(c_la,10^4),digits)}$ & $\sage{roundup(mainco4_b2(c_la,10^4),digits)}$ \\ \hline
$10^{6}$ & \!\!\!\!\! & $\sage{roundup(mainco2(c_la,10^6),digits)}$ & $\sage{roundup(mainco2_b1(c_la,10^6),digits)}$ & $\sage{roundup(mainco2_b2(c_la,10^6),digits)}$ & $\sage{roundup(mainco4(c_la,10^6),digits)}$ & $\sage{roundup(mainco4_b1(c_la,10^6),digits)}$ & $\sage{roundup(mainco4_b2(c_la,10^6),digits)}$ \\ \hline
$10^{10}$ & \!\!\!\!\! & $\sage{roundup(mainco2(c_la,10^10),digits)}$ & $\sage{roundup(mainco2_b1(c_la,10^10),digits)}$ & $\sage{roundup(mainco2_b2(c_la,10^10),digits)}$ & $\sage{roundup(mainco4(c_la,10^10),digits)}$ & $\sage{roundup(mainco4_b1(c_la,10^10),digits)}$ & $\sage{roundup(mainco4_b2(c_la,10^10),digits)}$ \\ \hline
$10^{15}$ & \!\!\!\!\! & $\sage{roundup(mainco2(c_la,10^15),digits)}$ & $\sage{roundup(mainco2_b1(c_la,10^15),digits)}$ & $\sage{roundup(mainco2_b2(c_la,10^15),digits)}$ & $\sage{roundup(mainco4(c_la,10^15),digits)}$ & $\sage{roundup(mainco4_b1(c_la,10^15),digits)}$ & $\sage{roundup(mainco4_b2(c_la,10^15),digits)}$ \\ \hline
$10^{20}$ & \!\!\!\!\! & $\sage{roundup(mainco2(c_la,10^20),digits)}$ & $\sage{roundup(mainco2_b1(c_la,10^20),digits)}$ & $\sage{roundup(mainco2_b2(c_la,10^20),digits)}$ & $\sage{roundup(mainco4(c_la,10^20),digits)}$ & $\sage{roundup(mainco4_b1(c_la,10^20),digits)}$ & $\sage{roundup(mainco4_b2(c_la,10^20),digits)}$ \\ \hline
$10^{30}$ & \!\!\!\!\! & $\sage{roundup(mainco2(c_la,10^30),digits)}$ & $\sage{roundup(mainco2_b1(c_la,10^30),digits)}$ & $\sage{roundup(mainco2_b2(c_la,10^30),digits)}$ & $\sage{roundup(mainco4(c_la,10^30),digits)}$ & $\sage{roundup(mainco4_b1(c_la,10^30),digits)}$ & $\sage{roundup(mainco4_b2(c_la,10^30),digits)}$ \\ \hline
$10^{40}$ & \!\!\!\!\! & $\sage{roundup(mainco2(c_la,10^40),digits)}$ & $\sage{roundup(mainco2_b1(c_la,10^40),digits)}$ & $\sage{roundup(mainco2_b2(c_la,10^40),digits)}$ & $\sage{roundup(mainco4(c_la,10^40),digits)}$ & $\sage{roundup(mainco4_b1(c_la,10^40),digits)}$ & $\sage{roundup(mainco4_b2(c_la,10^40),digits)}$ \\ \hline
\end{tabular}
\end{center}

For intervals of integration with extremum lower than $\sage{T0}$, one can estimate it directly using rigorous numerical integration implemented in the ARB package \cite{Joh18}. Computing the integral up to $T=1000$, for example, takes a couple of seconds using the function ``CBF.integral''.
\medskip

\textbf{Choice of $\lambda$.} Another significant choice that we have made concerns the parameter $\lambda$ appearing in \S\ref{sec:J} and \S\ref{sec:K}. First of all, the order of $\lambda$ as a function of $T$ has been chosen to give the optimal order of error in the main theorem for the case $\tau=\frac{1}{2}$. We could not have chosen $\lambda=o\left(\frac{1}{\log(T)}\right)$, or else, in \S\ref{sec:K}, \eqref{eq:Kalot1} would have been too large. Nor could we have chosen $\frac{1}{\lambda}=o(\log(T))$, or else, in \S\ref{sec:J}, \eqref{estimation:L2} would have been too large. It is noteworthy that an error of order $\sqrt{T}\log^{2}(T)$ emerges also as consequence of \eqref{estimation:L1}, regardless of the choice of $\lambda$, as this comes from the use of the convexity bounds described in Corollary \ref{convexity}.

\footnotesize
\begin{mysage}
###
# Function for finding an optimal c, given T0: for us, what we want
# to optimize is the first error term in {th:main}, of order sqrt(T)*log(T)^2.
# Since 0<lambda<1/2, we search by brute force among all 0<c<log(T0)/2,
# and we find the best among them (with as many digits as we use elsewhere).
def best_c(ta):
    ta = RIF( ta )
    pace = 10^(-digits)
    limit = rounddown(log_int(ta)/2,digits)
    best_c = 0
    best_c_yields = infinity
    c = pace
    while c<=limit:
        c_yields=mainco2_b1(c,ta)
        if c_yields<best_c_yields:
            best_c=c
            best_c_yields=c_yields
        c=c+pace
    return best_c
###
# In the following, we assert that these are the values that
# the function best_c retrieves for certain choices of T0:
# T0=100 -> c=1.501
# T0=10^3 -> c=1.622
# T0=10^4 -> c=1.688
# T0=10^6 -> c=1.758
# T0=10^10 -> c=1.819
# For verifications about the truth of these assertions,
# change the value of check_c to True at the beginning
# of the paper and then see the list of checks at the end.
\end{mysage}
\normalsize

Upon fixing $\lambda$ as function of $T$, it remains to choose the optimal value of $\mathrm{c}$ according to the expression $\lambda=\frac{\mathrm{c}}{\log(T)}$. For simplicity, since the optimal $\mathrm{c}$ may vary with $\tau$, we chose to optimize only with respect to the case $\tau=\frac{1}{2}$. Hence, we have selected $\mathrm{c}=\sage{c_la}$ after numerical experimentation through a computer search. The only constraint we are facing is that $0<\lambda<\frac{1}{2}$, so our goal is to minimize the coefficient $\mathfrak{m}_{11}(T_{0})$ of the error term of order $\sqrt{T}\log^{2}(T)$ inside the more precise Theorem~\ref{th:main} (when $\tau=\frac{1}{2}$), given the fact that we are bounding $\lambda$ by its maximum value $\frac{\mathrm{c}}{\log(T_{0})}$ whenever necessary. For this matter, we have considered the range $\mathrm{c}\in(0,\sage{rounddown(log_int(T0)/2,digits)})$ and checked for an optimal $\mathrm{c}\in\frac{1}{\sage{10^digits}}\mathbb{N}$. Moreover, it is clear that changing $T_{0}$ and $\tau$ may change the best $\mathrm{c}$ to select. Here follows a table featuring an analogous optimization of $\mathrm{c}$ and its effect on the error terms of Theorems~\ref{th:mainpure} and \ref{th:main} in the case $\tau=\frac{1}{2}$.

\begin{center}
\begin{tabular}{|l|l|l|l|l|l|}
\hline
$T_{0}$ & $\mathrm{c}$ & \!\!\!\!\! & $\mathfrak{e}_{1}(T_{0})$ & $\mathfrak{m}_{11}(T_{0})$ & $\mathfrak{m}_{12}(T_{0})$ \\ \hline
$10^{3}$ & $1.622$ & \!\!\!\!\! & $\sage{roundup(mainco2(1.622,10^3),digits)}$ & $\sage{roundup(mainco2_b1(1.622,10^3),digits)}$ & $\sage{roundup(mainco2_b2(1.622,10^3),digits)}$ \\ \hline
$10^{4}$ & $1.688$ & \!\!\!\!\! & $\sage{roundup(mainco2(1.688,10^4),digits)}$ & $\sage{roundup(mainco2_b1(1.688,10^4),digits)}$ & $\sage{roundup(mainco2_b2(1.688,10^4),digits)}$ \\ \hline
$10^{6}$ & $1.758$ & \!\!\!\!\! & $\sage{roundup(mainco2(1.758,10^6),digits)}$ & $\sage{roundup(mainco2_b1(1.758,10^6),digits)}$ & $\sage{roundup(mainco2_b2(1.758,10^6),digits)}$ \\ \hline
$10^{10}$ & $1.819$ & \!\!\!\!\! & $\sage{roundup(mainco2(1.819,10^10),digits)}$ & $\sage{roundup(mainco2_b1(1.819,10^10),digits)}$ & $\sage{roundup(mainco2_b2(1.819,10^10),digits)}$ \\ \hline
\end{tabular}
\end{center}

The fact that the values $\mathfrak{e}_{1}(T_{0})$ are worse than in the previous table is due to our choice of optimizing only $\mathfrak{m}_{11}(T_{0})$ in Theorem~\ref{th:main}, thus not taking into consideration the contribution of the smaller terms.

\section*{Acknowledgements}

The authors would like to thank Harald Helfgott for his valuable suggestions.
       
Daniele Dona was supported by the European Research Council under Programme H2020-EU.1.1., ERC Grant ID: 648329 (codename GRANT) during his permanence at Georg-August-Universit\"at G\"ottingen. He was also supported by the Israel Science Foundation Grants No. 686/17 and 700/21 of A. Shalev, and the Emily Erskine Endowment Fund during his permanence at the Hebrew University of Jerusalem; he has been a postdoc at the Hebrew University of Jerusalem under A. Shalev in 2020/21 and 2021/22.

Sebastian Zuniga Alterman was supported by the postdoctoral grant  of the DAMSI University Centre of Excellence during his permanence at the Nicolaus Copernicus University at Toru\'n.

\Addresses

\footnotesize
\begin{mysage}
###
# Safety checks about the computations in the paper are performed here.
###
# Here we check that T0>=e^2, to ensure that we absorb errors correctly
# for log(T)/sqrt(T).
String_T0=''
if T0<=exp_int(2):
    String_T0=String_T0+'Error in T0 too small. '
# Then we check that T0/2pi>1, for {pr:sumdivwei}-{pr:sumdivwei_B} to work.
if T0/(2*pi_int)<=1:
    String_T0=String_T0+'Error in T0 too small. '
# Then we check that T0/(4pi)<(T0-sqrt(T0))/(2pi), for {le:sumdivlog} to be neater.
if T0/(4*pi_int)>=(T0-sqrt_int(T0))/(2*pi_int):
    String_T0=String_T0+'Error in T0 too small. '
# Then we check that the negative term in the first sum
# in the error term of {eq:alot} is smaller than the positive one,
# so as to be able to eliminate it and simplify the reasoning.
if 2*log_int(T0/(2*pi_int))<=4*(1-gamma_int):
    String_T0=String_T0+'Error in T0 too small. '
# Then we check that the term in sqrt(T) in {eq:Klogerr} is negative,
# so as to be able to eliminate it and simplify the reasoning.
if -2*sqrt_int(2)+log_int(1+sqrt_int(1+1/sqrt_int(T0)))+log_int(7/3) \
   -log_int((sqrt_int(2)+1)/(sqrt_int(2)-1))>=0:
    String_T0=String_T0+'Error in T0 too small. '
# Then we check that T0>=50, for Backlund's result to hold for T.
if T0<50:
    String_T0=String_T0+'Error in T0 too small. '
###
# Here we check that the main theorem has the correct constant.
String_Main=''
Main_true = RIF( Main_a + Main_b/log_int(T0) )
Main_true = roundup( Main_true , digits )
modifier = 1/2-1/4
Main14_true = RIF( Main4_2q+Main4_c*modifier^2 )
Main14_true = roundup( Main14_true , digits )
Main34_true = RIF( Cor4__final_2q+Cor4__final_c*modifier^2 )
Main34_true = roundup( Main34_true , digits )
if Main!=Main_true:
    String_Main=String_Main+'Error in main result for 1/2. It should be '+str(Main_true)+'. '
if Main14!=Main14_true:
    String_Main=String_Main+'Error in main result before 1/2. It should be '+str(Main14_true)+'. '
if Main34!=Main34_true:
    String_Main=String_Main+'Error in main result after 1/2. It should be '+str(Main34_true)+'. '
###
# Here we check that the values in the statement of Prop. 3.1 match our computations.
String_I=''
if I_a!=I2__12l_final:
    String_I=String_I+'Error in main I. It should be '+str(I2__12l_final)+' for 1/2. '
if I4_2!=I4__32m2tl_final_2:
    String_I=String_I+'Error in main I. It should be '+str(I4__32m2tl_final_2) \
             +' for [1/4,1/2) and the 1/2-tau term. '
if I4_c!=I4__32m2tl_final_c:
    String_I=String_I+'Error in main I. It should be '+str(I4__32m2tl_final_c) \
             +' for [1/4,1/2) and the constant term. '
###
# Here we check that Prop. 3.2 for tau in [1/4,1/2) has no term T^(1-tau)*log(T)^2*1/(1/2-tau)^2.
String_Jorder=''
if J4__1mtlq_final_2q!=0:
    String_Jorder=String_Jorder+'Error in J. Wrong order for tau in [1/4,1/2). '
###
# Here we check that the values in the statement of Prop. 3.2 match our computations.
String_J=''
if J_a!=J2__12lq_final:
    String_J=String_J+'Error in main J. 1st should be '+str(J2__12lq_final)+'. '
if J_b!=J2__12l_final:
    String_J=String_J+'Error in main J. 2nd should be '+str(J2__12l_final)+'. '
if J4_2q!=J4__32m2tl_final_2q:
    String_J=String_J+'Error in main J. 1/(1/2-tau)^2 should be '+str(J4__32m2tl_final_2q)+'. '
if J4_c!=J4__32m2tl_final_c:
    String_J=String_J+'Error in main J. Constant in tau case should be '+str(J4__32m2tl_final_c)+'. '
if J4_cp!=J4__1mtlq_final_c:
    String_J=String_J+'Error in main J. Constant in 2nd tau case should be '+str(J4__1mtlq_final_c)+'. '
###
# Here we check that lambda<1/2, for the second integral in {Xn:est}
# to be bounded correctly.
String_lambda=''
if c_la/log_int(T0)>=1/2:
    String_lambda=String_lambda+'Error in lambda too large. '
###
# Here we check that the values in the statement of Prop. 3.3 match our computations.
String_K=''
if K_a!=K2__12lq_final:
    String_K=String_K+'Error in main K. 1st should be '+str(K2__12lq_final)+'. '
if K_b!=K2__12l_final:
    String_K=String_K+'Error in main K. 2nd should be '+str(K2__12l_final)+'. '
if K4_2!=K4__32m2tlq_final_2:
    String_K=String_K+'Error in main K. 1/(1/2-tau) should be '+str(K4__32m2tlq_final_2)+'. '
if K4_c!=K4__32m2tlq_final_c:
    String_K=String_K+'Error in main K. Constant in tau case should be '+str(K4__32m2tlq_final_c)+'. '
###
# Here we check the values of the optimal c in the last section are correct.
# This check is performed only if check_c is set to True
# at the beginning of the paper.
String_bestc=''
if check_c:
    if best_c(T0)!=1.501:
        String_bestc=String_bestc+'Error in c. Best(T0) is '+str(best_c(T0))+'. '
    if best_c(10^3)!=1.622:
        String_bestc=String_bestc+'Error in c. Best(10^3) is '+str(best_c(10^3))+'. '
    if best_c(10^4)!=1.688:
        String_bestc=String_bestc+'Error in c. Best(10^4) is '+str(best_c(10^4))+'. '
    if best_c(10^6)!=1.758:
        String_bestc=String_bestc+'Error in c. Best(10^6) is '+str(best_c(10^6))+'. '
    if best_c(10^10)!=1.819:
        String_bestc=String_bestc+'Error in c. Best(10^10) is '+str(best_c(10^10))+'. '
\end{mysage}
\normalsize

$\sage{String_T0}\sage{String_Main}\sage{String_I}\sage{String_Jorder}\sage{String_J}\sage{String_lambda}\sage{String_K}\sage{String_bestc}$


\begin{thebibliography}{111}
\bibitem{AS72}
 M. Abramowitz, I. A. Stegun, \textit{Handbook of Mathematical Functions, Tenth Printing}, National Bureau of Standards Applied Mathematics Series, Washington D.C., 1972.
\bibitem{Atk39}
 F. V. Atkinson, \textit{The mean value of the zeta-function on the critical line}, Q. J. Math. 10(1) (1939), 122--128.
\bibitem{Atk49}
 F. V. Atkinson, \textit{The mean value of the Riemann zeta function}, Acta Math. 81 (1949), 353--376.
\bibitem{Bac18}
 R. J. Backlund, \textit{\"Uber die Nullstellen der Riemannschen Zetafunktion}, Acta Math. 41 (1918), 345--375 (in German).
\bibitem{Bal78}
 R. Balasubramanian, \textit{An improvement on a theorem of Titchmarsh on the mean square of $|\zeta(\frac{1}{2}+it)|$}, Proc. Lond. Math. Soc. (3) 36 (1978), 540--576.
\bibitem{Bou17}
 J. Bourgain, \textit{Decoupling, exponential sums and the Riemann zeta function}, J. Amer. Math. Soc. 30(1) (2017), 205--224.
\bibitem{DHA19}
 D. Dona, H. Helfgott, S. Zuniga Alterman, \textit{Explicit $L^2$ bounds for the Riemann $\zeta$ function.} Journal de Th\'eorie des Nombres de Bordeaux. Accepted.
\bibitem{For02}
 K. Ford, \textit{Vinogradov's integral and bounds for the Riemannn zeta function}, Proc. Lond. Math. Soc. (3) 85 (2002), 565--633.
\bibitem{Goo77}
 A. Good, \textit{Ein $\Omega$-Resultat f\"ur das quadratische Mittel der Riemannschen Zetafunktion auf der kritische Linie}, Invent. Math. 41 (1977), 233--251 (in German).
\bibitem{Hia16}
 G. A. Hiary, \textit{An explicit van der Corput estimate for $\zeta(1/2+it)$}, Indag. Math. (N.S.) 27(2) (2016), 524--533.
\bibitem{Ing28}
 A. E. Ingham, \textit{Mean-value theorems in the theory of the Riemann
 zeta-function}, Proc. Lond. Math. Soc. (2) 27 (1928), 273--300.
\bibitem{Ivi85}
 A. Ivi\'c, \textit{The Riemann zeta-function}, John Wiley \& Sons, New York (1985).
\bibitem{Joh13}
 F. Johansson, \textit{Arb: a C library for ball arithmetic}, ACM Communications in Computer Algebra 47(4) (2013), 166--169.
\bibitem{Joh18}
 F. Johansson, \textit{Numerical integration in arbitrary-precision ball arithmetic}, in: International Congress on Mathematical Software, Springer, 2018, 255--263.
\bibitem{Lan09}
 E. Landau, \textit{Handbuch der Lehre von der Verteilung der Primzahlen}, Teubner, Leipzig, 1909 (in German).
\bibitem{Lav76}
 A. F. Lavrik, \textit{On the principal term in the divisor problem and the power series of the Riemann zeta-function in a neighborhood of its pole}, Tr. Mat. Inst. Steklova 142 (1976), 165--173 (in Russian), translated in Proc. Steklov Inst. Math. 1979(3) (1979), 175--183.
\bibitem{Leh70}
 R. S. Lehman, \textit{On the distribution of zeros of the Riemann zeta-function}, Proc. Lond. Math. Soc. (3) 20 (1970), 303--320.
\bibitem{Lit22}
 J. E. Littlewood, \textit{Researches in the theory of the Riemann $\zeta$-function}, Proc. Lond. Math. Soc. (2) 20 (1922), xxii--xxviii.
\bibitem{Mat89}
 K. Matsumoto, \textit{The mean square of the Riemann zeta-function in the critical strip}, Jpn. J. Math. 15 (1989), 1--13.
\bibitem{Mat00}
 K. Matsumoto, \textit{Recent Developments in the Mean Square Theory of the Riemann Zeta and Other Zeta-Functions}, in: Number Theory, Birkhauser, 2000, 241--286.
\bibitem{MM93}
 K. Matsumoto, T. Meurman, \textit{The mean square of the Riemann zeta-function in the critical strip III}, Acta Arith. 64(4) (1993), 357--382.
\bibitem{MV07}
 H. L. Montgomery, R. C. Vaughan, \textit{Multiplicative number theory: I. Classical theory}, Cambridge University Press, Cambridge, 2007.
\bibitem{Ram13}
 O. Ramar\'e, \textit{Some elementary explicit bounds for two mollifications of the Moebius function}, Funct. Approx. Comment. Math. Volume 49, Number 2 (2013), 229--240.
\bibitem{Rem98}
 R. Remmert, \textit{Classical Topics in Complex Function Theory}, Springer-Verlag, New York, 1998.
\bibitem{Sag19}
 Sage Developers, \textit{SageMath, the Sage Mathematics Software System
 (Version 8.9)}, \texttt{http://www.sagemath.org} (2019).
\bibitem{Sim20}
 A. Simoni\v{c}, \textit{Explicit zero density estimate for the {R}iemann zeta-function near the critical line}, J. Math. Anal. Appl. 491(1) (2020).
\bibitem{Sim22}
 A. Simoni\v{c}, \textit{On explicit estimates for $S(t)$, $S_{1}(t)$, and $\zeta(1/2+it)$ under the Riemann hypothesis}, J. Number Theory 231 (2022), 464--491.
\bibitem{SS21}
 A. Simoni\v{c}, V. V. Starichkova, \textit{Atkinson's formula for the mean square of $\zeta(s)$ with an explicit error term}, \texttt{arXiv:2105.06821} (2021).
\bibitem{Tit34}
 E. C. Titchmarsh, \textit{On van der Corput's method and the zeta-function of Riemann (V)}, Q. J. Math. Ser. 5(1) (1934), 195--210.
\bibitem{Tit86}
 E. C. Titchmarsh, \textit{The Theory of the Riemann Zeta-function. 2nd Edition}, Oxford University Press, New York, 1986.
\bibitem{Vor04}
 G. Vorono\"i, \textit{Sur une fonction transcendante et ses applications \`a la sommation de quelques s\'eries}, Ann. Sci. \'Ec. Norm. Sup\'er. (3) 21 (1904), 207--267, 459--533 (in French).
\end{thebibliography}
\end{document}